\documentclass[a4paper,11pt,english]{amsart}

\usepackage{multirow}
\usepackage[english]{babel}
\usepackage[utf8]{inputenc}
\usepackage[T1]{fontenc}

\usepackage{graphicx}
\usepackage{hyperref}
\usepackage{amssymb}
\usepackage{amsmath}
\usepackage{verbatim,enumerate}
\usepackage{tikz}

\newtheorem{thm}{Theorem}[section]
\newtheorem{lemma}[thm]{Lemma}
\newtheorem{prop}[thm]{Proposition}
\newtheorem{cor}[thm]{Corollary}

\newtheorem{defi}[thm]{Definition}
{\theoremstyle{definition}
\newtheorem{exa}[thm]{Example}
\newtheorem{rem}[thm]{Remark}}

\newcommand{\T}{\mathcal{T}}

\newcommand{\CP}{\mathbb{C}P}

\renewcommand{\epsilon}{\varepsilon}
\newcommand{\R}{\mathbb{R}}
\newcommand{\Z}{\mathbb{Z}}

\newcommand{\D}{\mathcal{D}}

\renewcommand{\o}{\preccurlyeq}

\newcommand{\Card}{\operatorname{Card}}

\newcommand{\codeg}{\operatorname{codeg}}
\newcommand{\Int}{\operatorname{Int}}

\newcommand{\dive}{\operatorname{div}}

\newcommand{\coef}[2][j]{\left\langle #2 \right\rangle_{#1}}

\usepackage{subcaption}
\begin{document}

\title{Polynomiality properties of  tropical refined invariants}

\author{Erwan Brugall\'e}
\address{Erwan Brugall\'e, Universit\'e de Nantes, Laboratoire de
  Math\'ematiques Jean Leray, 2 rue de la Houssini\`ere, F-44322 Nantes Cedex 3,
France}
\email{erwan.brugalle@math.cnrs.fr}

\author{Andr\'es {Jaramillo Puentes}}
\address{Andr\'es {Jaramillo Puentes}, Universit\'e de Nantes, Laboratoire de
  Math\'ematiques Jean Leray, 2 rue de la Houssini\`ere, F-44322 Nantes Cedex 3,
France}
\email{andres.jaramillo-puentes@univ-nantes.fr}


\begin{abstract}
  Tropical refined invariants of toric surfaces constitute a
  fascinating interpolation between real and complex enumerative
  geometries via tropical geometry. They
  were originally introduced
 by Block and Göttsche,  and further extended by Göttsche and
 Schroeter in the case of rational curves. 

 In this paper, we study the polynomial behavior of coefficients 
of  these tropical refined
invariants. We prove that coefficients of small codegree
are polynomials in the Newton polygon of
the curves under enumeration, when one fixes the genus of the latter.
This provides  a  surprising reappearance, in a dual setting, of 
the so-called node polynomials and the Göttsche conjecture. 
Our methods, based on floor diagrams
introduced by  Mikhalkin and the first author,
are entirely combinatorial.
Although the combinatorial treatment needed here is different,
we follow the overall strategy designed by Fomin and Mikhalkin and further developed by Ardila and Block.
Hence our results may suggest
phenomena in complex enumerative geometry that have not been studied yet.

In the particular case of  rational curves, we extend
our polynomiality results by including the extra parameter $s$
recording the number of $\psi$ classes.
Contrary to the polynomiality with respect to $ \Delta$, the one with
respect to $s$ may be expected from considerations on Welschinger
invariants in real enumerative geometry.
This pleads in particular
in favor of a  geometric definition of Göttsche-Schroeter invariants.
\end{abstract}
\maketitle
\tableofcontents

\section{Introduction}\label{sec:main}

\subsection{Results}
Tropical refined invariants of toric surfaces have been introduced
in \cite{BlGo14} and further explored in several directions since then, see for example
 \cite{IteMik13,FiSt15,BlGo14bis,GoKi16, Mik15,NPS16,
Shu18,BleShu17,Bou17,GotSch16,Bou19,Blo19,Bru18, Blo20a,Blo20b}.
In this paper, we study the polynomial behavior of the coefficients 
of  these tropical refined
invariants, in connection with node polynomials and the Göttsche conjecture
on  
one hand, and with Welschinger invariants on the other hand.
Our methods are entirely combinatorial and do not require any
specific knowledge in complex or real enumerative
geometry, nor in tropical, algebraic or symplectic geometry. Nevertheless
our work probably only gains meaning in the light of these frameworks,
so we briefly indicate below how tropical refined invariants arose
from  enumerative geometry considerations, and 
present some further connections in Section \ref{sec:got}.
We also provide in Section \ref{sec:comp} a few explicit computations
in genus 0
that are interesting to interpret in the light of Section \ref{sec:got}.

\medskip
Given a convex integer polygon $\Delta\subset\R^2$, i.e. the convex hull of
finitely many points in $\Z^2$, 
Block and Göttsche proposed in \cite{BlGo14}
to enumerate  irreducible tropical
curves with Newton polygon $\Delta$ and genus $g$ as proposed in \cite{Mik1},
but replacing Mikhalkin's complex multiplicity with its
quantum analog. 
Itenberg and Mikhalkin proved in  \cite{IteMik13} that the resulting 
symmetric Laurent polynomial in the variable $q$
 does not depend on the configuration of points
chosen to define it. This Laurent polynomial
is called
a \emph{tropical refined invariant} and is denoted by
$G_{\Delta}(g)$.
As a main feature,
tropical refined invariants interpolate between
 Gromov-Witten invariants (for $q=1$) and tropical Welschinger
invariants (for $q=-1$) of the toric surface $X_\Delta$
defined by
the polygon $\Delta$.
They are also conjectured to agree with the
$\chi_y$-refinement of Severi degrees of $X_\Delta$
introduced in 
\cite{GotShe12}.

Göttsche and Schroeter extended
the work of \cite{BlGo14} in the case when  $g=0$.
They  defined in \cite{GotSch16} some tropical
refined descendant invariants, denoted by $G_{\Delta}(0;s)$, 
depending now on an additional integer parameter $s\in\Z_{\ge 0}$. 
On the complex side, the value at $q=1$ of $G_{\Delta}(0;s)$ recovers some
 genus 0 relative Gromov-Witten invariants (or  some
 descendant invariants) of $X_\Delta$. On the real side and when $X_\Delta$ is an
 unnodal del Pezzo surface,
 plugging  $q=-1$  in $G_{\Delta}(0;s)$ recovers Welschinger invariants
counting real algebraic (or
symplectic) rational curves  passing through a generic
real configuration of $\Card(\partial \Delta\cap\Z^2)-1$
  points in $X_\Delta$ containing exactly 
 $s$ pairs of complex conjugated points. The case when $s=0$
 corresponds to tropical Welschinger invariants, and
$ G_{\Delta}(0;0)=G_{\Delta}(0)$
 for any polygon $\Delta$.

For the sake of brevity,
we do not recall the definition of tropical refined invariants in this
paper. Nevertheless we provide in Theorems \ref{thm:fd} and \ref{thm:psi fd} a
combinatorial recipe that computes them when $\Delta$ is an 
\emph{$h$-transverse} polygon, via the so-called \emph{floor diagrams}
introduced by Mikhalkin and the first author in \cite{Br7,Br6b}.
Since the present
work in entirely based on these floor diagram computations,
the reader unfamiliar with the invariants
$G_\Delta(g)$ and $G_\Delta(0;s)$ may take Theorems  \ref{thm:fd}
and \ref{thm:psi fd} as  definitions rather than statements.

\medskip
Denoting by $\iota_\Delta$ the number of integer points contained in
the interior of $\Delta$, 
the invariant $G_\Delta(g)$ is non-zero if and only if
$g\in\{0,1,\cdots,\iota_\Delta\}$.
It is known furthermore,
see for example \cite[Proposition 2.11]{IteMik13}, that
in this case $G_\Delta(g)$
has degree\footnote{As for
polynomials, the degree of a Laurent polynomial $\sum_{j=-m}^na_jq^j$
with $a_n\ne 0$
is defined to be $n$.} $\iota_\Delta-g$.
In this paper we establish that
 coefficients of
small \emph{codegree} of  $G_{\Delta_{a,b,n}}(g)$ and $G_{\Delta_{a,b,n}}(0;s)$ are
asymptotically polynomials in $a,b,n$, and $s$, where
$\Delta_{a,b,n}$ is the convex  polygon depicted in
Figure \ref{fig:hirz}.
\begin{figure}[h]
\begin{center}
\begin{tabular}{ccc}
  \includegraphics[height=2.5cm]{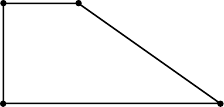}
  \put(-160,-10){$(0,0)$}
  \put(-160,75){$(0,a)$}
  \put(-105,75){$(b,a)$}
  \put(-15,-10){$(an+b,0)$}
&\hspace{12ex} &
  \includegraphics[height=2.5cm]{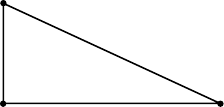}
  \put(-160,-10){$(0,0)$}
  \put(-160,75){$(0,a)$}
  \put(-15,-10){$(an,0)$}

  \\ \\ a) $\Delta_{a,b,n}$ && b) $\Delta_{a,0,n}$
\end{tabular}
\end{center}
\caption{}
\label{fig:hirz}
\end{figure}
By definition
the coefficient of codegree $i$ of a Laurent polynomial $P(q)$ of
degree $d$ is its coefficient of degree $d-i$, and is
denoted by  $\coef[i]{P}$.

\begin{thm}\label{thm:maing}
  For any $i,g\in\Z_{\ge 0}$, the function
\[
\begin{array}{ccc}
\Z_{\ge 0}^3&\longrightarrow & \Z_{\ge 0}
\\ (a,b,n)& \longmapsto & \coef[i]{G_{\Delta_{a,b,n}}(g)}
\end{array}
\]
is polynomial on the set  $\mathcal U_{i,g}$
defined by
\[
\left\{\begin{array}{l}
n\ge 1
\\ b> i
\\ b+n> (g+2)i+g
\\ a\ge i+2g+2
\end{array}
\right. 
\]
and has degree $i+g$ in each of the variables
$b$ and $n$, and degree $i+2g$ in the variable $a$.
\end{thm}

Theorem \ref{thm:maing} requires $n$ to be positive, and has the
    following version for $n=0$.
    
\begin{thm}\label{thm:maing n=0}
  For any $i,g\in\Z_{\ge 0}$, the function
\[
\begin{array}{ccc}
\Z_{\ge 0}^2&\longrightarrow & \Z_{\ge 0}
\\ (a,b)& \longmapsto & \coef[i]{G_{\Delta_{a,b,0}}(g)}
\end{array}
\]
is polynomial on the set defined by
\[
\left\{\begin{array}{l}
 b> (g+2)i+g
\\ a\ge i+2g+2
\end{array}
\right. 
\]
and has degree $i+g$ in each of the variables $a$ and $b$.
\end{thm}

In connection to the Göttsche conjecture (see Section \ref{sec:got}), one may also be interested in
fixing $b=0$ and $n\ge 1$, and varying $a$. Theorem \ref{thm:maing}
    can be adapted in this case.

\begin{thm}\label{thm:maing2}
  For any $i,g\in\Z_{\ge 0}$, and $n\in\Z_{>0}$, the function
\[
\begin{array}{ccc}
\Z_{\ge 0}&\longrightarrow & \Z_{\ge 0}
\\ a& \longmapsto & \coef[i]{G_{\Delta_{a,0,n}}(g)}
\end{array}
\]
is polynomial of degree $i+2g$ for $a\ge i+2g+2$.
\end{thm}

\begin{exa}\label{exa:g i=0}
Theorem \ref{thm:maing} may be seen as a partial generalisation of the
fact 
that for any convex integer polygon $\Delta$, one has
\[
\coef[0]{G_{\Delta}(g)}={{\iota_\Delta}\choose {g}}
\]
(see {\cite[Proposition
    2.11]{IteMik13} and \cite[Proposition 4.10]{BlGo14}}).  Indeed,
when $\Delta=\Delta_{a,b,n}$, this identity can be rewritten as
\[
\coef[0]{G_{\Delta_{a,b,n}}(g)}={{\frac{a^2n+2ab-(n+2)a -2b +2}{2}}\choose {g}},
\]
which is a polynomial of degree $g$ in the variables $b$ and $n$, and
of degree $g$ or $2g$ in the variable $a$ depending on whether $n=0$
or not.
\end{exa}

The particular case $g=0$ is much simpler to deal with, and the three
theorems above can be made more precise. Let us define
\[
\eta(\Delta)=\Card(\partial \Delta\cap\Z^2)-1.
\]
Since there is the additional
parameter $s$ in this case, one may also study polynomiality with
respect to $s$.
Note that the invariant  $G_\Delta(0;s)$ is non-zero if and only if
\[
s\in\left\{0,\cdots,\left[\frac{\eta(\Delta)}{2}\right]
\right\},
\]
in which case it has degree $\iota_\Delta$.

\begin{thm}\label{thm:main1}
For any $i\in\Z_{\ge 0}$, the function
\[
\begin{array}{ccc}
\Z^4&\longrightarrow & \Z_{\ge 0}
\\ (a,b,n,s)& \longmapsto & \coef[i]{G_{\Delta_{a,b,n}}(0;s)}
\end{array}
\]
is polynomial on the set $\mathcal U_{i}$ defined by 
\[
\left\{\begin{array}{l}
an+b\ge i+2s
\\ b>i
\\ a>i
\end{array}
\right. .
\]
Furthermore it has  degree $i$ in each of the variables
$a$, $b$,  $n$, and $s$.
\end{thm}

Theorem \ref{thm:main1} is an easy-to-state version of
Theorem \ref{thm:main1 expl g0} where we also provide an explicit
expression for $ \coef[i]{G_{\Delta_{a,b,n}}(0;s)}$.
As in the higher genus case,  Theorem \ref{thm:main1} can be adapted to the case
when $b=0$ and $n$ is fixed. 

\begin{thm}\label{thm:main1 expl g02}
For any $(i,n)\in\Z_{\ge 0}\times\Z_{>0}$, the function
\[
\begin{array}{ccc}
\Z_{\ge 0}^2&\longrightarrow & \Z_{\ge 0}
\\ (a,s)& \longmapsto & \coef[i]{G_{\Delta_{a,0,n}}(0;s)}
\end{array}
\]
 is polynomial on the set defined by 
 \[
 \left\{\begin{array}{l}
 an\ge i+2s
 \\a\ge i+2
 \end{array}
\right. 
\]
Furthermore it has  degree $i$ in each of the variables
$a$ and $s$.
\end{thm}

As mentioned above, floor diagrams allow the computation of the invariants
$G_{\Delta}(g)$ and $G_{\Delta}(0;s)$ when $\Delta$ is an
$h$-transverse polygons. The polygons
$\Delta_{a,b,n}$ are $h$-transverse, but the converse may not be true.
We do not see any difficulty other than technical to generalize all
the statements above to the case of $h$-transverse polygons, in the
spirit of \cite{ArdBlo,BlGo14bis}. Since this paper is already quite
long and technical, we have restricted ourselves to the case of polygons
$\Delta_{a,b,n}$.
From an algebro-geometric perspective, these polygons correspond to
the toric surfaces
$\CP^2$,  the $n$-th Hirzebruch surface $\mathbb F_n$, and the
weighted projective plane $\CP^2(1,1,n)$.

\medskip
It emerges from Section \ref{sec:got} that
 polynomiality with respect to $s$ deserves a separate  study 
 from  polynomiality with respect to $\Delta$.
Clearly, the values $\coef[i]{G_\Delta(0;0)},\cdots,\coef[i]{G_\Delta(0;s_{max})}$
are interpolated by a polynomial of degree at most $s_{max}$, where
\[
s_{max}=\left[ \frac{\eta(\Delta)}{2}\right].
\]
It is nevertheless reasonable to expect,
at least for  ``simple'' polygons, 
this interpolation polynomial to be of degree  $\min(i,s_{max})$.
The next Theorem
states that this is indeed the case for small values of $i$.
Given a convex integer
polygon $\Delta\subset \R^2$, we denote by
$d_b(\Delta)$ the  length of the bottom horizontal edge of
$\Delta$. Note that $d_b(\Delta)=0$ if this edge is reduced to a point.

\begin{thm}\label{thm:poly s}
  Let $\Delta$ be an $h$-transverse polygon in $\R^2$.
  If
$2i\le d_b(\Delta)+1 $ and $i\le\iota_\Delta$, then the values
$\coef[i]{G_\Delta(0;0)},\cdots,\coef[i]{G_\Delta(0;s_{max})}$
are interpolated by a polynomial of degree $i$, whose leading
coefficient is $\frac{(-2)^i}{i!}$.
If $\Delta=\Delta_{a,b,n}$, then the result holds also for
$2i= d_b(\Delta)+2 $.
\end{thm}
Observe that even when
$\Delta=\Delta_{a,b,n}$, Theorem \ref{thm:poly s} cannot be deduced from  
Theorems \ref{thm:main1} or from  
Theorems \ref{thm:main1 expl g02}.
Since
the proof of Theorem \ref{thm:poly s} does not seem  easier when
 restricting
to polygons $\Delta_{a,b,n}$ for $2i\le d_b(\Delta)+1$, we provide a proof valid for any
$h$-transverse polygon.
We expect that the upper bounds $2i\le d_b(\Delta)+1$ and
$2i\le d_b(\Delta)+2$ can be weakened;
nevertheless the proof via floor diagrams becomes more and more
intricate as $i$ grows, as is visible in our proof of Theorem \ref{thm:poly s}.

\subsection{Connection to complex and real enumerative geometry}\label{sec:got}
Let  $N_{\CP^2}^\delta(d)$ be the number of
irreducible algebraic curves of degree $d$, with $\delta$ nodes, and
passing through a generic configuration of $\frac{d(d+3)}{2}-\delta$
points in $\CP^2$. 
For a fixed $\delta\in\Z_{\ge 0}$, this number is polynomial in $d$ of
degree $2\delta$
for $d\ge \delta +2$. For example, one
has
\begin{align*}
 & \forall d\ge 1, \ N_{\CP^2}^0(d)=1
\\&  \forall d\ge 3,\ N_{\CP^2}^1(d)=3(d-1)^2
\\ & \forall d\ge 4,\ N_{\CP^2}^2(d)=\frac 3 2(d - 1)(d - 2)(3d^2 - 3d - 11)
\end{align*}
These \emph{node polynomials} have a long history. After some
computations for small values of $\delta$, they were
conjectured to exist for any $\delta$ by Di Francesco and Itzykson in \cite{DiFrIt}.
By around 2000,
they were computed up to $\delta=8$, see \cite{KlePie04} and references
therein for an historical account.
Göttsche proposed in \cite{Got98} a more general conjecture: given a
non-singular complex algebraic surface $X$,  a non-negative integer
$\delta$, and a line bundle $\mathcal L$
on $X$ that is sufficiently ample with respect to $\delta$, the number
$N_X^\delta(\mathcal L)$ of
irreducible algebraic curves in the linear system $|\mathcal L|$, with $\delta$ nodes, and
passing through a generic configuration of
$\frac{\mathcal L^2 + c_1(X)\cdot\mathcal L}{2}-\delta$
points in $X$ equals
$P_\delta(\mathcal L^2, c_1(X)\cdot\mathcal L,c_1(X)^2,c_2(X))$,
with $P_\delta(x,y,z,t)$ a  universal polynomial depending only on $\delta$. 

The Göttsche conjecture was proved in full generality by Tzeng in
\cite{Tze12}, and an alternative proof was
proposed shortly thereafter in \cite{KST11}. Both proofs use
algebro-geometric methods.
Fomin and Mikhalkin gave in \cite{FM} a combinatorial proof of the Di
Francesco-Itzykson conjecture
by mean of floor diagrams. This was  generalized by Ardila and Block 
in \cite{ArdBlo}
to a proof of the Göttsche conjecture restricted to the case when $X$ is
the toric surface associated to an $h$-transverse polygon.
Ardila and Block's work contains
an interesting outcome:  combinatorics
allows one to transcend from the original realm of the Göttsche
conjecture, and to consider  algebraic surfaces with mild
singularities as well. We are not aware of any algebro-geometric
approach to the Göttsche conjecture in the case of singular surfaces. 

Motivated by the paper \cite{KST11}, Göttsche and Shende defined in \cite{GotShe12} a
$\chi_y$-refined version of the numbers $N_X^\delta(\mathcal L)$. In the case
when $X$ is the toric surface  $X_\Delta$ associated to the polygon $\Delta$,
these refined invariants  are
conjecturally equal to the refined tropical invariants
$G_\Delta(\frac{\mathcal L^2 - c_1(X_\Delta)\cdot\mathcal L+2}{2}-\delta)$
that were simultaneously defined by Block and Göttsche in
\cite{BlGo14}.
In light of  the Göttsche conjecture,
it is reasonable to expect  the coefficients of
$G_\Delta(\frac{\mathcal L^2 - c_1(X_\Delta)\cdot\mathcal L+2}{2}-\delta)$ to be
asymptotically 
polynomial with respect to  $\Delta$.
Block and
Göttsche adapted in \cite{BlGo14}
 the methods from \cite{FM,ArdBlo} to show
that this is indeed the case.

\medskip
In all the story above, the parameter $\delta$ is fixed and the line bundle
$\mathcal L$ varies. In other words, we are enumerating algebraic
curves with a fixed number of nodes in a varying linear system.
In particular, the genus of the curves under enumeration in
the linear system $d\mathcal L$ grows quadratically with respect to $d$.
In a kind of dual setup, one may  fix the genus of curves under
enumeration. For example one may consider the numbers
$N^{\frac{(d-1)(d-2)}{2}-g}_{\CP^2}(d)$ in the case of $\CP^2$, and
let $d$ vary. However in this case it seems hopeless to seek for  any
polynomiality behavior. Indeed, the sequence 
$N^{\frac{(d-1)(d-2)}{2}-g}_{\CP^2}(d)$ tends to infinity  more than
exponentially fast. 
This has been proved by Di Francesco and Itzykson in
\cite{DiFrIt} when $g=0$, and the general case can be  obtained for example by an
easy adaptation of the proof of Di Francesco and Itzykson's result
via floor diagrams proposed in \cite{Br6b,Br10}.

Nevertheless, our results can be interpreted as a reappearance of the
Göttsche conjecture
at the refined level: coefficients of small codegrees of
$G_{\Delta_{a,b,n}}(g)$ behave polynomially asymptotically with respect to $(a,b,n)$.
It is somewhat reminiscent of the Itenberg-Kharlamov-Shustin
conjecture \cite[Conjecture 6]{IKS2}: although it has been shown to be wrong in
\cite{Wel4,Br8}, its refined version turned out to be true by
\cite[Corollary 4.5]{Bru18} and Corollary \ref{cor:decreasing} below.
Anyhow, it may be interesting to understand further this reappearance of 
the Göttsche conjecture.

In the same range of ideas, it may be worthwhile to investigate the existence
of universal polynomials
giving  asymptotic values of 
$\coef[i]{G_{\Delta_{a,b,n}}(g)}$. It follows from Examples
\ref{exa:coef g0} and \ref{exa:coef g0 bis} that the polynomials whose
existence is attested in Theorems \ref{thm:maing} and
\ref{thm:maing2} are not equal.
Nevertheless, 
we do not know whether there
exists  a universal polynomial $Q_{g,i}(x,y,z,t)$ such that,
under the assumption that the toric
surface $X_{\Delta_{a,b,n}}$ is non-singular, the
equality
\[
\coef[i]{G_{\Delta_{a,b,n}}(g)}=Q_{g,i}(\mathcal L_{a,b,n}^2,
c_1(X_{\Delta_{a,b,n}})\cdot\mathcal L_{a,b,n},c_1(X_{\Delta_{a,b,n}})^2,c_2(X_{\Delta_{a,b,n}}))
\]
holds in each of the three regions described in Theorems
\ref{thm:maing},
\ref{thm:maing n=0}, and
\ref{thm:maing2}. In the expression above $\mathcal L_{a,b,n}$ denotes
the line bundle on $X_{\Delta_{a,b,n}}$ defined by $\Delta_{a,b,n}$.
As explained in \cite[Section 1.3]{ArdBlo}, it is unclear what
should generalize the four intersection numbers in the formula above 
 when $X_{\Delta_{a,b,n}}$ is singular.
Recall that the surface $X_{\Delta_{a,b,n}}$ is non-singular precisely 
when $b\ne 0$ or $n= 1$, in which case one has
\[
\mathcal L_{a,b,n}^2=a^2n+2ab,
\qquad
c_1(X_{\Delta_{a,b,n}})\cdot\mathcal L_{a,b,n}=(n+2)a+2b,
\]
and
\[
c_1(X_{\Delta_{a,b,n}})^2=8 \mbox{ and } c_2(X_{\Delta_{a,b,n}})=4 \quad\mbox{if }b\ne 0,
\qquad \qquad c_1(X_{\Delta_{a,0,1}})^2=9 \mbox{ and } c_2(X_{\Delta_{a,0,1}})=3.
\]
It follows from the adjunction formula
combined with Pick's formula that
\[
\iota_{\Delta_{a,b,n}}= \frac{\mathcal
           L_{a,b,n}^2-c_1(X_{\Delta_{a,b,n}})\cdot\mathcal
           L_{a,b,n}+2}{2}.
\]
As a consequence, for $i=0$, the universal polynomials $Q_{g,0}$ exist and are given by
\[
Q_{g,0}(x,y,z,t)={{\frac{x-y+2}{2}}\choose{g}}.
\]
At the other extreme, Examples \ref{exa:coef g0} and
\ref{exa:coef g0 bis} suggest that $Q_{0,i}$ may not depend on $x$.

\bigskip
If this kind of ``dual'' Göttsche conjecture phenomenon may come as a
surprise, 
polynomiality with respect to $s$ of
$\coef[i]{G_{\Delta_{a,b,n}}(0;s)}$ is quite expected. It
is also related to complex and real enumerative geometry, and
pleads  in favor of a more geometric
definition of refined tropical invariants as conjectured, for example, in  \cite{GotShe12}.
Given a real projective algebraic surface $X$, we denote by 
$W_{X}(d;s)$ the Welschinger invariant  of $X$ 
counting (with signs) real  $J$-holomorphic rational curves  realizing
the class $d\in H_2(X;\Z)$, and
 passing through a generic
real configuration of $c_1(X)\cdot d-1$
  points in $X$ containing exactly 
  $s$ pairs of complex conjugated points (see \cite{Wel1,Bru18}).
  Welschinger exhibited in \cite[Theorem 3.2]{Wel1} a very simple
  relation between Welschinger invariants of a real algebraic surface $X$ and
  its blow-up $\widetilde X$ at a real point, with exceptional
  divisor $E$:
  \begin{equation}\label{eq:w rel}
  W_{X}(d;s+1)=W_{X}(d;s) - 2W_{\widetilde X}(d -2[E];s).
  \end{equation}
This equation is also obtained in \cite[Corollary 2.4]{Bru18} as a special
case of a formula relating Welschinger invariants of real surfaces
differing by a surgery along a real Lagrangian sphere. 
As suggested in \cite[Section 4]{Bru18},
it is reasonable to expect
that  such formulas admit a refinement.
The refined Abramovich-Bertram formula \cite[Corollary 5.1]{Bou19}, proving
\cite[Conjecture 4.6]{Bru18},
provides a piece of evidence for such
expectation.
Hence one may expect that a refinement of
  formula $(\ref{eq:w rel})$ holds both for tropical refined invariants from
  \cite{BlGo14,GotSch16}
  and for
  $\chi_y$-refined invariants from \cite{GotShe12}.
  
  As mentioned earlier, one has
  \[
G_{\Delta}(0;s)(-1)=  W_{X_\Delta}(\mathcal L_\Delta;s)
  \]
  when $X_\Delta$ is an
  unnodal del Pezzo surface. In particular
  \cite[Proposition 4.3]{Bru18} and 
 Proposition \ref{prop:wel} below state precisely that the refinement
of formula $(\ref{eq:w rel})$
holds true in the tropical set-up when both $X_\Delta$ and
$\widetilde{X_\Delta}$ are  unnodal toric del Pezzo surfaces.

In any event, reducing inductively to $s=0$,
one sees easily that $\coef[i]{G_{\Delta}(d,0;s)}$ is polynomial of
degree $i$ in $s$ if one takes for granted that
\begin{itemize}
  \item  tropical refined invariants $ G_{\Delta}(0;s)$ generalize
    to some $\chi_y$-refined tropical invariants
    $G_{X,\mathcal L}(0;s)$, where $X$ is an arbitrary projective
    surface and $\mathcal L\in Pic(X)$ is a line bundle;
\item $G_{X,\mathcal L}(0;s)$ is a symmetric Laurent series of degree
  $\frac{\mathcal L^2-c_1(X)\cdot \mathcal L+2}{2} $
  with leading coefficient equal to 1;
\item a refined version of formula $(\ref{eq:w rel})$ holds for
  refined invariants $G_{X,\mathcal L}(0;s)$.  
\end{itemize}
Since none of the last three conditions are established yet, Theorem
\ref{thm:poly s} may be seen as an evidence that these conditions
actually hold.

\bigskip
To end this section, note that  all the mentioned asymptotical problems
require one to fix either the number $\delta$ of nodes
of the curves under enumeration, or their genus $g$.
These two numbers are related by the adjunction formula
\[
g+\delta = \frac{\mathcal L^2-c_1(X)\cdot \mathcal L+2}{2}.
\]
One may wonder whether these  asymptotical results generalize when both
$g$ and $\delta$ are 
allowed  
to vary, as long as they satisfy the equation above.

\subsection{Some explicit computations in genus 0}\label{sec:comp}

Here we present a few  computations that illustrates Theorems \ref{thm:main1 expl g0},
\ref{thm:main1 expl g02}, and \ref{thm:poly s}, and which, in the light of Section
\ref{sec:got}, may point towards interesting directions.

\begin{exa}\label{exa:coef g0}
Theorem  \ref{thm:main1 expl g0} allows one to compute 
$\coef[i]{G_{\Delta_{a,b,n}}(0;s)}$ for small values of $i$.
For example
one computes easily that
 (recall that the sets $\mathcal U_i$ are defined in the statement of Theorem \ref{thm:main1})
 \[
 \forall (a,b,n)\in\mathcal U_{1},\  \coef[1]{G_{\Delta_{a,b,n}}(0;s)}=(n+2)a+2b+2-2s.
  \]
In relation to the Göttsche conjecture, one may try to express
$\coef[i]{G_{\Delta_{a,b,n}}(0;s)}$
in terms of topological numbers related to the linear
system $\mathcal L_{a,b,n}$ defined
by the polygon $\Delta_{a,b,n}$ in the Hirzebruch surface
$X_{\Delta_{a,b,n}}=\mathbb F_n$. Surprisingly, the values of
$\coef[i]{G_{\Delta_{a,b,n}}(0;s)}$ we compute can be expressed in
terms of $c_1(\mathbb F_n)\cdot \mathcal L_{a,b,n}=(n+2)a+2b$ and $s$
only.
Furthermore expressing these values in terms of the number of real points  rather
than in terms of the number $s$ of pairs of complex conjugated points
simplifies even further the final expressions.
More
precisely,  setting $y=(n+2)a+2b$ and $t=y-1-2s$, we obtain
\begin{align*}
&  \coef[0]{G_{\Delta_{a,b,n}}(0;s)}=1
 \\&\coef[1]{G_{\Delta_{a,b,n}}(0;s)}=t+3
 \\&\coef[2]{G_{\Delta_{a,b,n}}(0;s)}=\frac{t^2 + 6t + y + 19}{2}
 \\&\coef[3]{G_{\Delta_{a,b,n}}(0;s)}=\frac{t^3 + 9t^2 + (3y + 59)t + 9y + 147}{3!}
 \\&\coef[4]{G_{\Delta_{a,b,n}}(0;s)}=\frac{t^4 + 12t^3 + (6y + 122)t^2 + (36y + 612)t + 3y^2 + 120y + 1437}{4!}
 \\&\coef[5]{G_{\Delta_{a,b,n}}(0;s)}=\frac{1}{5!}\times \left(
 t^5 + 15t^4 + (10y + 210)t^3 + (90y + 1590)t^2 + (15y^2 + 620y +
 7589)t\right.
 \\ & \qquad\qquad\qquad\qquad\qquad\quad\left. + 45y^2 + 1560y + 16035\right)
 \\&\coef[6]{G_{\Delta_{a,b,n}}(0;s)}=\frac{1}{6!}\times \left(
 t^6 + 18t^5 + (15y + 325)t^4 + (180y + 3300)t^3 + (45y^2 + 1920y +
 24019)t^2\right.
 \\ & \qquad\qquad\qquad\qquad\qquad\quad\left.
 + (270y^2 + 9720y + 102522)t + 15y^3 + 945y^2 + 23385y + 207495\right)
\end{align*}
for any $(a,b,n,s)$ in the corresponding $\mathcal U_{i}$.
It appears from these computations that the polynomial
$\coef[i]{G_{\Delta_{a,b,n}}(0;s)}$
has total degree $i$ if $t$ has degree 1 and $y$ and degree 2. In
addition, 
its coefficients
seem to be all positive and  to also have some polynomial
behavior with respect to $i$:
\[
i!\times \coef[i]{G_{\Delta_{d}}(0;s)}=
t^i+ 3it^{i-1} +  \frac{i(i-1)}{6}\left( 3y+ 2i +53\right)t^{i-2}
+
\frac{i(i-1)(i-2)}{2}\left( 3y +2i + 43\right )t^{i-3}
+
\cdots 
\]
It could be interesting to study further these observations.
\end{exa}

\begin{exa}\label{exa:coef g0 bis}
Throughout 
the text, we  use the more common notation $\Delta_d$ rather than $\Delta_{d,0,1}$.
  It follows from Theorem \ref{thm:main1 expl g02} combined with
  Examples \ref{exa:fd cubic s} and \ref{exa:G quartic} that
  \[
  \forall d\ge 3,\ \coef[1]{G_{\Delta_{d}}(0;s)}=3d+1-2s.
  \]
  Further computations allow one to compute
  $\coef[i]{G_{\Delta_{d}}(0;s)}$ for the first values of $i$. Similarly to
  Example  \ref{exa:coef g0},
  it is interesting  to express
$\coef[i]{G_{\Delta_{d}}(0)}$
  in terms of $y=3d=c_1(\CP^2)\cdot d\mathcal L_1$ and $t=y-1-2s$:
  \begin{align*}
 &\forall d\ge 3,\  \coef[1]{G_{\Delta_{d}}(0;s)}=t+2
\\ &  \forall d\ge 4,\  \coef[2]{G_{\Delta_{d}}(0;s)}=   \frac{t^2 + 4t + y + 11}{2}
 \\&\forall d\ge 5,\  \coef[3]{G_{\Delta_{d}}(0;s)}=\frac{t^3 + 6t^2 + (3y + 35)t + 6y + 72}{3!}
 \\&\forall d\ge 6,\  \coef[4]{G_{\Delta_{d}}(0;s)}=\frac{t^4 + 8t^3 + (6y + 74)t^2 + (24y + 304)t + 3y^2 + 72y + 621}{4!}
 \\&\forall d\ge
 7,\  \coef[5]{G_{\Delta_{d}}(0;s)}=\frac{1}{5!}\times\left(
 t^5 + 10t^4 + (10y + 130)t^3 + (60y + 800)t^2 + (15y^2 + 380y +
 3349)t \right.
 \\ & \qquad\qquad\qquad\qquad\qquad\qquad\quad \left. + 30y^2 + 780y + 6030\right)
  \end{align*}
We observe the same phenomenon for
  the coefficients of the polynomial
  $\coef[i]{G_{\Delta_{d}}(0;s)}$ as in Example \ref{exa:coef g0}. In particular
  they
  seem  to have some polynomial
behavior with respect to $i$:
\[
i!\times \coef[i]{G_{\Delta_{d}}(0;s)}=
t^i+ 2it^{i-1} +   \frac{i(i-1)}{6}\left(3y+ 2i +29\right)t^{i-2}
+
\frac{i(i-1)(i-2)}{3}\left( 3y +2i + 30\right )t^{i-3}
+
\cdots 
\]
\end{exa}

\begin{exa}
For $n\ge 2$, one computes easily that 
  \[
  \coef[1]{G_{\Delta_{2,0,n}}(0;s)}=2n+2-2s=c_1(\mathbb F_n)\cdot \mathcal L_{2,0}-2s.
  \]
In particular, one notes a discrepancy with the case of $\CP^2$, i.e.
when $n=1$. This originates probably from the fact that the toric complex
algebraic surface $X_{\Delta_{a,0,n}}$ is singular
as soon as $n\ge 2$.
\end{exa}

\begin{exa}
Here we illustrate Theorem \ref{thm:poly s} in the case of $\Delta_4$.
  According to Example \ref{exa:G quartic} and setting $t=11-2s$, one has
\begin{align*}
G_{\Delta_4}(0;s)=& q^{-3} & +& (2+t) q^{-2} &+&
\frac{t^2+4t+23}{2} q^{-1} &+&
\frac{t^3 + 3t^2 + 59t + 81}{6}+
\\ & q^3  &+&  (2+t) q^{2}&+&\frac{t^2+4t+23}{2} q \ .
\end{align*}
In particular one has
\[
\coef[3]{G_{\Delta_{4}}(0;s)}=\frac{t^3 + 3t^2 + 59t + 81}{6}
\ne \frac{t^3 + 6t^2 + (3\times 12 +35)t + 6\times 12+ 72}{6}
\]
 meaning that the  threshold $d\ge i+2$ in Example \ref{exa:coef g0 bis} is sharp.
\end{exa}

\subsection{Method and outline of the paper}
Fomin and Mikhalkin were the first to use floor diagrams
to address the Göttsche conjecture in \cite{FM}. 
The basic strategy, further developed by Ardila and Block in 
\cite{ArdBlo}, is to 
 decompose floor diagrams into elementary building blocks, 
called \emph{templates}, that are suitable for a combinatorial treatment. 
Even though the basic idea 
to prove Theorems
\ref{thm:maing}, \ref{thm:maing n=0}, and \ref{thm:maing2}, is to
follow this
overall strategy, the
building blocks and their combinatorial treatment we need here differ from
those used in \cite{FM,ArdBlo}.

However in the special case when $g=0$, the situation simplifies
drastically, and there is no need of the templates machinery to prove
Theorems  \ref{thm:main1} and
\ref{thm:main1 expl g02}.
Indeed, one can
easily describe all floor diagrams coming into play, and perform a
combinatorial study by hand. In particular, we are able to
 provide an explicit
expression for $ \coef[i]{G_{\Delta_{a,b,n}}(0;s)}$ in Theorem \ref{thm:main1 expl g0}.

On the other hand, we  use a different
strategy than the one from \cite{FM,ArdBlo} to tackle polynomiality
with respect to  $s$ when $\Delta$ is fixed.
We prove Theorem \ref{thm:poly s}
by establishing that the sequence $(\coef[i]{G_\Delta(0;s)})_s$ is interpolated
by a polynomial whose $i$th discrete derivative (or $i$th difference)
is constant.

\medskip
The remaining part of this paper is organized as follows. We start by
recalling the definition of floor diagrams in Section \ref{sec:floor},
and how to use them to compute tropical refined invariants of $h$-transverse
polygons.
In particular, Theorems  \ref{thm:fd}
and \ref{thm:psi fd} may be considered as  definitions of these
invariants for readers unfamiliar with tropical geometry.
We collect some general facts about codegrees that will be used
throughout the text in Section \ref{sec:codeg}. In Section
\ref{sec:g=0},
we prove polynomiality
results for tropical refined invariants in genus 0. We first treat the
very explicit case when $\Delta=\Delta_{a,b,n}$ with $b\ne 0$, before turning
 to the slightly more technical situation when $b$ vanishes. We end
 this section by proving polynomiality with respect to $s$ alone with
 the help of
 discrete derivatives.
Lastly,  Section \ref{sec:gen} is devoted to higher genus and
becomes more technical. We define
a suitable notion of templates, and adapt the proofs from Section 
\ref{sec:g=0} in this more general situation.
Some well-known or easy identities on quantum numbers are recast
in Appendix \ref{app:quantum int} in order to ease the reading of the text.

\bigskip
\noindent{\bf Acknowledgment.}
We are grateful to Gurvan Mével, as well as two anonymous
referees, for pointing us out several typos and
inaccuracies  in
a first version of this paper. 
This  work is
partially supported by the grant TROPICOUNT of Région Pays de la
Loire, and the ANR project ENUMGEOM NR-18-CE40-0009-02.

\section{Floor diagrams}\label{sec:floor}


\subsection{$h$-transverse polygons}

The class of $h$-transverse polygons enlarges slightly the class of
polygons $\Delta_{a,b,n}$.
\begin{defi}\label{df:ht}
A convex integer polygon $\Delta$ is called \emph{$h$-transverse} if
every edge contained in its boundary $\partial \Delta$ is either horizontal,
vertical, or has slope 
$\frac{1}{k}$, with $k\in\Z$. 
\end{defi}
Given an $h$-transverse polygon $\Delta$, we use the following
notation:
\begin{itemize}
 \item 
$\partial_l\Delta$ and $\partial_r\Delta$
 denote  the sets of 
edges~$e\subset\partial \Delta$ with an external normal vector
having negative and positive first
coordinate, respectively;

\item  $d_l\Delta$ and $d_r\Delta$ denote the unordered lists of integers~$k$ 
appearing~$j\in\Z_{>0}$ times, such that the vector $(jk,-j)$ belongs to $\partial_l\Delta$ and $\partial_r\Delta$, respectively, with $j$ maximal; 

\item $d_b\Delta$ and $d_t\Delta$
denote the  lengths  of the horizontal edges at the bottom
and top, respectively, of $\Delta$.
\end{itemize}

Note that both sets $d_l\Delta$ and $d_r\Delta$ have
the integer height of $\Delta$ for cardinality.

\begin{exa}
As said above, all polygons $\Delta_{a,b,n}$ are
$h$-transverse. Recall that we use the  notation $\Delta_d$ instead of $\Delta_{d,0,1}$.
  We depict in Figure~\ref{fig:htrans}  two examples of $h$-transverse
polygons,  where the integer close to a
  segment in $\partial_l \Delta$ or  $\partial_r \Delta$  denotes its
  contribution to $d_l\Delta$ or $d_r \Delta$, respectively.
\end{exa}

\begin{figure}[h]
\begin{center}
\begin{tabular}{ccc}
  \includegraphics[height=3cm]{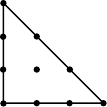}
   \put(-95,65){$0$}
 \put(-95,38){$0$}
  \put(-95,12){$0$}
  \put(-60,65){$1$}
  \put(-35,40){$1$}
  \put(-10,15){$1$}
  &
  \hspace{8ex} &
  \includegraphics[height=3cm]{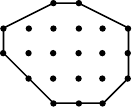}
     \put(-99,75){$-2$}
 \put(-113,50){$0$}
 \put(-103,25){$1$}
  \put(-82,6){$1$}
  \put(-21,75){$2$}
  \put(3,50){$0$}
  \put(3,30){$0$}
  \put(-10,5){$-1$}

 \\ \\a) $\begin{array}{l}d_l\Delta_3=\{0,0,0\},\\ d_r\Delta_3=\{1,1,1\},\\ d_b\Delta_3=3,\\ d_t\Delta_3=0.
  \end{array}$& &  b)
$\begin{array}{l}
   d_l\Delta=\{-2,0,1,1\},\\ d_r\Delta=\{2,0,0,-1\},\\ d_b\Delta=2,\\ d_t\Delta=1.\end{array}$
\end{tabular}
\end{center}
\caption{Examples of $h$-transverse polygons.}
\label{fig:htrans}
\end{figure}
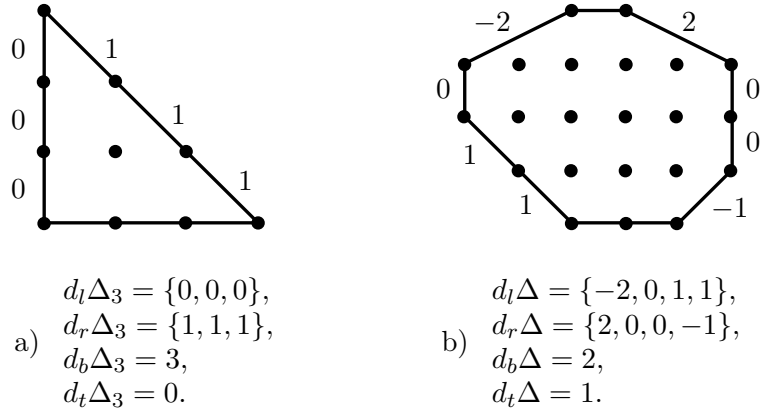

\subsection{Block-Göttsche refined invariants via floor diagrams}

In this text, an oriented multigraph $\Gamma$ consists of
a
set of vertices
$V(\Gamma)$, a collection
$E^0(\Gamma)$ of oriented
bivalent edges
in $V(\Gamma)\times V(\Gamma)$
and two collections of  monovalent edges:
a collection of sources
$E^{-\infty}(\Gamma)$, and a collection of sinks 
  $E^{+\infty}(\Gamma)$.
 A source adjacent to the vertex $v$ is oriented towards $v$, and
 a sink adjacent to the vertex $v$ is oriented away from $v$.
 Given such an oriented graph, we define the set of all edges of 
 $\Gamma$ by
 \[
 E(\Gamma)=E^0(\Gamma)\cup E^{-\infty}(\Gamma)\cup
 E^{+\infty}(\Gamma).
 \]
 We use the notation $\underset{\longrightarrow v}{e}$ and
 $\underset{v \longrightarrow}{e}$
 if the edge $e$ is oriented toward the vertex $v$ and away from $v$,
 respectively. 
 
 A \emph{weighted} oriented graph $(\Gamma,\omega)$ is an oriented
 graph endowed with a function $\omega:E(\Gamma)\to \Z_{>0}$.
 The \emph{divergence} of a vertex $v$ of a weighted oriented graph is
 defined as
 \[
\dive(v)=\sum_{\underset{\longrightarrow v}{e}}\omega(e)-\sum_{\underset{v \longrightarrow}{e}}\omega(e).
 \]

\begin{defi}\label{df:fd}
  A \emph{floor diagram} $\D$ with Newton polygon $\Delta$ is
  a quadruple
  $\D=(\Gamma, \omega,l,r)$ such that
  \begin{enumerate}
  \item $\Gamma$ is a connected weighted acyclic oriented graph
    with $\Card(d_l\Delta)$ vertices, with
    $d_b\Delta$ sources and $d_t\Delta$ sinks;
\item all sources and sinks have weight~$1$;
\item  $l:V(\Gamma)\longrightarrow d_l\Delta$ and $r: V(\Gamma)\longrightarrow d_r\Delta$
are bijections  such that for every vertex $v\in V(\Gamma)$,
one has
$ \dive(v)=r(v)-l(v)$.
\end{enumerate}

\end{defi}

By a slight abuse of notation, we
will not distinguish in this text between
a floor diagram $\D$ and its underlying graph $\Gamma$.
The first Betti number of $\D$ is called the
\emph{genus}  of the floor diagram $\D$. The vertices
of a floor diagram are called its \emph{floors}, and its edges
 are called \emph{elevators}. 
The \emph{degree} of a floor diagram~$\D$ is defined as
$$\deg(\D)=\sum_{e\in E(\D)}(\omega(e)-1).$$

Given an integer $k\in \Z$, the quantum integer $[k](q)$ is defined by
\[
\displaystyle
    [k](q)=\frac{q^{\frac{k}{2}}-q^{-\frac{k}{2}}}{q^{\frac{1}{2}}-q^{-\frac{1}{2}}}=
    q^{\frac{k-1}{2}}  +q^{\frac{k-3}{2}}+\cdots +q^{-\frac{k-3}{2}}+q^{-\frac{k-1}{2}}.
\]
For the reader's convenience, we collect some easy or well-known properties of quantum integers
in Appendix \ref{app:quantum int}.

\begin{defi}
The refined multiplicity of a floor diagram
$\D$ is the Laurent polynomial defined by 
\[
\mu(\D)(q)=\prod_{e\in E(\D)}[\omega(e)(q)]^2.
\]
\end{defi}
Note that $\mu(\D)(q)$ is in $\Z_{>0}[q^{\pm 1}]$, 
is symmetric, and has degree $\deg(\D)$.

\begin{exa}\label{exa:ex fd}
Examples of floor diagrams together with their refined multiplicities
are depicted in Figure \ref{fig:ex FD}. Conventionally, floors and
  elevators
  are represented
by ellipses and vertical lines, respectively. Orientation on elevators
is understood from bottom to top and will not be depicted; neither
will be the weight on elevators of weight~$1$.
All floor diagrams with Newton polygon $\Delta_3$
are depicted in Figures \ref{fig:ex FD}a), b), c), and 
d). Since both functions $l$ and $d$ are trivial in this case, we do
not
specify
them on the picture. An example of floor diagram with Newton polygon
depicted in Figure 
\ref{fig:htrans}b) is depicted in  Figure \ref{fig:ex FD}e). We specify
the value of $l$ and $r$ at each floor by an integer on the left and
on the right in the corresponding ellipse, respectively.
\end{exa}
\begin{figure}[h]
\begin{center}
\begin{tabular}{ccccc}
  \includegraphics[height=3cm, angle=0]{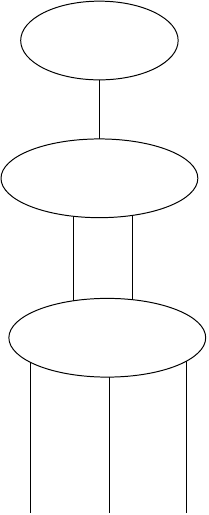}&
  \includegraphics[height=3cm, angle=0]{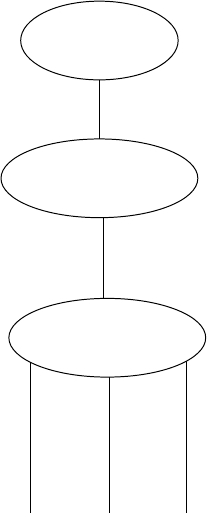}
  \put(-13,40){$2$}&
  \includegraphics[height=3cm, angle=0]{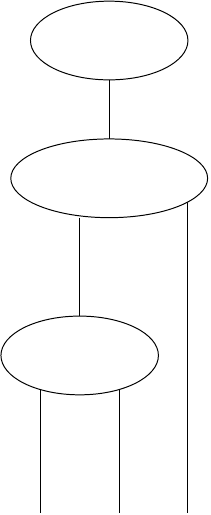}&
  \includegraphics[height=3cm, angle=0]{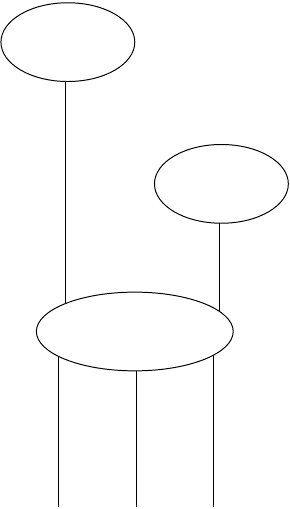}&
  \includegraphics[height=4cm, angle=0]{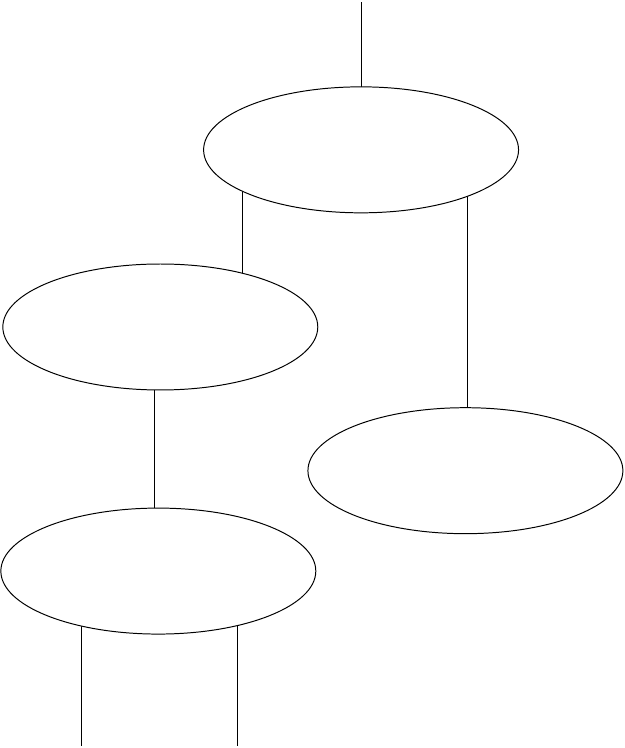}
   \put(-82,42){$3$}
  \put(-68,77){$2$}
  \put(-88,24){$1$}
  \put(-62,23){$0$}
  \put(-90,60){$-2$}
  \put(-65,60){$-1$}
  \put(-56,88){$0$}
  \put(-30,88){$2$}
  \put(-42,39){$1$}
  \put(-15,39){$0$}
 \\ \\a) $\mu=1$ &  b) $\mu=q + 2 + q^{-1}$ & c) $\mu=1$ & d) $\mu=1$ &
  e) $\mu=\ \  q^{3}\   +4q^{2}\ \ +8q\ \ +10 $
  \\&&&
  & $\quad \qquad \ +q^{-3}+4q^{-2}+8q^{-1}\qquad$
\end{tabular}
\end{center}
\caption{Example of floor diagrams with their refined multiplicities.}
\label{fig:ex FD}
\end{figure}
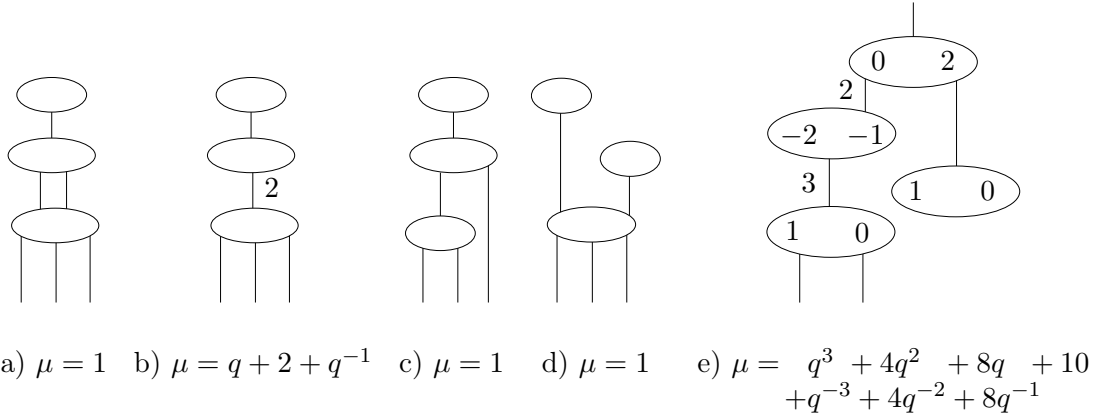

For a floor diagram $\D$
with Newton polygon $\Delta$ and genus $g$,
we define
\[
\eta(\D)=\eta(\Delta)+g.
\]
Note that, by a simple Euler characteristic computation, we also have
\[
\eta(\D)=\Card(V(\D))+\Card(E(\D)).
\]
The orientation of $\D$ induces a partial ordering on $\D$, that we
denote by $\preccurlyeq$. A map $m:A\subset\Z\to V(\D)\cup E(\D)$ is said to
be
\emph{increasing} if
 $i\le j$ whenever $m(i)\preccurlyeq m(j)$.
\begin{defi}
  A \emph{marking} of a floor diagram $\D$ with Newton polygon $\Delta$
  is an increasing bijection
  \[
  m\colon\{1,2,\dots,\eta(\D)\}\longrightarrow V(\D)\cup E(\D).
  \]

Two marked floor diagrams $(\D,m)$, $(\D',m')$
 with Newton polygon $\Delta$
are said to be
\emph{isomorphic} if there exists an isomorphism of weighted oriented
graphs $\varphi:\D\longrightarrow \D'$ such that $l=l'\circ\varphi$,
$r=r'\circ\varphi$, and $m'=\varphi\circ m$.
\end{defi}

The next theorem is a slight generalisation of \cite[Theorem 3.6]{Br6b}.

\begin{thm}[{\cite[Theorem 4.3]{BlGo14bis}}]\label{thm:fd}
Let $\Delta$ be an $h$-transverse polygon in $\R^2$, and $g\ge 0$ an
integer. Then one has
\[
G_\Delta(g)(q)=\sum_{(\D,m)} \mu(\D)(q),
\]
where the sum runs over all isomorphism classes of
marked floor diagrams with Newton polygon
$\Delta$ and genus $g$.
\end{thm}

\begin{exa}\label{exa:fd cubic}
  Using Figures \ref{fig:ex FD}a), b), c), and d)
 one obtains
  \[
  G_{\Delta_3}(1)(q)=1 \qquad\mbox{and}\qquad
  G_{\Delta_3}(0)(q)=q+10+q^{-1}.
  \]
\end{exa}

\begin{exa}\label{exa:fd quartic}
Combining Theorem \ref{thm:fd} with Figures \ref{degree 4 g=3,2},
\ref{degree 4 g=1}, and \ref{degree 4 g=0}, where all floor
diagrams with Newton polygon $\Delta_4$ are depicted, one obtains:
\[
G_{\Delta_4}(3)=1,  \qquad
G_{\Delta_4}(2) = 3q^{-1} + 21 + 3q, \qquad
G_{\Delta_4}(1) = 3q^{-2} + 33q^{-1} +153 + 33q + 3q^2,
\]
\[
G_{\Delta_4}(0) = q^{-3} +13q^{-2} + 94q^{-1} + 404 + 94q + 13q^2 + q^3.
\]
\begin{figure}[h]
\centering
\begin{tabular}{cccccc}
\includegraphics[height=3.5cm, angle=0]{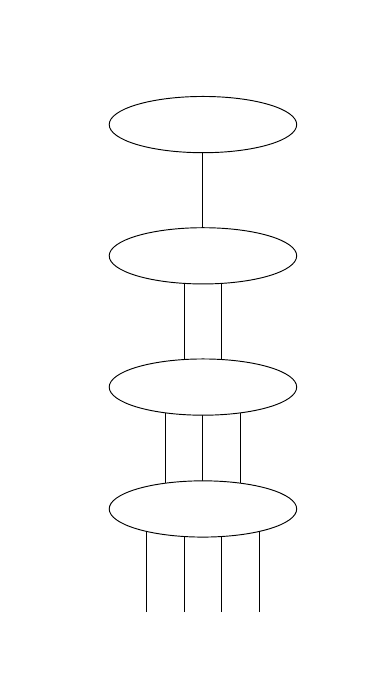}&
\includegraphics[height=3.5cm, angle=0]{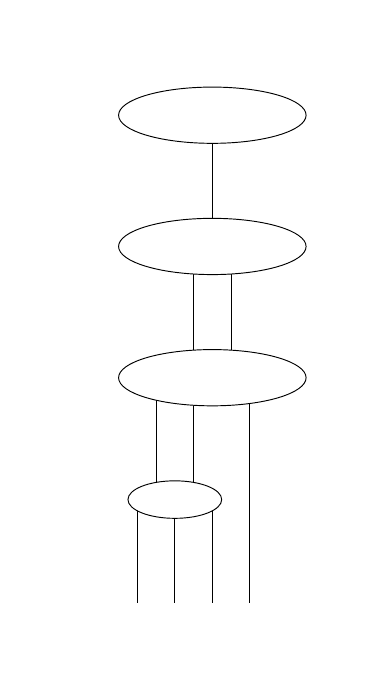}&
\includegraphics[height=3.5cm, angle=0]{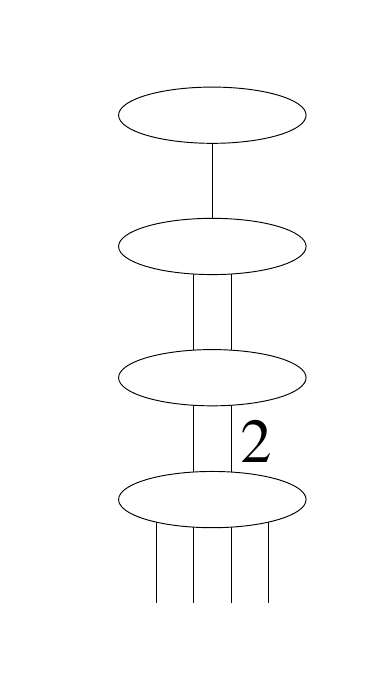}&
\includegraphics[height=3.5cm, angle=0]{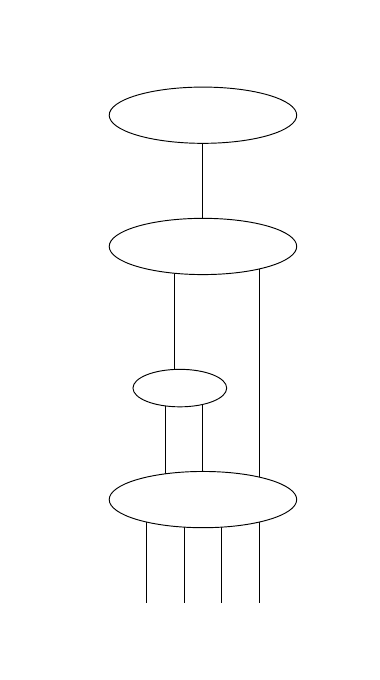}&
\includegraphics[height=3.5cm, angle=0]{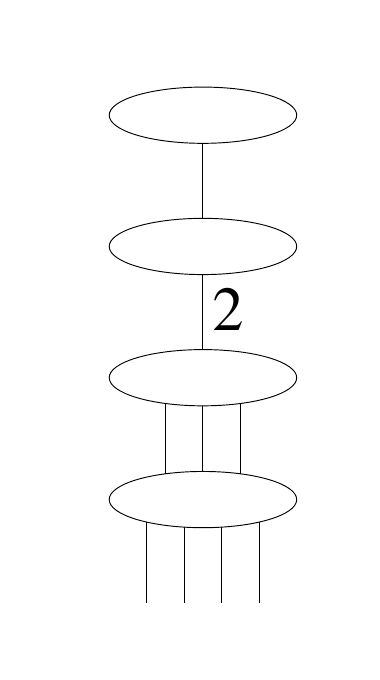}&
\includegraphics[height=3.5cm, angle=0]{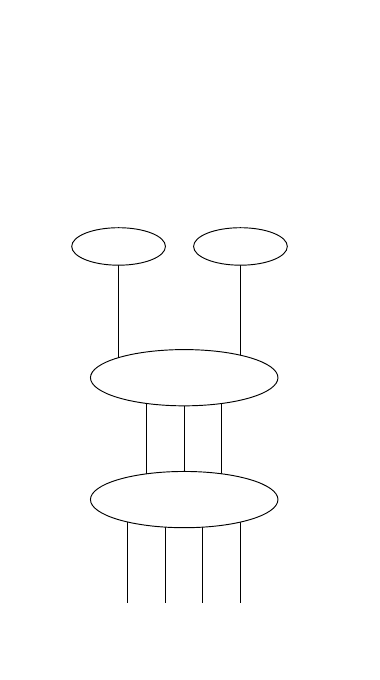}
\end{tabular}
\caption{Floor diagrams of genus 3 and 2 with Newton polygon
  $\Delta_4$}
\label{degree 4 g=3,2}
\end{figure}

\begin{figure}[h]
\centering
\begin{tabular}{cccccc}
\includegraphics[height=3.5cm, angle=0]{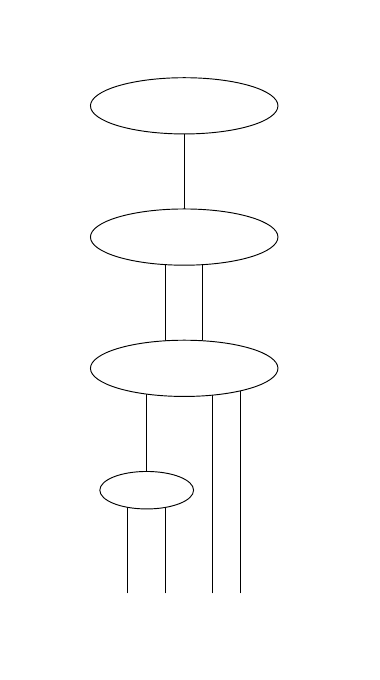}&
\includegraphics[height=3.5cm, angle=0]{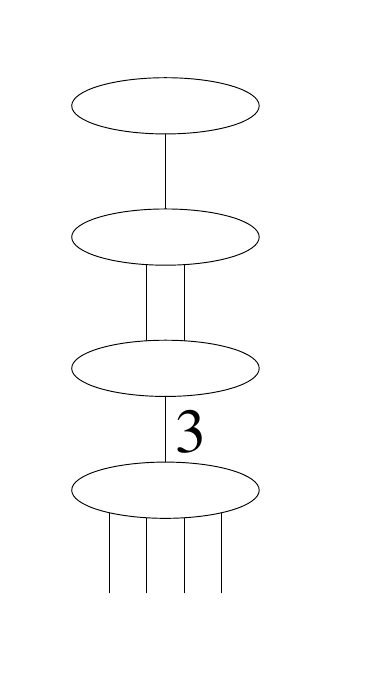}&
\includegraphics[height=3.5cm, angle=0]{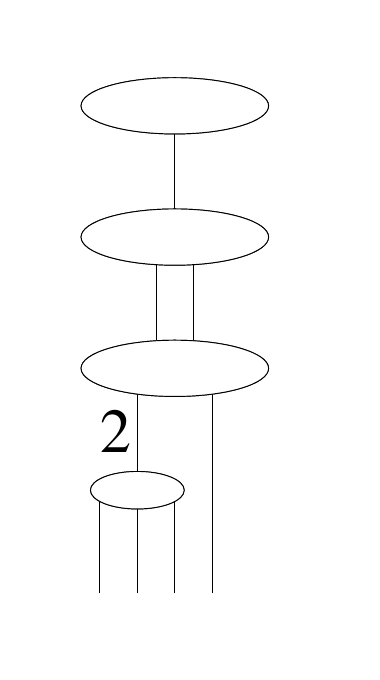}&
\includegraphics[height=3.5cm, angle=0]{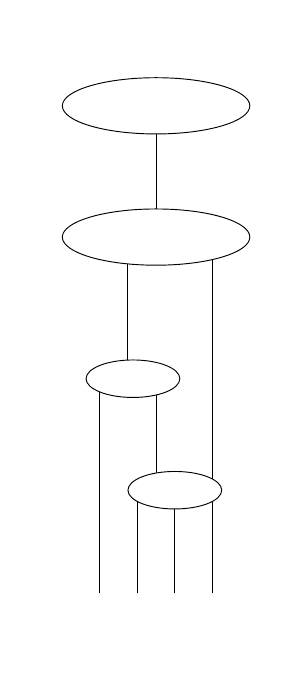}&
\includegraphics[height=3.5cm, angle=0]{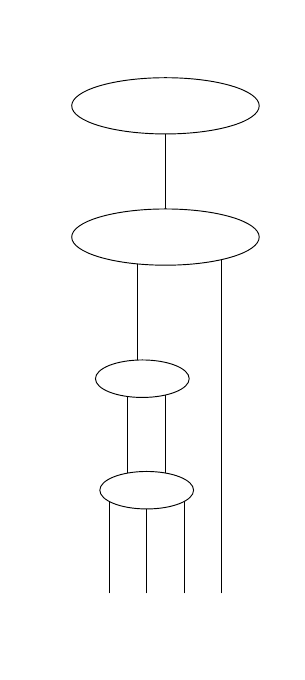}&
\includegraphics[height=3.5cm, angle=0]{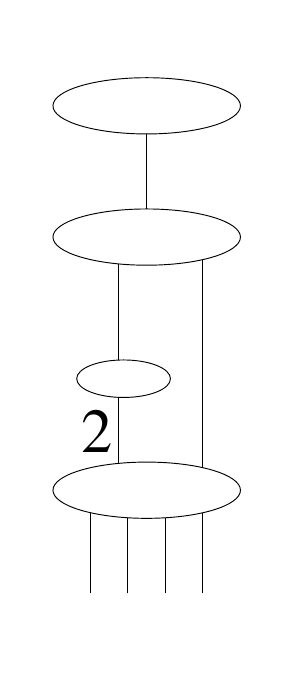}
\\
\includegraphics[height=3.5cm, angle=0]{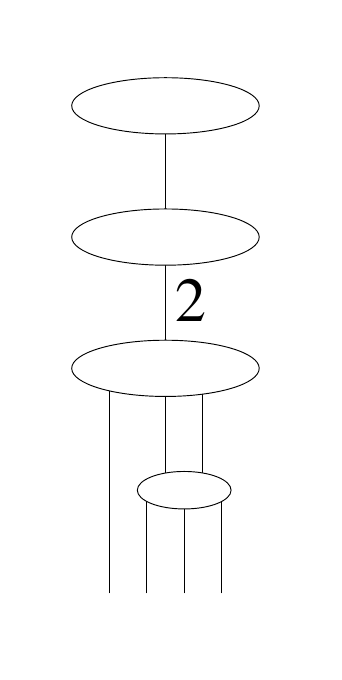}&
\includegraphics[height=3.5cm, angle=0]{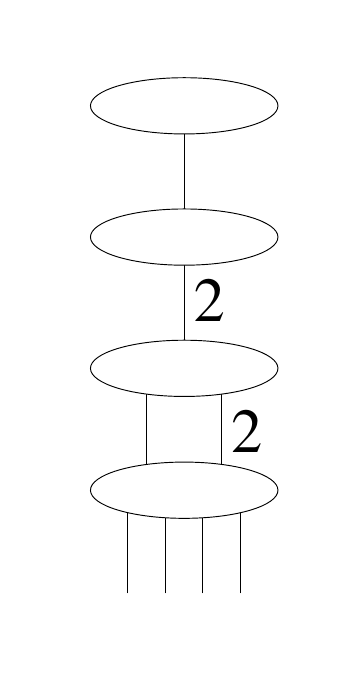}&
\includegraphics[height=3.5cm, angle=0]{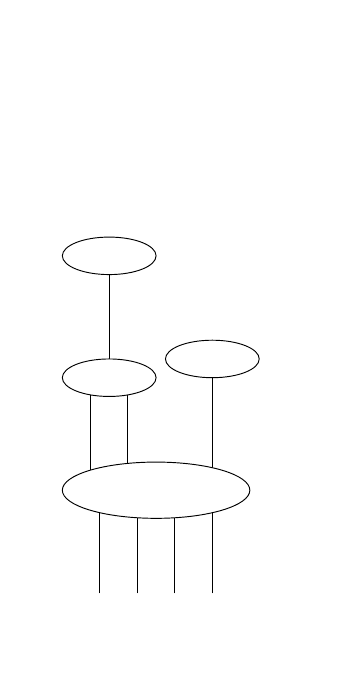}&
\includegraphics[height=3.5cm, angle=0]{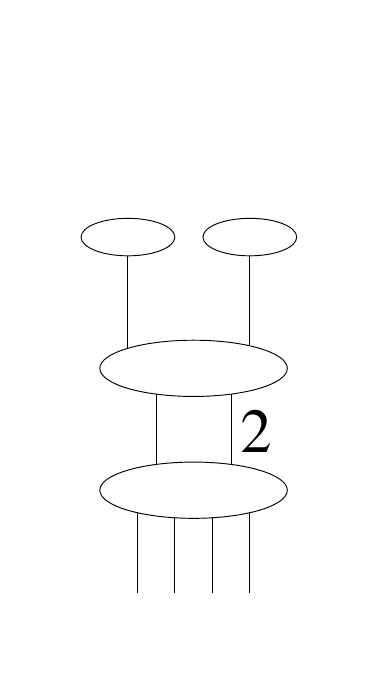}&
\includegraphics[height=3.5cm, angle=0]{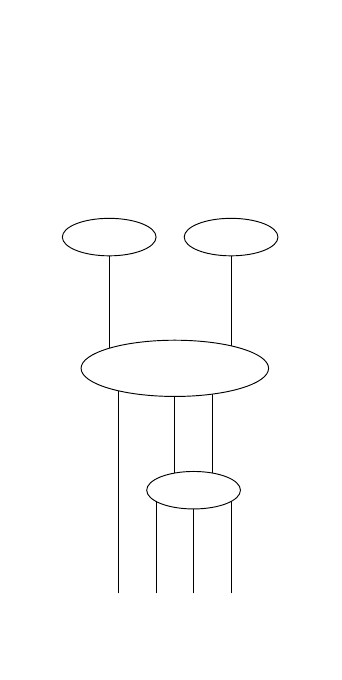}&

\end{tabular}
\caption{Floor diagrams of genus  1 with Newton polygon $\Delta_4$}
\label{degree 4 g=1}
\end{figure}

\begin{figure}[h]
\centering
\begin{tabular}{cccccc}
\includegraphics[height=3.5cm, angle=0]{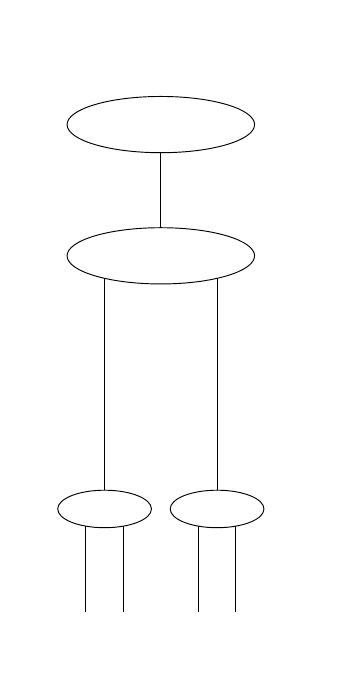}&
\includegraphics[height=3.5cm, angle=0]{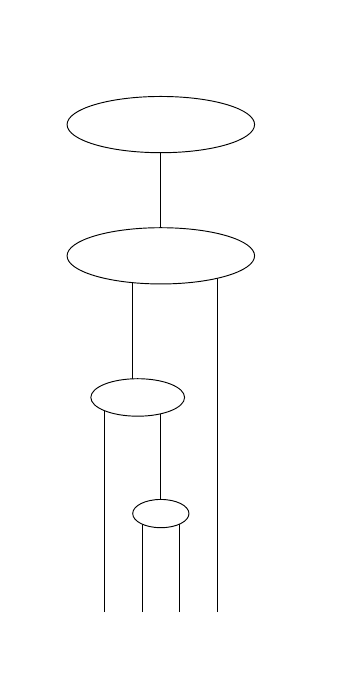}&
\includegraphics[height=3.5cm, angle=0]{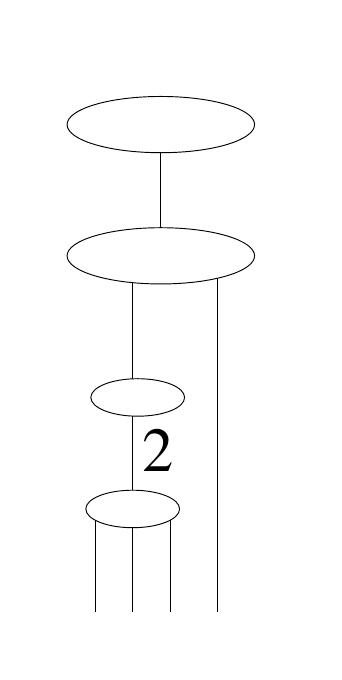}&
\includegraphics[height=3.5cm, angle=0]{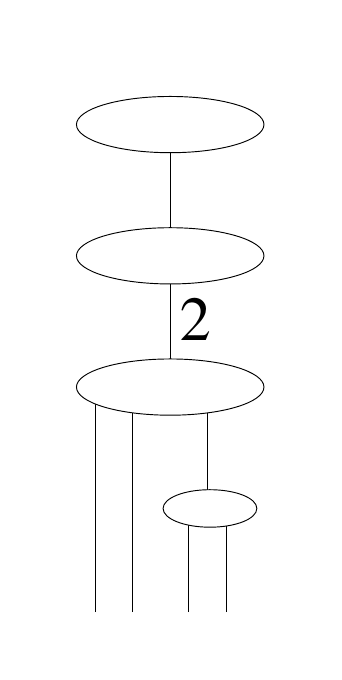}&
\includegraphics[height=3.5cm, angle=0]{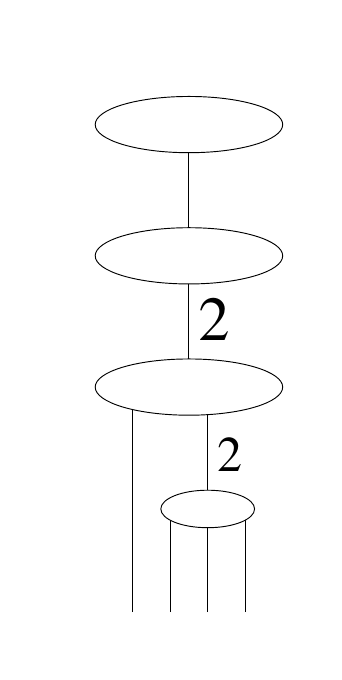}&
\includegraphics[height=3.5cm, angle=0]{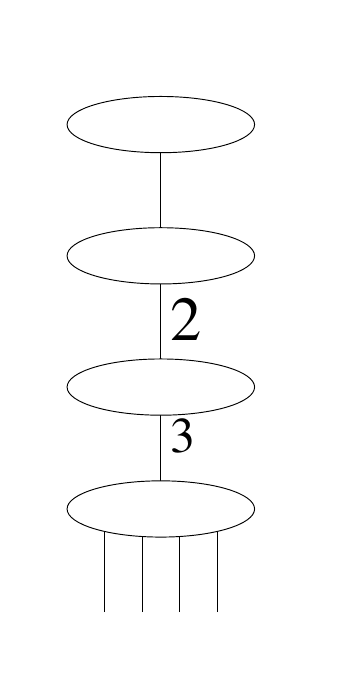}
\\
\includegraphics[height=3.5cm, angle=0]{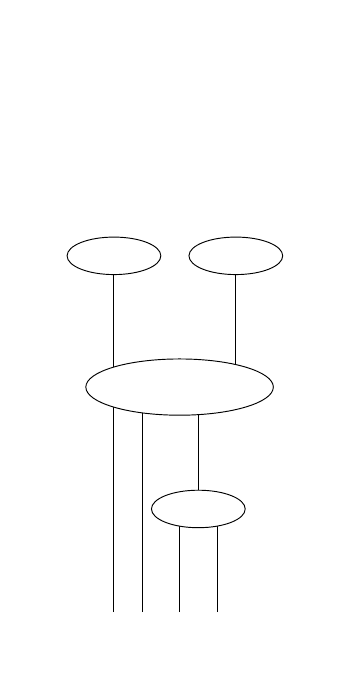}&
\includegraphics[height=3.5cm, angle=0]{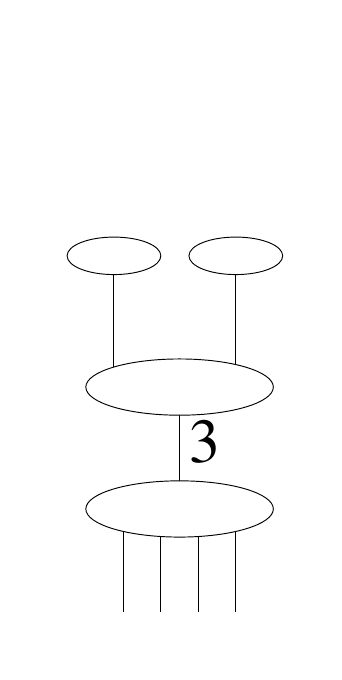}&
\includegraphics[height=3.5cm, angle=0]{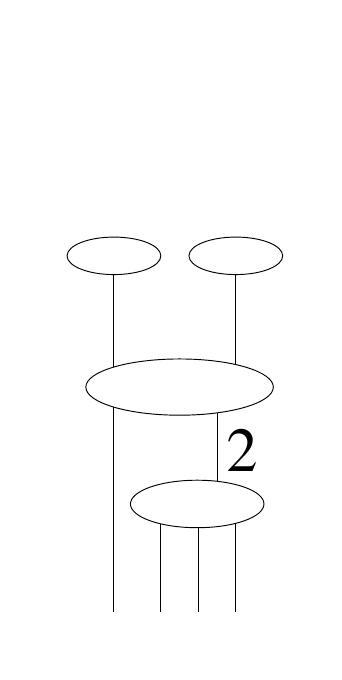}&
\includegraphics[height=3.5cm, angle=0]{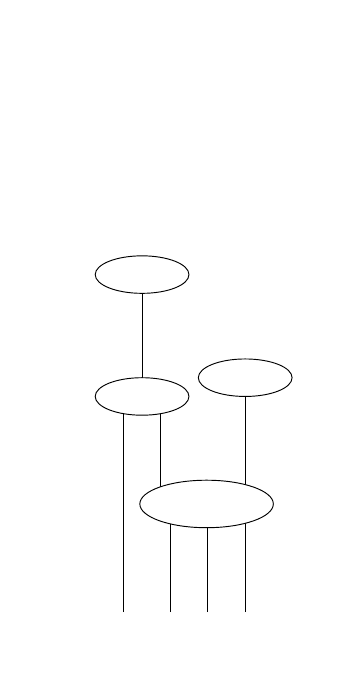}&
\includegraphics[height=3.5cm, angle=0]{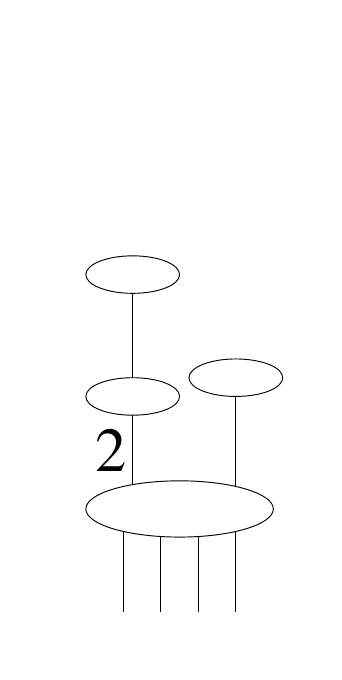}&
\includegraphics[height=3.5cm, angle=0]{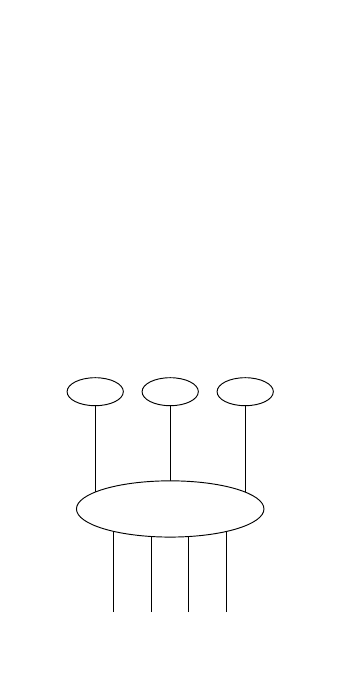}

\end{tabular}
\caption{Floor diagrams of genus  0 with Newton polygon $\Delta_4$}
\label{degree 4 g=0}
\end{figure}
\end{exa}

\subsection{Göttsche-Schroeter refined invariants via floor diagrams}

In the case when $g=0$, one can tweak the notion of marking of a
floor diagram in order to compute Göttsche-Schroeter refined
invariants.

\begin{defi}
A \emph{pairing of order $s$} of the set $\mathcal P=\{1,\cdots, n\}$ is a set
$S$ of $s$ disjoint pairs $\{i,i+1\}\subset \mathcal P$.

Given a floor diagram $\D$ and
a pairing $S$ of the set $\{1,\cdots,  \eta(\D)\}$, a marking
$m$ of $\D$ is said to be 
\emph{compatible with $S$} if for any $\{i,i+1\}\in S$, the set
$\{m(i),m(i+1)\}$ consists of one of the following sets
(see Figure \ref{fig:S-marking}): 
\begin{itemize}
\item an elevator and an adjacent
  floor;
  \item two elevators that have a common adjacent floor, from which  both
  are  emanating or both are ending. 
\end{itemize}
\end{defi}
\begin{figure}[h]
\begin{center}
\begin{tabular}{ccccc}
  \includegraphics[height=2.5cm, angle=0]{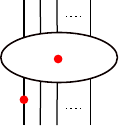}&
   \includegraphics[height=2.5cm, angle=0]{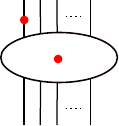}&
  \includegraphics[height=2.5cm, angle=0]{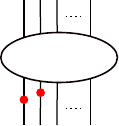}&
   \includegraphics[height=2.5cm, angle=0]{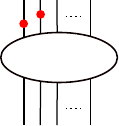}&
 \includegraphics[height=3.3cm, angle=0]{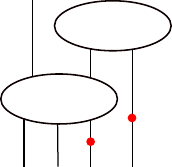}
 \\
 \\a) Compatible &b) Compatible &c) Compatible  &d) Compatible
 & e) Not compatible
\end{tabular}
\end{center}
\caption{Marking and pairing; the red dots corresponds to the image of
  $i$ and $i+1$.}
\label{fig:S-marking}
\end{figure}

We generalize the refined multiplicity of a marked floor
diagram in the presence of a pairing.
Given $(\D,m)$  a marked floor diagram compatible with a pairing $S$, 
we define the following sets of elevators of $\D$:
\[
\begin{array}{l}
  E_0=\{e\in E(\D)\ | \ e\notin m(S)\};
  \\ E_1=\{e\in E(\D)\ |\ \{e,v\}=m(\{i,i+1\})\mbox{ with }v\in V(\D)
  \mbox{ and } \{i,i+1\}\in S\};
  \\E_2=\left\{\{e,e'\}\subset E(\D)\ |\ \{e,e'\}=m(\{i,i+1\})\mbox{ with }
   \{i,i+1\}\in S\right\}.
\end{array}
\]

\begin{defi}\label{def:refined mult s}
  The refined $S$-multiplicity of a marked floor diagram $(\D,m)$  is defined by
  \[
  \mu_S(\D,m)(q)=\prod_{e\in E_0} [\omega(e)]^2(q)
    \prod_{e\in E_1}[\omega(e)](q^2)
\prod_{\{e,e'\}\in E_2} \frac{[\omega(e)]\times[\omega(e')]\times[\omega(e)+\omega(e')]}{[2]}(q)
\]
if  $(\D,m)$  is compatible with  $S$, and by
  \[
  \mu_S(\D,m)(q)=0
  \]
  otherwise.
\end{defi}

Clearly  $\mu_S(\D,m)(q)$ is 
 symmetric  in $q^{\frac{1}2}$, but more can be said. 
\begin{lemma}
For any marked floor diagram $(\D,m)$ compatible with a
pairing $S$, one has
\[
\mu_S(\D,m)(q)\in \Z_{\ge 0}[q^{\pm 1}].
\]
Furthermore $\mu_S(\D,m)(q)$ has degree $\deg(\D)$.
\end{lemma}
\begin{proof}
  The degree of $\mu_S(\D,m)(q)$ is clear.
  Next, the factors of $\mu_S(\D,m)(q)$
 coming  from 
  elevators in $E_0$ and $E_1$ are clearly in $\Z_{\ge 0}[q^{\pm 1}]$.
  Given a pair $\{e,e'\}$ in $E_3$, 
  one of the integers $\omega(e), \omega(e')$ or
 $\omega(e)+\omega(e')$ is even, and 
 the remaining two terms have the same parity. Hence it follows
 from Lemmas \ref{lem:quatum product} and \ref{lem:quantum div} that
 \[
 \frac{[\omega(e)]\times[\omega(e')]\times[\omega(e)+\omega(e')]}{[2]}(q)\in
 \Z_{\ge 0}[q^{\pm 1}],
 \]
 and the lemma is proved.
\end{proof}

Recall that we use the notation $\eta(\Delta)=\Card(\partial\Delta\cap \Z^2)-1$.
\begin{thm}\label{thm:psi fd}
Let $\Delta$ be an $h$-transverse polygon in $\R^2$, and let $s$ be a non-negative 
integer. Then for any pairing $S$ of order $s$ of
$\{1,\cdots , \eta(\Delta)\}$, one has
\[
G_\Delta(0;s)(q)=\sum_{(\D,m)}\mu_S(\D,m)(q),
\]
where the sum runs over all isomorphism classes of
marked floor diagrams with Newton polygon
$\Delta$ and  genus $0$.
\end{thm}
 \begin{proof}
Given a marked floor diagram $(\D,m)$ with Newton polygon
$\Delta$, of genus $0$, and compatible with $S$, we construct
a marked Psi-floor diagram of type $(\eta(\D)-2s,s)$ with
a fixed order induced by $S$ on the Psi-powers of the vertices
(in the terminology of \cite[Definition 4.1 and Remark 4.6]{BGM}), as depicted in
Figure \ref{fig:contraction} and its symmetry with respect to the
horizontal axis. This construction clearly establishes a
surjection $\Psi$ from the first set of floor diagrams to the second
one. Furthermore, given a marked Psi-floor diagram $(\D,m')$, all
marked
floor diagrams such that  $\Psi(\D,m)=\Psi(\D,m')$ are described by
the two conditions:
\begin{enumerate}
\item $m(\{i,i+1 \})=m'(\{i,i+1 \})$ if $\{i,i+1\}\in S$;
  \item $m(i)=m'(i)$ if $i$ does not belong to any pair in $S$.
\end{enumerate}
 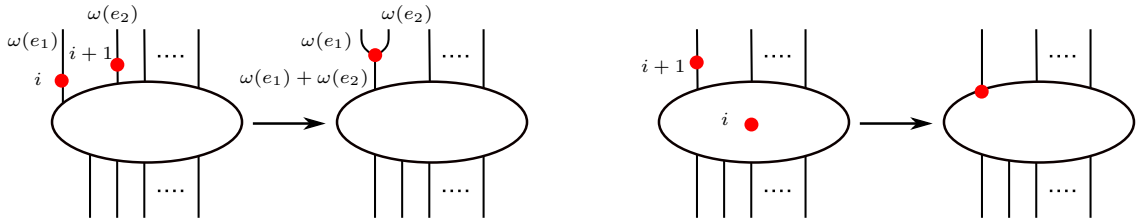
\begin{figure}[h]
\begin{center}
\begin{tabular}{ccc}
  \includegraphics[height=2.5cm, angle=0]{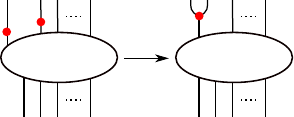}
   \put(-195,65){\tiny{$\omega(e_1)$}}
    \put(-165,75){\tiny{$\omega(e_2)$}}
 \put(-185,50){\tiny{$i$}}
   \put(-172,60){\tiny{$i+1$}}
   \put(-85,65){\tiny{$\omega(e_1)$}}
    \put(-55,75){\tiny{$\omega(e_2)$}}
    \put(-108,51){\tiny{$\omega(e_1)+\omega(e_2)$}}
& \hspace{1cm} &
  \includegraphics[height=2.5cm, angle=0]{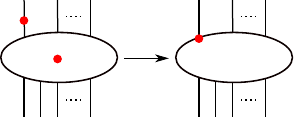}
 \put(-155,35){\tiny{$i$}}
   \put(-185,55){\tiny{$i+1$}}

\end{tabular}
\end{center}
\caption{From marked floor diagrams to Psi-floor diagrams
  ($\{i,i+1\}\in S$).}
\label{fig:contraction}
\end{figure}

 Substituting  the integer
 multiplicity of a Psi-floor diagram from \cite{BGM} of type
 $(k_0,k_1)$ by the refined
 multiplicity from \cite[Definition 3.1]{GotSch16}, see also
 \cite[Section 2.1]{BleShu17}, one 
 obtains the result.
 \end{proof}
 
 \begin{rem}
Theorem \ref{thm:psi fd} implies that the
right-hand side term only depends on $s$, and not on a particular
choice of $S$. This  does not look immediate to us. It may be
interesting to have a proof of this independence with respect to $S$
which does not go through tropical geometry as in \cite{GotSch16}.

Another type of pairing and multiplicities
has been proposed in \cite{Br6b} to compute
Welschinger invariants $W_{X_\Delta}(\mathcal L_\Delta;s)$, when $X_\Delta$
is a del Pezzo surface. Note that the multiplicities from \cite{Br6b}
do not coincide with the refined $S$-multiplicities defined in
Definition \ref{def:refined mult s} evaluated at $q=-1$.  
 \end{rem}

\begin{exa}\label{exa:fd cubic s}
  We continue Examples \ref{exa:ex fd} and \ref{exa:fd cubic}.
  All marked floor diagrams of  genus 0 and  Newton polygon $\Delta_3$
  are depicted in
  Table \ref{fig:cubic}. Below each of them, we write the multiplicity
  $\mu$ and the multiplicities $\mu_{S_i}$ for
  $S_i=\{(9-2i,10-2i),\cdots,(7,8)\}$.
 The first
floor diagram
has an elevator of weight 2, but we do not mention it in the picture to
avoid confusion.
According to Theorem \ref{thm:psi fd} we find
$G_{\Delta_3}(0;s)=q+10-2s+q^{-1}$.
It is interesting to compare this computation with \cite[Example 3.10]{Br6b}.
\end{exa}

\begin{table}[h]
\centering
\begin{tabular}{c|c|c|c|c|c|c|c|c|c}
&
\includegraphics[height=2.5cm, angle=0]{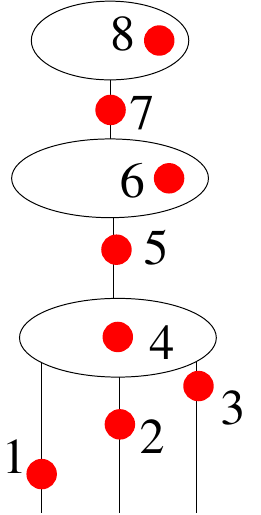}&
\includegraphics[height=2.5cm, angle=0]{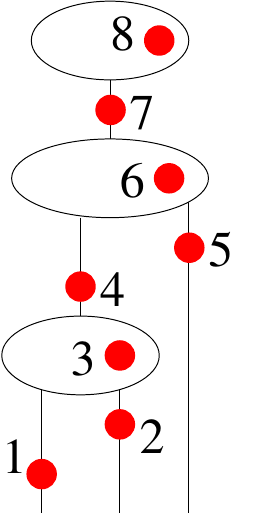}&
\includegraphics[height=2.5cm, angle=0]{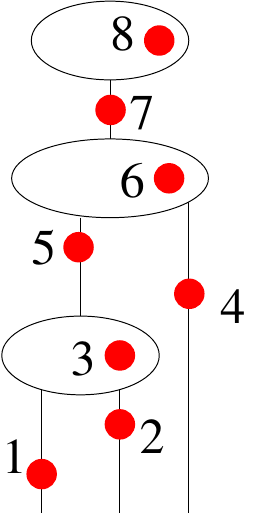}&
\includegraphics[height=2.5cm, angle=0]{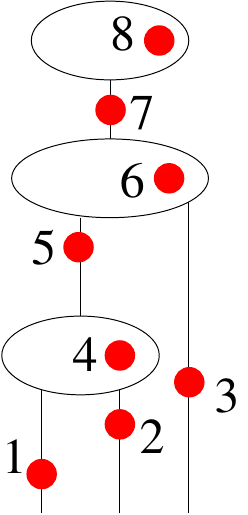}&
\includegraphics[height=2.5cm, angle=0]{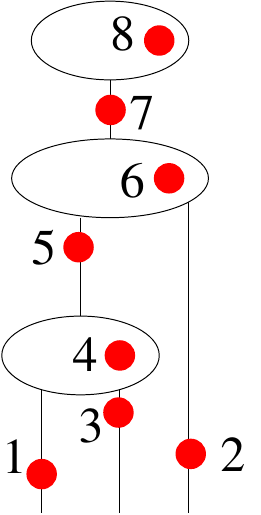}&
\includegraphics[height=2.5cm, angle=0]{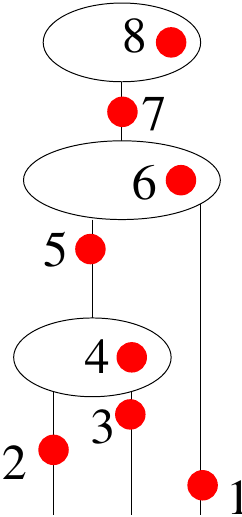}&
\includegraphics[height=2.5cm, angle=0]{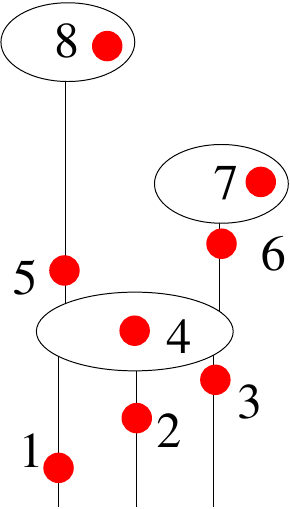}&
\includegraphics[height=2.5cm, angle=0]{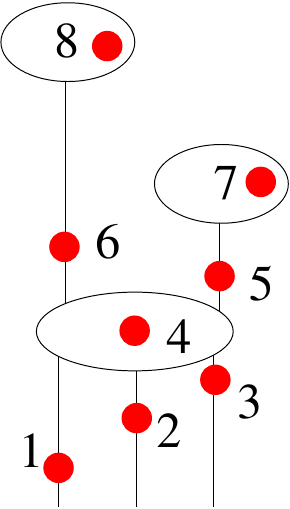}&
\includegraphics[height=2.5cm, angle=0]{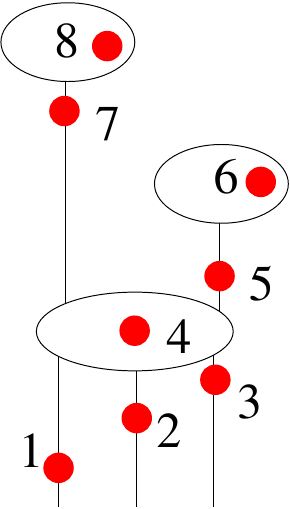}
\\ \hline $ \mu$ & $q+2+q^{-1}$      & 1 &1 &1 &1 &1 &1 &1 &1
\\ \hline $ \mu_{S_1}$ &$q+2+q^{-1}$  & 1 &1 &1 &1 &1 &0 &0 &1
\\ \hline $ \mu_{S_2}$ & $q+q^{-1}$   & 1 &1 &1 &1 &1 &0 &0 &1
\\ \hline $ \mu_{S_3}$ & $q+q^{-1}$   & 1 &0 &0 &1 &1 &0 &0 &1
\\ \hline $ \mu_{S_4}$ & $q+q^{-1}$   & 1 &0 &0 &0 &0 &0 &0 &1

\end{tabular}
\\
\begin{tabular}{c}
\end{tabular}
\caption{Computation of $G_{\Delta_3}(0;s)$.}
\label{fig:cubic}
\end{table}

The following proposition states  that the decreasing of 
$\mu_{S}(\D,m)$ with respect to $S$ that one  observes in Table
\ref{fig:cubic} is actually a general phenomenon.
    Given two elements $f,g  \in\Z_{\ge 0}[q^{\pm 1}]$, we
      write $f\ge g$ if $f-g  \in\Z_{\ge 0}[q^{\pm 1}]$.
\begin{prop}\label{prop:decreasing}
Let $(\D,m)$ be a marked floor diagram of genus 0, and $S_1\subset S_2$ be two
pairings of the set $\{1,\cdots,\eta(\D)\}$.
Then one has
\[
\mu_{S_1}(D,m)(q)\ge \mu_{S_2}(D,m)(q).
\]  
\end{prop}
\begin{proof}
  Since $\mu_{S_1}(D,m)\in \Z_{\ge 0}[q^{\pm 1}]$, the result 
    obviously holds if $\mu_{S_2}(D,m)=0$. If $\mu_{S_2}(D,m)\ne 0$,
    then the
    result follows from Corollary \ref{cor:quantum ineq2}, and from the
    inequality
    \[
[k](q^2)\le [2k-1](q)\le [k]^2(q),
    \]
    the last inequality holding by Lemma \ref{lem:quatum product}.
\end{proof}

The next corollary generalizes \cite[Corollary 4.5]{Bru18} to arbitrary
$h$-transverse polygon. Recall that we use the notation
\[
s_{max}=\left[ \frac{\eta(\Delta)}{2}\right].
\]
\begin{cor}\label{cor:decreasing}
  For any  $h$-transverse polygon  $\Delta$ in $\R^2$ and any
  $i\in\Z_{\ge 0}$,  one has
\[
\coef[i]{G_\Delta(0;0)}\ge
\coef[i]{G_\Delta(0;1)}\ge
\coef[i]{G_\Delta(0;2)}\ge \cdots \ge
\coef[i]{G_\Delta\left(0;s_{max}\right)}\ge 0.
\]
\end{cor}
\begin{proof}
Since  $\mu_{S}(D,m)\in \Z_{\ge 0}[q^{\pm 1}]$ for any marked floor
diagram $(\D,m)$ and any pairing $S$, we have that
$\coef[i]{G_\Delta(0;s)}\ge 0$ for
  any $s$. The decreasing of the sequence $(\coef[i]{G_\Delta(0;s)})_s$ is a direct consequence
  of Proposition \ref{prop:decreasing}  and Theorem \ref{thm:psi fd}.
\end{proof}

\begin{exa}\label{exa:G quartic}
  Thanks to Figure \ref{degree 4 g=0}, one can compute:
  \[
  \begin{array}{cclclclclclclcl}
    G_{\Delta_4}(0;0)&=&  q^{-3} & +&  13q^{-2} &+ &94q^{-1}& +& 404
    & + &94 q &+ &13 q^2& +&  q^3
\\    G_{\Delta_4}(0;1)&=&  q^{-3} & +&  11q^{-2} &+ &70q^{-1}& +& 264
    & + &70 q &+ &11 q^2& +&  q^3
\\    G_{\Delta_4}(0;2)&=&  q^{-3} & +&  9q^{-2} &+ &50q^{-1}& +& 164
    & + &50 q &+ &9 q^2& +&  q^3
\\    G_{\Delta_4}(0;3)&=&  q^{-3} & +&  7q^{-2} &+ &34q^{-1}& +& 96
    & + &34 q &+ &7 q^2& +&  q^3
\\    G_{\Delta_4}(0;4)&=&  q^{-3} & +&  5q^{-2} &+ &22q^{-1}& +& 52
    & + &22 q &+ &5 q^2& +&  q^3
\\    G_{\Delta_4}(0;5)&=&  q^{-3} & +&  3q^{-2} &+ &14q^{-1}& +& 24
    & + &14 q &+ &3 q^2& +&  q^3
  \end{array}
  \]
  
\end{exa}

A particular case of Corollary \ref{cor:decreasing} has first been
proved  in
\cite{Bru18} using 
the 
next proposition. For the sake of brevity, the
proof of
\cite[Proposition 4.3]{Bru18} has been omitted there. We close this
gap here.

\begin{figure}[h]
\begin{center}
\begin{tabular}{ccccccc}
  \includegraphics[height=2.5cm, angle=0]{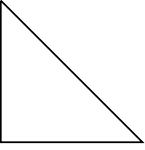}
   \put(-80,-5){\tiny{$(0,0)$}}
   \put(-90,70){\tiny{$(0,d)$}}
      \put(-10,-5){\tiny{$(d,0)$}}

& \hspace{0.5cm} &
  \includegraphics[height=2.5cm, angle=0]{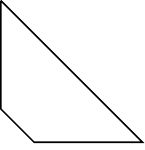}
  \put(-90,20){\tiny{$(0,a)$}}
    \put(-60,-5){\tiny{$(a,0)$}}
   \put(-90,70){\tiny{$(0,d)$}}
      \put(-10,-5){\tiny{$(d,0)$}}
& \hspace{0.5cm} &
  \includegraphics[height=2.5cm, angle=0]{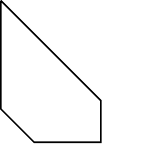}
  \put(-90,20){\tiny{$(0,a)$}}
    \put(-60,-5){\tiny{$(a,0)$}}
   \put(-90,70){\tiny{$(0,d)$}}
      \put(-18,2){\tiny{$(d-b,0)$}}
      \put(-18,20){\tiny{$(d-b,b)$}}
 & \hspace{0.5cm} &
  \includegraphics[height=2cm, angle=0]{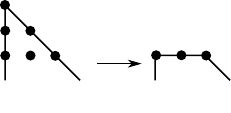}
       \put(-100,-5){$\Delta$}
      \put(-25,-5){$ \widetilde\Delta$}
    \\
      \\ a) $d\ge 2$ && b) $d-a\ge 2$
       && c) $d-\max(a,b)\ge 2$ && d) 
\end{tabular}
\end{center}
\caption{}
\label{fig:np wel}
\end{figure}
\begin{prop}[{\cite[Proposition 4.3]{Bru18}}]\label{prop:wel}
  Let $\Delta$ be one of the integer polygons depicted in Figures 
  \ref{fig:np wel}a),b), or c), 
  and let $\widetilde \Delta$ be the integer polygon
  obtained by chopping off 
  the top of $\Delta$ as depicted in Figure
  \ref{fig:np wel}d).
If $2s\le \eta(\Delta)-2$, then one has
  \[
  G_{\Delta}(0;s+1)=G_{\Delta}(0;s) - 2G_{\widetilde \Delta}(0;s).
  \]
\end{prop}
\begin{proof}
   Let $S$ be a pairing of order $s$ of
  the set $\{1,\cdots,\eta(\Delta)-2\}$,
  and let $\widetilde S=S\cup\{\eta(\Delta)-1,\eta(\Delta)\}$.
  Let $A\sqcup B$ be the partition of
  the set of marked floor diagrams $(\D,m)$
  with Newton polygon
  $\Delta$ and genus 0 such that
  \[
   (\D,m)\in B\Longleftrightarrow m(\{\eta(\Delta)-1,\eta(\Delta)\})\subset V(\D),
  \]
  see Figure \ref{fig:QWel}.  
\begin{figure}[h]
\begin{center}
\begin{tabular}{ccc}
  \includegraphics[height=2cm, angle=0]{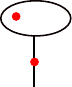}
  & \hspace{1cm} &
  \includegraphics[height=2cm, angle=0]{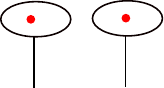}

      \\ a) $A$ && b) $B$
\end{tabular}
\end{center}
\caption{A partition of the set of marked floor diagrams, the red dots
represent $m(\eta(\Delta)-1)$ and $m(\eta(\Delta))$}
\label{fig:QWel}
\end{figure}
Then by Theorem \ref{thm:psi fd}, one has
\[
G_\Delta(0;s)(q)=\sum_{(\D,m)\in A\cup B}\mu_S(\D,m)(q)  \qquad\mbox{and}
\qquad G_\Delta(0;s+1)(q)=\sum_{(\D,m)\in A}\mu_{\widetilde S}(\D,m)(q).
\]
By chopping off 
the two top floors of the floor diagrams from the set
$B$, it follows again from Theorem \ref{thm:psi fd} that
\[
G_{\widetilde \Delta}(0;s)(q)=\frac{1}{2}\sum_{(\D,m)\in  B}\mu_S(\D,m)(q).
\]
Since the divergence of any top floor of $\D$ is 1, we have
that $\mu_{\widetilde S}(\D,m)(q)=\mu_{ S}(\D,m)(q)$ for any
marked floor diagram $(\D,m)\in A$. Hence we have
\begin{align*}
  G_\Delta(0;s+1)(q)&=\sum_{(\D,m)\in A}\mu_{S}(\D,m)(q)
  \\ & = \sum_{(\D,m)\in A\cup B}\mu_S(\D,m)(q)  - \sum_{(\D,m)\in  B}\mu_S(\D,m)(q),
\end{align*}
which concludes the proof.
\end{proof}


\section{Codegrees}\label{sec:codeg}

\subsection{Codegree of a floor diagram}
Recall that we use the notation
\[
\iota_\Delta=\Card(\Int(\Delta)\cap \Z^2).
\]
We define the \emph{codegree} of a floor diagram $\D$ of genus $g$ with Newton
polygon $\Delta$ by
\[
\codeg(\D)= \iota_\Delta-g - \deg(\D).
\]
By {\cite[Proposition 2.11]{IteMik13}}, one has
$\deg(\D)\le \iota_\Delta-g$, and so
$\codeg(\D)\ge 0$.
Furthermore, the codegree of $\D$ is zero if and only if
\begin{itemize}
\item the order $\preccurlyeq$ is total on the set of floors of $\D$;
  \item any edge in $E^0(\D)$ connects two consecutive vertices;
  \item elevators in $E^{-\infty}(\D)$ are all adjacent to the minimal floor
    of $\D$, and elevators in $E^{+\infty}(\D)$ are all adjacent to the maximal floor
    of $\D$;
\item the function $l:V(\D)\to d_l$ is decreasing, and the function
  $r:V(\D)\to d_r$ is increasing.
\end{itemize}
With this characterization, one sees easily that there exists exactly
${\iota_\Delta\choose g}$ marked floor diagrams of
genus $g$ with Newton 
polygon $\Delta$ and codegree 0 (see {\cite[Proposition
    2.11]{IteMik13} and \cite[Proposition 4.10]{BlGo14}}).
\begin{exa}
Figures \ref{fig:ex FD}a) and b) depict  the
only floor diagrams of codegree 0 with Newton polygon~$\Delta_3$ and
genus 1 and 0, respectively.
The only codegree 0 floor diagram with
Newton polygon depicted in Figure \ref{fig:htrans}b)  and genus 0 is
depicted in Figure
\ref{fig:ex codeg}a). All codegree 0 floor diagrams with
Newton polygon depicted in Figure \ref{fig:htrans}b)  and genus 1 are
depicted in Figures 
\ref{fig:ex codeg}b), c),  d), e), f), and g). Note that the floor
diagram depicted in Figure \ref{fig:ex codeg}b) admits a single
marking, while the floor
diagrams depicted in Figures
 \ref{fig:ex codeg}c),  d), e), f), and g)
admit exactly two different markings.

\begin{figure}[h!]
\begin{center}
\begin{tabular}{ccccccccccccc}
  \includegraphics[height=4cm, angle=0]{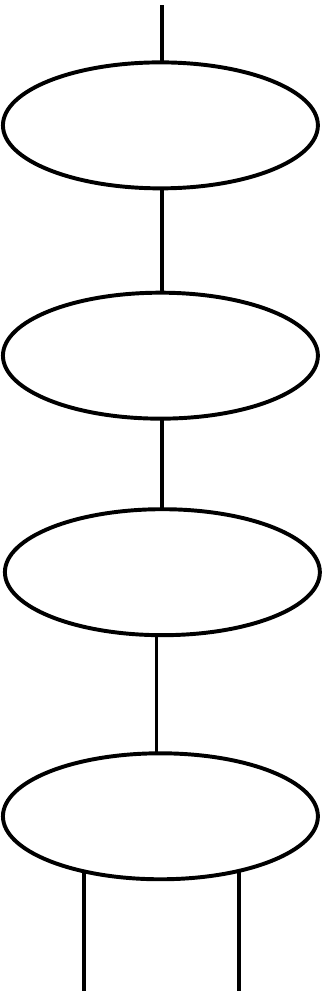}
   \put(-28,31){\small{$4$}}
  \put(-28,58){\small{$5$}}
  \put(-28,84){\small{$5$}}
  \put(-33,17){\small{$1$}}
  \put(-18,17){\small{$-1$}}
   \put(-33,46){\small{$1$}}
  \put(-11,46){\small{$0$}}
   \put(-33,70){\small{$0$}}
  \put(-11,70){\small{$0$}}
   \put(-33,97){\small{$-2$}}
  \put(-11,97){\small{$2$}}
 & \hspace{0.4cm} &
  \includegraphics[height=4cm, angle=0]{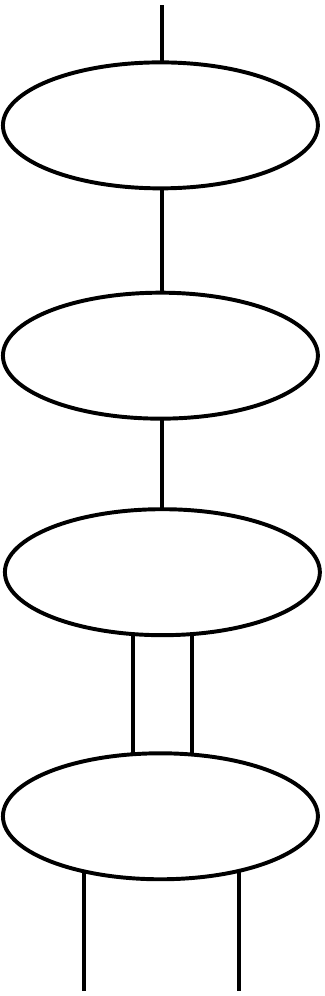}
   \put(-31,31){\small{$2$}}
    \put(-13,31){\small{$2$}}
 \put(-28,58){\small{$5$}}
  \put(-28,84){\small{$5$}}
  \put(-33,17){\small{$1$}}
  \put(-18,17){\small{$-1$}}
   \put(-33,46){\small{$1$}}
  \put(-11,46){\small{$0$}}
   \put(-33,70){\small{$0$}}
  \put(-11,70){\small{$0$}}
   \put(-33,97){\small{$-2$}}
  \put(-11,97){\small{$2$}}
 & \hspace{0.4cm} &
  \includegraphics[height=4cm, angle=0]{Figures1/FDht3.pdf}
    \put(-31,31){\small{$1$}}
    \put(-13,31){\small{$3$}}
  \put(-28,58){\small{$5$}}
  \put(-28,84){\small{$5$}}
  \put(-33,17){\small{$1$}}
  \put(-18,17){\small{$-1$}}
   \put(-33,46){\small{$1$}}
  \put(-11,46){\small{$0$}}
   \put(-33,70){\small{$0$}}
  \put(-11,70){\small{$0$}}
   \put(-33,97){\small{$-2$}}
  \put(-11,97){\small{$2$}}
 & \hspace{0.4cm} &
  \includegraphics[height=4cm, angle=0]{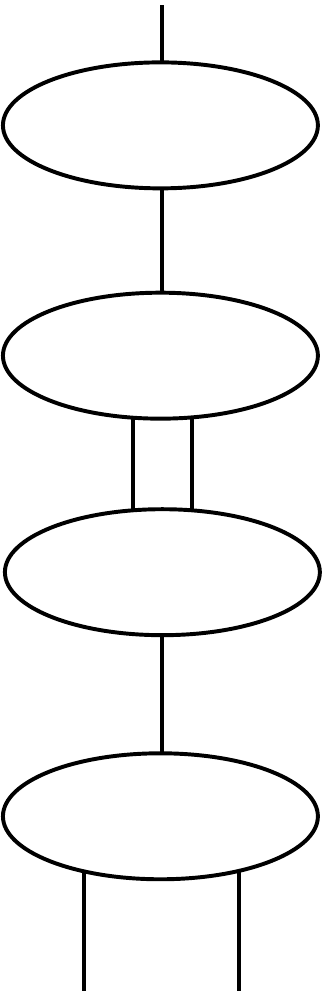}
   \put(-28,31){\small{$4$}}
      \put(-31,58){\small{$1$}}
    \put(-13,58){\small{$4$}}
  \put(-28,84){\small{$5$}}
  \put(-33,17){\small{$1$}}
  \put(-18,17){\small{$-1$}}
   \put(-33,46){\small{$1$}}
  \put(-11,46){\small{$0$}}
   \put(-33,70){\small{$0$}}
  \put(-11,70){\small{$0$}}
   \put(-33,97){\small{$-2$}}
  \put(-11,97){\small{$2$}}
 & \hspace{0.4cm} &
  \includegraphics[height=4cm, angle=0]{Figures1/FDht4.pdf}
   \put(-28,31){\small{$4$}}
       \put(-31,58){\small{$2$}}
    \put(-13,58){\small{$3$}}
  \put(-28,84){\small{$5$}}
  \put(-33,17){\small{$1$}}
  \put(-18,17){\small{$-1$}}
   \put(-33,46){\small{$1$}}
  \put(-11,46){\small{$0$}}
   \put(-33,70){\small{$0$}}
  \put(-11,70){\small{$0$}}
   \put(-33,97){\small{$-2$}}
  \put(-11,97){\small{$2$}}
 & \hspace{0.4cm} &
  \includegraphics[height=4cm, angle=0]{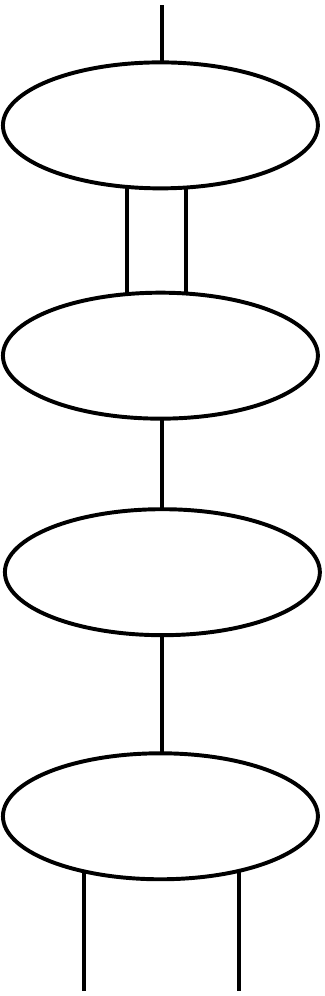}
   \put(-28,31){\small{$4$}}
  \put(-28,58){\small{$5$}}
      \put(-31,84){\small{$1$}}
    \put(-13,84){\small{$4$}}
  \put(-33,17){\small{$1$}}
  \put(-18,17){\small{$-1$}}
   \put(-33,46){\small{$1$}}
  \put(-11,46){\small{$0$}}
   \put(-33,70){\small{$0$}}
  \put(-11,70){\small{$0$}}
   \put(-33,97){\small{$-2$}}
  \put(-11,97){\small{$2$}}
 & \hspace{0.4cm} &
  \includegraphics[height=4cm, angle=0]{Figures1/FDht5.pdf}
   \put(-28,31){\small{$4$}}
  \put(-28,58){\small{$5$}}
      \put(-31,84){\small{$2$}}
    \put(-13,84){\small{$3$}}
  \put(-33,17){\small{$1$}}
  \put(-18,17){\small{$-1$}}
   \put(-33,46){\small{$1$}}
  \put(-11,46){\small{$0$}}
   \put(-33,70){\small{$0$}}
  \put(-11,70){\small{$0$}}
   \put(-33,97){\small{$-2$}}
  \put(-11,97){\small{$2$}}

  \\ a)&&b)&&c)&&d)&&e)&&f)&&g)
\end{tabular}
\end{center}
\caption{Codegree 0 floor diagrams of genus 0 and 1 with
Newton polygon from Figure \ref{fig:htrans}b).}
\label{fig:ex codeg}
\end{figure}
\end{exa}

Throughout the remainder of
the text, we  will make an extensive use of
the following four operations on a  floor diagram $\D$:
\begin{enumerate}
\item[$A^+$:] Suppose that there exist two floors
  $v_1$ and $v_2$ of  $\D$ connected by an elevator
  $e_1$ from $v_1$ to $v_2$, and an additional elevator $e_2$
  originating from $v_1$ but not adjacent to $v_2$. Then construct a
  new floor diagram  $\D'$ out of $\D$ as depicted in Figure
  \ref{fig:operationA}a).
  
\item[$A^-$:] Suppose that there exist two floors
  $v_1$ and $v_2$ of  $\D$ connected by an elevator
  $e_1$ from $v_1$ to $v_2$, and an additional elevator $e_2$
  ending at $v_2$ but not adjacent to $v_1$. Then construct a
  new floor diagram   $\D'$ out of $\D$ as depicted in Figure
  \ref{fig:operationA}b).
\begin{figure}[h]
\begin{center}
\begin{tabular}{ccc}
  \includegraphics[height=3cm, angle=0]{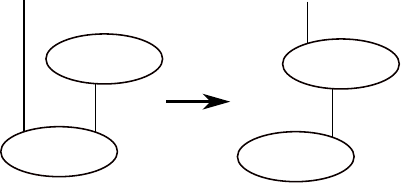}
   \put(-136,30){\tiny{$\omega(e_1)$}}
  \put(-203,58){\tiny{$\omega(e_2)$}}
   \put(-25,30){\tiny{$\omega(e_1)+\omega(e_2)$}}
  \put(-70,78){\tiny{$\omega(e_2)$}}
  \put(-160,-20){$\D$}
  \put(-55,-20){$\D'$}
 & \hspace{1cm} &
  \includegraphics[height=3cm, angle=0]{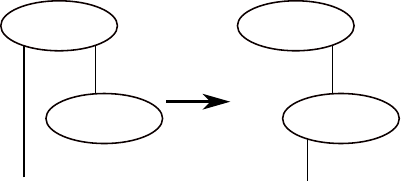}
   \put(-136,55){\tiny{$\omega(e_1)$}}
  \put(-203,48){\tiny{$\omega(e_2)$}}
   \put(-25,55){\tiny{$\omega(e_1)+\omega(e_2)$}}
  \put(-70,8){\tiny{$\omega(e_2)$}}
  \put(-160,-20){$\D$}
  \put(-55,-20){$\D'$}

\\  \\ a) Operation $A^+$ &&b) Operation $A^-$ 
\end{tabular}
\end{center}
\caption{Operations $A$ on floor diagrams.}
\label{fig:operationA}
\end{figure}
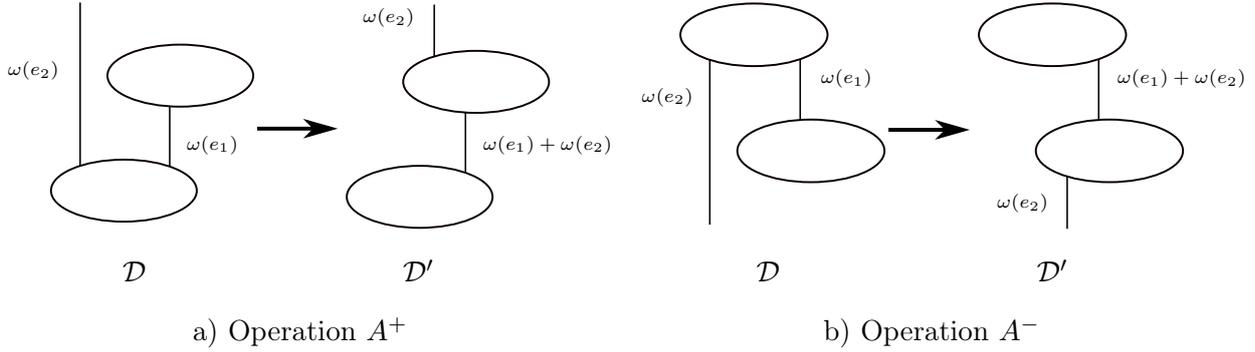

    \item[$B^l$:] Suppose that there exist two
    consecutive floors 
    $v_1 \preccurlyeq v_2$ of  $\D$ such that $l(v_1)<l(v_2)$.
    Then construct a
  new floor diagram   $\D'$ out of $\D$ as depicted in Figure
  \ref{fig:operationB}a), where $e$ is any elevator adjacent to $v_1$ and $v_2$.

    \item[$B^r$:] Suppose that there exist two
    consecutive floors 
    $v_1 \preccurlyeq v_2$ of  $\D$ such that $r(v_1)>r(v_2)$.
    Then construct a
  new floor diagram   $\D'$ out of $\D$ as depicted in Figure
  \ref{fig:operationB}b), where $e$ is any elevator adjacent to $v_1$ and $v_2$.
\begin{figure}[h]
\begin{center}
\begin{tabular}{ccc}
  \includegraphics[height=2.5cm, angle=0]{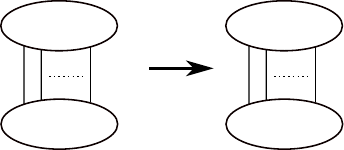}
   \put(-113,26){\tiny{$\omega(e)$}}
  \put(-155,58){\tiny{$l(v_2)$}}
   \put(-155,11){\tiny{$l(v_1)$}}
   \put(-8,26){\tiny{$\omega(e)+l(v_2)-l(v_1)$}}
  \put(-50,58){\tiny{$l(v_1)$}}
   \put(-50,11){\tiny{$l(v_2)$}}
   \put(-140,-20){$\D$}
  \put(-35,-20){$\D'$}
& \hspace{1.8cm} &
  \includegraphics[height=2.5cm, angle=0]{Figures1/OperationB.pdf}
   \put(-113,26){\tiny{$\omega(e)$}}
  \put(-132,58){\tiny{$r(v_2)$}}
   \put(-132,11){\tiny{$r(v_1)$}}
   \put(-8,26){\tiny{$\omega(e)+r(v_1)-r(v_2)$}}
  \put(-26,58){\tiny{$r(v_1)$}}
   \put(-26,11){\tiny{$r(v_2)$}}
   \put(-140,-20){$\D$}
  \put(-35,-20){$\D'$}
  \\
  \\a) Operation $B^l$ &&b)  Operation $B^r$ 
\end{tabular}
\end{center}
\caption{Operations $B$ on floor diagrams.}
\label{fig:operationB}
\end{figure}

\end{enumerate}

The following  lemma is straightforward.
\begin{lemma}\label{lem:codegree}
Genus and  Newton polygon  are invariant under
  operations $A^\pm$, $B^l$, and $B^r$.
  Furthermore, the codegree  decreases by $w(e_2)$ under operations $A^\pm$,
   by $l(v_2)-l(v_1)$ under operations $B^l$, and 
 by $r(v_1)-r(v_2)$ under operations $B^r$.
\end{lemma}

The next lemma is an example of application of Lemma \ref{lem:codegree}.
For the sake of simplicity, we  state and prove it only for
floor diagrams with constant divergence. Generalizing it to  floor
diagrams with any $h$-transverse Newton polygon
presents no difficulties other than technical.
\begin{lemma}\label{lem:min}
Let $\D$ be a floor diagram with constant divergence $n\in  \Z$.
If $\D$ has $k$ minimal floors,
then one has that
\[
\codeg(\D)\ge (k-1)\ \left(\Card(E^{-\infty}(\D)) -n\frac{k}{2}\right).
\]
\end{lemma}
\begin{proof}
Denote by $v_1,\cdots,v_k$ these minimal floors, and by $u_i$ the
number of elevators in $E^{-\infty}(\D)$ to which $v_i$ is adjacent.
By a
finite succession of operations $A^-$ and applications 
of Lemma
\ref{lem:codegree},
we may assume that
\[
\sum_{i=1}^ku_i=\Card(E^{-\infty}(\D)).
\]
Next, by a
finite succession of operations $A^\pm$  and applications 
of Lemma
\ref{lem:codegree}, we may assume that there exists $v\in V(\D)$
greater than all floors
$v_1,\cdots,v_k$, and such that any
elevator in $E(\D)\setminus E^{-\infty}(\D)$ adjacent to $v_i$ is also adjacent to $v$, see
Figure \ref{fig:op lower}a).
\begin{figure}[h]
\begin{center}
\begin{tabular}{ccc}
  \includegraphics[height=3.5cm, angle=0]{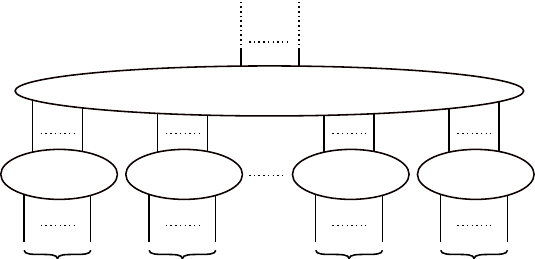}
   \put(-140,-10){$u_2$}
  \put(-25,-10){$u_k$}
   \put(-190,-10){$u_1$}
  \put(-80,-10){$u_{k-1}$}
& \hspace{1cm} &
  \includegraphics[height=3.5cm, angle=0]{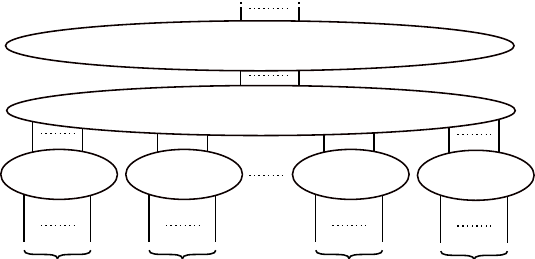}
   \put(-140,-10){$u_2$}
   \put(-190,-10){$u_1$}
    \put(-80,-10){$u_{k-2}$}
\put(-45,-10){$u_{k-1}+u_k$}
  \\
  \\a) $\D$ &&b)  $\D'$ 
\end{tabular}
\end{center}
\caption{}
\label{fig:op lower}
\end{figure}
This implies in particular that if $e_{i,1},\cdots,e_{i,k_i}$ are the
elevators in $E^0(\D)$ adjacent to $v_i$, then one has
\[
\sum_{j=1}^{k_i}\omega(e_{i,j})=u_i-n.
\]
By a
finite succession of operations $A^-$  and applications 
of Lemma
\ref{lem:codegree}, we now construct a floor diagram
$\D'$ with $k-1$ minimal floors and satisfying
(see Figure \ref{fig:op lower}b)
\[
\codeg(\D)=\codeg(\D')+ \Card(E^{-\infty}(\D))  - n(k-1).
\]
Now the result follows by induction on $k$.
\end{proof}

\subsection{Degree of codegree coefficients}
Here we prove a couple of intermediate results regarding the degree of  codegree $i$
coefficients of some families of Laurent polynomials.
Given two integers $k,l\ge 0$, we define 
\[
F(k,l)=\sum_{\substack{i_1+i_2+\cdots+ i_k=l\\ i_1,\cdots,i_k\ge
    1}}\quad \prod_{j=1}^k i_j
\qquad\mbox{and}\qquad \Phi_l(k)=F(k,k+l).
\]

\begin{exa}\label{ex:G0}
  One computes easily that
  \[
  \Phi_0(k)=1
\qquad\mbox{and}\qquad \Phi_1(k)=2k.
  \]
\end{exa}

\begin{lemma}
  For any fixed $l\in\Z_{\ge 0}$, the function
  $\Phi_l:k\in\Z_{\ge 0}\mapsto \Phi_l(k)$ is polynomial of degree $l$.
\end{lemma}
\begin{proof}
  The proof goes by induction on $l$.
The case $l=0$ is covered by Example \ref{ex:G0}.
Now suppose that $l\ge 1$ and that
the lemma holds up to $l-1$. For $l\ge k$, one has
\begin{align*}
  F(k,l)&=\sum_{i_1=1}^{l-k+1} i_1
  \sum_{\substack{i_2+\cdots+ i_k=l-i_1\\ i_2,\cdots,i_k\ge
      1}}\quad \prod_{j=2}^k i_j
  \\ &= \sum_{i_1=1}^{l-k+1} i_1 \ F(k-1,l-i_1),
\end{align*}
and so
\begin{align*}
  \Phi_l(k)&= F(k,k+l)
  \\ &= \sum_{i_1=1}^{l+1} i_1 \ F(k-1,k+l-i_1)
  \\ &= \sum_{i_1=1}^{l+1} i_1 \ \Phi_{l-i_1+1}(k-1)
  \\ &= \Phi_l(k-1)+ \sum_{i_1=2}^{l+1} i_1 \ \Phi_{l-i_1+1}(k-1).
\end{align*}
By induction on $l$, the function $P_l:k\mapsto \Phi_l(k)-\Phi_l(k-1)$ is
then polynomial of degree $l-1$. Since $\Phi_l(0)=F(0,l)=0$, one has
\begin{align*}
  \Phi_l(k)&=\sum_{i=0}^{k-1} \left( \Phi_l(k-i)-\Phi_l(k-(i+1)) \right)
  \\ &= \sum_{i=1}^{k} P_l(i).
\end{align*}
By Faulhaber's formula, the function $\Phi_l(k)$ is polynomial of degree
$l$, and the proof is complete.
\end{proof}

The next corollaries  constitute key steps in our polynomiality proofs.
Recall that the notation  $\coef[i]{P}$ denotes the
the coefficient of codegree $i$ of a Laurent polynomial $P(q)$.
\begin{cor}\label{cor:key}
  Let $i,k\ge 0$ and $a_1,\cdots,a_k> i$ be integers. Then one has
  \[
  \coef[i]{\prod_{j=1}^{k}[a_j]^2}=\Phi_i(k).
    \]
    In particular, the function
    $(k,a_1,\cdots,a_k)\mapsto \coef[i]{\prod_{j=1}^{k}[a_j]^2}$
    only depends on $k$ on the set
    $\{a_1>i,\cdots,a_k> i\}$,
    and is polynomial of degree $i$.  
\end{cor}
\begin{proof}
  Since $\coef[i]{[a]^2}=i+1$ if $a>i$, one has
  \begin{align*}
    \coef[i]{\prod_{j=1}^{k}[a_j]^2}&=
    \sum_{\underset{i_1,\dots,i_k\geq 0}{i_1+i_2+\dots+i_k=i}}
\quad \prod_{j=1}^{k}\coef[i_j]{ [a_j]^2}
\\ &=\sum_{\underset{i_1,\dots,i_k\geq 0}{i_1+i_2+\dots+i_k=i}}
\quad \prod_{j=1}^{k}(i_j+1)
 \\ &=\sum_{\underset{i_1,\dots,i_k\geq 1}{i_1+i_2+\dots+i_k=i+k}}
 \quad \prod_{j=1}^{k}i_j
 \\ &= \Phi_i(k),
  \end{align*}
  as announced.
\end{proof}

\begin{cor}\label{cor:key2}
  Let $P(q)$ be a Laurent polynomial, and 
  $i\ge 0$ an integer. Then the function
    $(k,a_1,\cdots,a_k)\mapsto \coef[i]{P(q)\times \prod_{j=1}^{k}[a_j]^2}$
    only depends on $k$ on the set
    $\{a_1>i,\cdots,a_k> i\}$, and is polynomial of degree $i$.  
\end{cor}
\begin{proof}
  One has
  \[
  \coef[i]{P(q)\times \prod_{j=1}^{k}[a_j]^2}=
  \sum_{\underset{i_1,i_2\geq 0}{i_1+i_2=i}} \coef[i_1]{P(q)} \times\coef[i_2]{\prod_{j=1}^{k}[a_j]^2}.
\quad
  \]
 The statement now follows from Corollary \ref{cor:key}. 
\end{proof}

\section{The genus 0 case}\label{sec:g=0}

\subsection{Proof of Theorem \ref{thm:main1}}\label{sec:g=0 gen}
The main step  is Lemma \ref{lem:d rat} below. 
It can be summarized as follows:
for $(a,b,n,s)$ satisfying the condition from 
Theorem \ref{thm:main1}, all floor diagrams of codegree at most $i$
can easily be described. Then Theorem \ref{thm:main1} simply follows from an
explicit computation of the multiplicity and the number of markings of
each such floor diagram.

Given $i\in\Z_{\ge 0}$, and
$(u,\widetilde u)\in \Z_{\ge 0}^i\times\Z_{\ge0}^i$, we define
  \[
  \codeg(u,\widetilde u)=\sum_{j=1}^{i}j\ (u_j +\widetilde u_j),
  \]
  and we consider the finite set
  \[
  C_i=\left\{(u,\widetilde u)\in \Z_{\ge 0}^i\times\Z_{\ge 0}^i\ | \
  \codeg(u,\widetilde u)\le i   \right\}.
  \]

  For $(u,\widetilde u)\in C_i$, and integers   $b,n\ge 0$, and $a>i$, we denote by
    $\D_{a,b,n,u,\widetilde u}$ the floor diagram of genus 0 and
  Newton polygon $\Delta_{a,b,n}$  depicted in Figure
  \ref{fig:d rat} (we do not specify the weight on elevators in
  $E^0(\D_{a,b,n,u,\widetilde u})$ there since they can be recovered from
  $a,b,n,u$, and $\widetilde u$).
\begin{figure}[h]
\begin{center}
  \begin{tabular}{c}
    \includegraphics[height=12cm, angle=0]{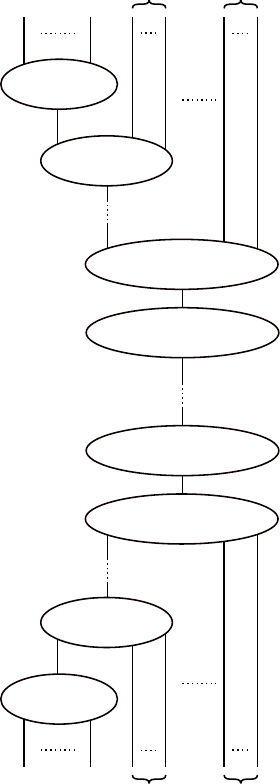}
     \put(-100,35){$v_1$}
   \put(-80,68){$v_2$}
     \put(-50,113){$v_{i+1}$}
     \put(-100,302){$v_a$}
   \put(-83,270){$v_{a-1}$}
     \put(-55,225){$v_{a-i}$}
 \put(-62,-10){$u_1$}
  \put(-20,-10){$u_{i}$}
   \put(-62,348){$\widetilde u_1$}
   \put(-20,348){$\widetilde u_{i}$}
   \end{tabular}
\end{center}
\caption{The floor diagram $\D_{a,b,n,u,\widetilde u}$}
\label{fig:d rat}
\end{figure}
In particular the partial ordering $\o$ on $\D_{a,b,n,u,\widetilde u}$ induces a total
  ordering on its floors
  \[
  v_1\prec \cdots\prec v_a.
  \]
 Note that $\widetilde u_k=0$ (resp. $u_k= 0$) for $k>i-j$ as soon as
  $u_j\ne 0$ (resp. $\widetilde u_j\ne 0$).

  \begin{lemma}\label{lem:d rat}
  Let $i,n\in\Z_{\ge 0}$, and let
  $\D$ be a floor diagram of genus 0 with Newton polygon
  $\Delta_{a,b,n}$ with $a,b,$ and $i$ satisfying 
\[
\left\{\begin{array}{l}
b> i
\\ a>i
\end{array}
\right. .
\]
Then one has
\[
\codeg(\D)\le i \Longleftrightarrow \exists (u,\widetilde u)\in C_i,\ \D=\D_{a,b,n,u,\widetilde u}.
\]
Furthermore in this case, any elevator $e\in E^0(\D)$ satisfies
$\omega(e)>i-\codeg(\D)$.
\end{lemma}
  \begin{proof}
 Given $(u,\widetilde u)\in C_i$,  one has
$\codeg(\D_{a,b,n,u,\widetilde u})=\sum_{j=1}^{i}j\ (u_j +\widetilde u_j)$
 by  a
finite succession of operations $A^\pm$  and applications 
of Lemma
\ref{lem:codegree}.

 Let $\D$ be of codegree at most $i$, and   
 suppose that the order $\o$ is not total on the set of floors of~$\D$. Since~$\D$ 
 is a tree, this is equivalent to say that there
 exist at least two minimal or two maximal floors for $\o$. Denote by
 $k_t$ and $k_b$ the number of  maximal and minimal floors of $\D$,
 respectively.

 By Lemma \ref{lem:min} applied to the orthogonal symmetric of the
 polygon $\Delta_{a,b,n}$ with respect to the $x$-axis, 
 one has
 \[
 \codeg(\D)\ge (k_t-1)\ \left(b +n\frac{k_t}{2}\right).
 \]
 Hence $k_t\ge 2$ implies that
  \[
 \codeg(\D)\ge  b +n >i,
 \]
 contrary to our assumption.

 Analogously, by Lemma \ref{lem:min}, one has that  
 \[
 \codeg(\D)\ge (k_b-1)\ \left( \left(a  -\frac{k_b}{2}\right)n +b\right).
 \]
 Since $k_b\le a-1$, one deduces that $a  -\frac{k_b}{2}\ge 1$.
 Hence $k_b\ge 2$ implies that
  \[
 \codeg(\D)\ge  b +n >i,
 \]
 contrary to our assumption.
 Hence we proved that the order $\o$ is total on the set of floors of $\D$.

 Denoting by $u_j$ (resp. $\widetilde u_j$) the number of elevators in
 $E^{-\infty}(\D)$ (resp. $E^{+\infty}(\D)$) adjacent to the floor
 $v_{j+1}$ (resp. $v_{a-j}$), we then have 
$\D=\D_{a,b,n,u,\widetilde u}$.
Since
 \[
\codeg(\D)= \sum_{j=1}^{a-1} j(u_j+\widetilde u_j),
\]
we deduce that $(u,\widetilde u)\in C_i$.

 To end the proof of the lemma, just note that the elevator in
 $E^0(\D)$ with the lowest weight is either one of the elevators adjacent to 
 the floors $v_k$ and $v_{k+1}$, with $1\le k\le i$,  or the highest
 one for $\o$. The former has weight at least
 \[
 (a-k)n+b-\sum_{j=k}^iu_j \ge b-\codeg(\D)>i-\codeg(\D),
 \]
 while the latter has weight at least
$n+b-\sum_{j=1}^i\widetilde u_j>i-\codeg(\D)$.
\end{proof}

  Let us now count the number of markings of the floor diagram
  $\D_{a,b,n,u,\widetilde u}$. 
 Given $(u,\widetilde u)\in C_i$, we define the functions
  \[
  \widetilde \nu_{u}(a,b,n,s)=
  \sum_{s_0+s_1+\cdots+s_i=s}
\frac{s!}{s_0! s_1! \cdots s_i!}
  \prod_{j=1}^i
       {{an+b+2j-2s_0-2s_1-\cdots - 2s_j-u_{j+1}-\cdots-u_i}\choose{u_j-2s_j}}
       \]
       and
       \[
       \nu_{u,\widetilde u}(a,b,n,s)=
       \widetilde \nu_{u}(a,b,n,s)\times   
        \widetilde \nu_{\widetilde  u}(0,b,0,0). 
      \]
       
\begin{lemma}\label{lem:marking g=0}
  If   $(u,\widetilde u)\in C_i$ and $(a,b,n,s)$ is an element of the
  subset of $\Z_{\ge 0}^4$ defined by
  \[
  \left\{
  \begin{array}{l}
    b\ge i
    \\ an+b\ge i+2s
  \end{array}
  \right.,
  \]
then $\nu_{u,\widetilde u}(a,b,n,s)$ is the
  number of markings of the floor diagram $\D_{a,b,n,u,\widetilde u}$ 
  that are compatible with the pairing
   $\{\{1,2\},\{3,4\},\cdots,\{2s-1,2s\}\}$.
Furthermore the function
$(a,b,n,s)\mapsto \nu_{u,\widetilde u}(a,b,n,s)$ is polynomial
on
this set, and has degree at  most $ \sum_{j=1}^i(u_j+\widetilde u_j)$ in each variable.
If $(u,\widetilde u)=((i,0,\cdots,0),0)$, then the degree in each variable is
exactly $i$.
\end{lemma}
\begin{proof}
 Recall that  $u_j\ne 0$ (resp. $\widetilde u_j\ne 0$) implies that
 $\widetilde u_k=0$ (resp. $u_k= 0$) for $k>i-j$.
 Next,
 if  $an+b\ge i+2s$, then any marking $m$ of
 $\D_{a,b,n,u,\widetilde u}$ satisfies
 $m(j)\in E^{-\infty}(\D_{a,b,n,u,\widetilde u})$
  if $j\le 2s$. 
  From these two observations, it is straightforward to
  compute  the
  number of markings of $\D_{a,b,n,u,\widetilde u}$ 
  compatible with $\{\{1,2\},\cdots,\{2s-1,2s\}\}$. This proves the first assertion
  of the lemma.

  To prove the second assertion, notice that
  the number of possible values of $s_1,\cdots,s_i$ giving rise to a
  non-zero summand of $\widetilde \nu_{u}(a,b,n,s)$
  is finite and only depends on the vector $u$. Hence this assertion
  follows from the fact that,
  for such a fixed choice of $s_1,\cdots,s_i$,
  the function
  \[
  (a,b,n,s)\longmapsto
  \frac{s!}{s_0! s_1! \cdots s_i!}
  \prod_{j=1}^i
       {{an+b+2j-2s_0-2s_1-\cdots - 2s_j-u_{j+1}-\cdots-u_i}\choose{u_j-2s_j}}
  \]
  is polynomial as soon as $an+b\ge i+2s$,   of degree at most
  $ \sum_{j=1}^i(u_j-2s_j)$ in the variables $a,b,$ and $n$, and of degree
at most $ \sum_{j=1}^i(u_j-s_j)$ in the variable $s$.
The third assertion also follows from this computation.
\end{proof}
  
\begin{thm}\label{thm:main1 expl g0}
  For any $i\in\Z_{\ge 0}$, and any $(a,b,n,s)$ in the set
  $\mathcal U_{i}\subset \Z_{\ge 0}^4$ 
 defined by
\[
\left\{\begin{array}{l}
an+b\ge i+2s
\\ b>i
\\ a>i
\end{array}
\right. ,
\]
one has 
\[
\coef[i]{G_{\Delta_{a,b,n}}(0;s)}= \sum_{(u,\widetilde u)\in C_i}
        \nu_{u,\widetilde u}(a,b,n,s)
       \times   \Phi_{i-\codeg(u,\widetilde u)}(a-1).
\]
In particular, 
  the function
\[
\begin{array}{ccc}
\mathcal U_{i}&\longrightarrow & \Z_{\ge 0}
\\ (a,b,n,s)& \longmapsto & \coef[i]{G_{\Delta_{a,b,n}}(0;s)}
\end{array}
\]
is polynomial 
of degree $i$ in each variable.
\end{thm}
\begin{proof}
Let $(a,b,n,s)\in \mathcal U_{i}$. Since  $an+b\ge i+2s$, any marking $m$ of
 $\D_{a,b,n,u,\widetilde u}$ satisfies
 $m(j)\in E^{-\infty}(\D_{a,b,n,u,\widetilde u})$
if $j\le 2s$. In particular, one has
\[
\mu_{\{\{1,2\},\cdots,\{2s-1,2s\}\}}(\D_{a,b,n,u,\widetilde u},m)=\mu(\D_{a,b,n,u,\widetilde u})
\]
for any marking $m$ of $\D_{a,b,n,u,\widetilde u}$ compatible with
the pairing $\{\{1,2\},\{3,4\},\cdots,\{2s-1,2s\}\}$.
Lemma \ref{lem:d rat} and Corollary \ref{cor:key} give
  \[
  \coef[i-\codeg(u,\widetilde u)]{\mu(\D_{a,b,n,u,\widetilde u})}=
  \Phi_{i-\codeg(u,\widetilde u)}(a-1).
  \]
  By Lemma \ref{lem:marking g=0},  one has then
  \begin{align*}
    \coef[i]{G_{\Delta_{a,b,n}}(0;s)}&=
    \coef[i]{ \sum_{(u,\widetilde u)\in C_i}
      \nu_{u,\widetilde u}(a,b,n,s)
      \times \mu(\D_{a,b,n,u,\widetilde u}) }
    \\ &=
     \sum_{(u,\widetilde u)\in C_i}
       \nu_{u,\widetilde u}(a,b,n,s)
       \times \coef[i-\codeg(u,\widetilde u)]{
              \mu(\D_{a,b,n,u,\widetilde u})} 
       \\ &= \sum_{(u,\widetilde u)\in C_i}
        \nu_{u,\widetilde u}(a,b,n,s)
       \times   \Phi_{i-\codeg(u,\widetilde u)}(a-1).
  \end{align*}
  Hence Corollary \ref{cor:key} and Lemma \ref{lem:marking g=0}
  imply that the function
  $(a,b,n,s)\in\mathcal U_{i} \mapsto\coef[i]{G_{\Delta_{a,b,n}}(0;s)}$ is
  polynomial. Furthermore, its degree in $b,n$ and $s$ is $i$, since it
  is the maximal degree of a function $\nu_{u,\widetilde u}$.
  The degree in the variable $a$ of
$\nu_{u,\widetilde u}(a,b,n,s)\times   \Phi_{i-\codeg(u,\widetilde u)}(a-1)$
  is at most
 \[
 i-\codeg(u,\widetilde u)+\sum_{j=1}^i(u_j+\widetilde u_j)
 = i-\sum_{j=2}^{i}(j-1)\ (u_j +\widetilde u_j).
 \]
 Hence this degree is at most $i$, with equality if $u=\widetilde u=0$.
\end{proof}

\subsection{$b=0$ and $n$ fixed}\label{sec:g=b=0}
Here we explain how to modify the proof of Theorem \ref{thm:main1 expl g0} in
the case when one wants to fix $b=0$ and $n\ge 1$. This covers in
particular the case of $X_{\Delta_d}=\CP^2$.
The difference with Section \ref{sec:g=0 gen} is that now a floor diagram $\D$
contributing to 
$ \coef[i]{G_{\Delta_{a,0,n}}(0;s)}$ may have several maximal floors for
the order $\o$. Nevertheless for fixed $n$ and $i$, we show that the set of possible
configurations of these multiple maximal floors is finite and does not depend
on $a$.
In order to do so, we introduce the notion of \emph{capping
  tree}.
\begin{defi}\label{def:capping}
  A \emph{capping tree} 
with Newton polygon $\Delta_{a,n}$ is
  a couple
  $\T=(\Gamma, \omega)$ such that
  \begin{enumerate}
  \item $\Gamma$ is a connected weighted oriented tree
    with $a$ floors and with no
    sources nor sinks;
  \item $\Gamma$ has a unique minimal floor $v_1$, and
    $\Gamma\setminus \{v_1\}$ is not connected;    
\item   for every floor $v\in V(\Gamma)\setminus \{v_1\}$,
one has $ \dive(v)=n$.
  \end{enumerate}
  The codegree of a capping tree $\T$ with Newton polygon
  $\Delta_{a,n}$ is defined as
  \[
  \codeg(\T)=\frac{(a-1)(na-2)}{2}-\sum_{e\in E(\T)}(\omega(e)-1)
  \]
\end{defi}

\begin{figure}[h]
\begin{center}
  \begin{tabular}{ccc}
    \includegraphics[height=3cm, angle=0]{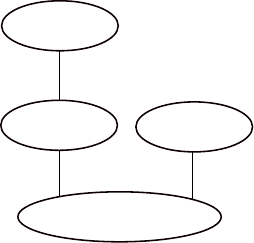}
     \put(-78,22){$2$}
   &\hspace{3cm}
   & \includegraphics[height=3cm, angle=0]{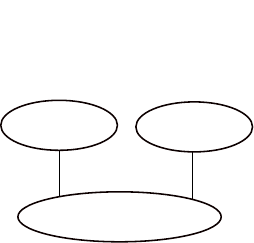}
     \put(-78,22){$2$}
      \put(-15,22){$2$}
 \end{tabular}
\end{center}
\caption{Two examples of capping trees of codegree 2}
\label{fig:ex top}
\end{figure}
\begin{exa}
  Examples of capping trees of codegree 2 and with Newton polygon $\Delta_{4,1}$ and
  $\Delta_{3,2}$  are depicted in Figure \ref{fig:ex top}. We use the
  same convention to depict capping trees as to depict floor diagrams.
\end{exa}

\begin{lemma}\label{lem:codeg capping}
  A capping tree with  Newton polygon
  $\Delta_{a,n}$ has codegree at least $n(a-2)$.  
\end{lemma}
\begin{proof}
  Let $\T$ be such a capping tree, and denote by
  $\omega_1,\cdots,\omega_k$ the weight of the elevators of $\T$ adjacent
  to $v_1$, and by $a_1,\cdots,a_k$ the number of floors of the
  corresponding connected component of $\T\setminus\{v_1\}$.
  By  Definition \ref{def:capping},
  one has
 $\omega_j=na_j$.
   By  a
finite succession of operations $A^+$ and applications 
of Lemma
\ref{lem:codegree}, we reduce the proof 
successively to the
case  when
\begin{enumerate}
  \item $\o$ induces a total order on each connected component of
    $\T\setminus\{v_1\}$;
    \item $k=2$.
\end{enumerate}
\begin{figure}[h]
\begin{center}
  \begin{tabular}{ccc}
    \includegraphics[height=4cm, angle=0]{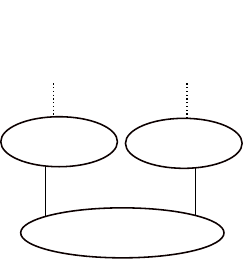}
    \put(-55,10){$v_1$}
        \put(-15,28){$na_1$}
    \put(-108,28){$na_2$}
   &\hspace{1cm}
   & \includegraphics[height=4cm, angle=0]{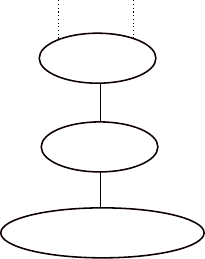}
      \put(-55,10){$v_1$}
   \put(-42,28){$n(a-1)$}
     \put(-42,67){$n(a-2)$}
    \\  $\mathcal T$&&  $\mathcal T'$
 \end{tabular}
\end{center}
\caption{Bounding $\codeg(\mathcal T)$ from below}
\label{fig:capping T'}
\end{figure}
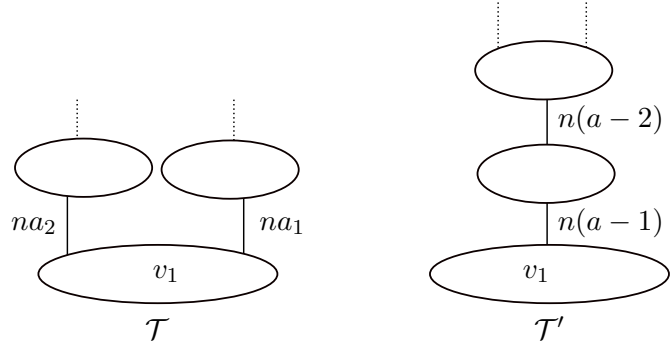
By  two additional  operations $A^+$, we construct a capping tree
$\mathcal T'$ such that (see Figure \ref{fig:capping T'})
\[
\codeg(\mathcal T)\ge \codeg(\mathcal T')+n(a_1+a_2-1)= \codeg(\mathcal T')+n(a-2).
\]
This proves the lemma since $ \codeg(\mathcal T')\ge 0$.
\end{proof}

\begin{proof}[Proof of Theorem \ref{thm:main1 expl g02}]
 Let $\D$ be a floor diagram of genus 0, with Newton polygon
 $\Delta_{a,0,n}$, and of codegree at most $i$.
 Suppose that $\D$ has $k_b\ge 2$ minimal floors for $\o$.
Then exactly as in the proof of Lemma 
 \ref{lem:d rat}, we have that
  \[
  \codeg(\D)\ge  n(k_b-1)\left(a-\frac{k_b}{2}\right)\ge n(a-1)\ge
  n(i+1)>i.
 \]
This contradicts  our assumptions, and  $k_b=1$.

Suppose that $\D$ has at least two maximal floors.
 Denote by $v_o$ the lowest floor
of $\D$ having at least two adjacent outgoing elevators.
Since $k_b=1$, the order $\o$ induces a total ordering on
floors $v$ of $\D$ such that $v\o v_o$. Let $\T$ be the weighted 
subtree of $\D$ 
obtained by removing from $\D$ all elevators and floors strictly below
$v_o$, and denote by $a_o$ the number of floors of $\T$.
Suppose that $\T$ is not a capping tree, i.e. 
$E^{-\infty}(\T)\ne\emptyset$. By a finite succession of $A^-$
operations, we construct a floor diagram $\D'$ with the same floors
as $\D$, the same elevators as well, except for 
elevators in $E^{-\infty}(\T)$,
which become adjacent to $v_o$
in $\D'$.  By Lemma
\ref{lem:codegree}, we have
\[
\codeg(\D)>\codeg(\D').
\]
Let $\T'$ be the capping
tree  
obtained by removing from $\D'$ all elevators and floors strictly below
$v_o$. By Lemma \ref{lem:codeg capping},
 it  
has  codegree at least $n(a_o-2)$.
Since at least one elevator in $E^{-\infty}(\D')$ is adjacent to
$v_o$, we deduce that
\[
\codeg(\D')\ge n(a_o-2) + a-a_o = a +(n-1)a_o -2n.
\]
Since $a_o\ge 3$, we obtain
\[
\codeg(\D')\ge a +n-3\ge i.
\]
As a consequence we get that $\codeg(\D)>i$, 
contrary to our assumption that $\T$ is not a capping tree.

\begin{figure}[h]
\begin{center}
  \begin{tabular}{c}
    \includegraphics[height=9.2cm, angle=0]{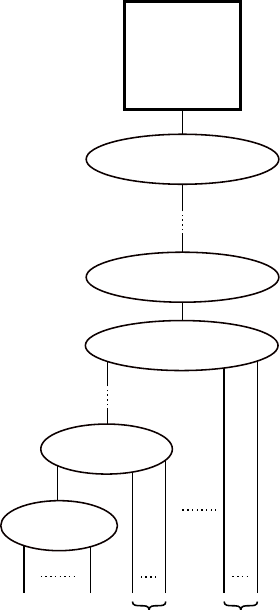}
     \put(-100,35){$v_1$}
   \put(-80,68){$v_2$}
     \put(-50,113){$v_{i+1}$}
   \put(-45,235){$\T$}
 \put(-62,-10){$u_1$}
  \put(-20,-10){$u_{i}$}
   \end{tabular}
\end{center}
\caption{$\codeg(\T)+\sum_{j=1}^i ju_j\le i$.}
\label{fig:d rat2}
\end{figure}
Hence the floor diagram $\D$
either is $\D_{a,0,n,u,0}$, or
looks like the floor diagram 
 $\D_{a,0,n,u,0}$,
except that the top part is replaced by a capping tree of codegree at
most $i$.
In any case $\D$ looks like the floor diagram depicted in Figure
\ref{fig:d rat2} where $\T$ is either a single vertex or a 
 capping tree of codegree at
 most $i$.
 Note that the number of edges $e$ of $\D$ with
 $\omega(e)\le i-\codeg(\D)$, as well as 
 the Laurent polynomial
 \[
 P(q)=\prod_{\substack{e\in E^0(\D)\\ \omega(e)\le i-\codeg(\D)}}[w(e)]^2
 \]
 do not depend on $a$. Indeed, let  $k$ be such that there
 exists $l\ge k$ with $u_l\ne 0$. Denoting by $e$ the elevator $e\in E^0(\D)$ adjacent to the floors $v_k$ and $v_{k+1}$, we have that
\[
\omega(e)=n(a-k)-\sum_{j=k}^i u_j> i-k+1 -\sum_{j=k}^i u_j\ge
i-\sum_{j=k}^i ju_j\ge i-\codeg(\D).
\]
Hence by Corollary \ref{cor:key2}, the coefficient
  $\coef[i-\codeg(\D)]{\mu(\D)}$ is polynomial in $a$ of degree $i-\codeg(\D)$.
Furthermore since $an\ge i+2s$, any increasing bijection
\[
\left\{\eta(\D)-\Card(V(\T)  \cup E(\T))+1  ,\cdots, \eta(\D)\right\}\longrightarrow V(\T)  \cup
E(\T)
\]
extends to exactly $\widetilde \nu_{u}(a,0,n,s)$ markings of $\D$
compatible with $\{\{1,2\},\cdots,\{2s-1,2s\}\}$.

Since there exists finitely many such increasing maps, 
and  
finitely many capping trees of codegree at
most $i$ by Lemma \ref{lem:codeg capping}, the  end of the proof
 is now entirely analogous to  the  proof of
Theorem \ref{thm:main1 expl g0}.
\end{proof}

\subsection{Polynomiality with respect to $s$}
We use a different method to prove polynomiality with respect to $s$
when $\Delta$ is fixed, namely we prove that the
\emph{$i$-th discrete derivative} of the map
$s\mapsto \coef[i]{G_\Delta(0;s)}$ is constant.
Recall that the
$n$-th discrete derivative of a univariate polynomial $P(X)$ is defined by
\[
P^{(n)}(X)=\sum_{l=0}^n(-1)^l{n \choose l}P(X+l).
\]

\begin{lemma}\label{lem:degree derivation}
 One has
  \[
  (P^{(n)})^{(1)}(X)=P^{(n+1)}(X) \qquad
  \mbox{and}\qquad
  \deg P^{(n)}(X)=\deg P(X) -n.
  \]
  Furthermore, if 
   the leading coefficient of $P(X)$ is $a$, then the
  leading coefficient of $P^{(n)}(X)$ is
  \[
  (-1)^n\ a\deg P(X) (\deg P(X)-1)\cdots (\deg P(X)-n+1).
  \]
\end{lemma}
\begin{proof}
  The first assertion is a simple application of Descartes' rule for binomial
  coefficients:
  \begin{align*}
    (P^{(n)})^{(1)}(X) &= P^{(n)}(X)- P^{(n)}(X+1)
    \\ &= \sum_{l=0}^n(-1)^l{n \choose l}P(X+l) -
    \sum_{l=1}^{n+1}(-1)^{l-1}{n \choose l-1}P(X+l)
    \\ &=  \sum_{l=0}^{n+1}(-1)^l\left( {n \choose l} + {n \choose
      l-1} \right) P(X+l)
    \\ &= P^{(n+1)}(X).
  \end{align*}
Hence  the second and third assertions follow by induction starting with
the straightforward 
 case $n=1$. 
\end{proof}

\begin{proof}[Proof of Theorem \ref{thm:poly s}]
Recall that
\[
\eta(\Delta)=\Card(\partial \Delta\cap \Z^2)-1,
\qquad \iota(\Delta)=\Card(\Delta\cap \Z^2)
- \Card(\partial \Delta\cap \Z^2),
\qquad\mbox{and}\qquad
s_{max}=\left[ \frac{\eta(\Delta)}{2}\right].
\]
 We denote by $a_i(X)$ the polynomial of degree  at most $s_{max}$
  that interpolates the values
 \[
 \coef[i]{G_\Delta(0;0)},\cdots,\coef[i]{G_\Delta(0;s_{max})}.
 \]
  By Lemma \ref{lem:degree derivation},
the polynomial $a_i^{(i)}(X)$ has degree at most $s_{max}-i$, and
  we are left to prove that
  \[
  a_i^{(i)}(0)=\cdots
  =a_i^{(i)}(s_{max}-i)=2^i.
\]

Let $s\in\{0,1,\cdots, s_{max}-i\}$, and 
$S$ be a
pairing of order $s$ of the set $\{2i+1,\cdots, \eta(\Delta)\}$.
Given $I\subset \{1,\dots,i\}$, we denote by
$S^I$ 
the pairing
\[
S^I=S\cup \bigcup_{j\in I}\{\{2j-1,2j\}\}.
\]
Given  $(\D,m)$  a marked floor diagram with Newton polygon
$\Delta$ and of genus $0$, we define
\[
\kappa(\D,m)(q)=\sum_{l=0}^{i}\sum_{\tiny{\begin{array}{c}I\subset  \{1,\cdots,i\}\\
|I|=l\end{array}}}  (-1)^l \mu_{S^I}(\D,m)(q).
\]
By Theorem \ref{thm:psi fd}, we have
\begin{align*}
\sum_{j=-\iota(\Delta)}^{\iota(\Delta)} a_{\iota(\Delta)-|j|}^{(i)}(s)q^j & =
\sum_{l=0}^{i}\sum_{\tiny{\begin{array}{c}I\subset  \{1,\cdots,i\}\\
|I|=l\end{array}}}  (-1)^l \sum_{(\D,m)} \mu_{S^I}(\D,m)(q)
\\ & =
\sum_{(\D,m)} \kappa(\D,m)(q),
\end{align*}
where the sum over $(\D,m)$ runs over all isomorphism classes of
marked floor diagrams with Newton polygon
$\Delta$ and of genus $0$.

Let  $(\D,m)$ be one of these marked floor diagrams, and
denote by $i_0$ the minimal element of $\{1,\cdots,\eta(\Delta)\}$
such that 
$m(i_0)\in V(\D)$. We also
denote by
$J\subset\{1,\cdots,2i\}$ the set of
elements $j$ such that $m(j)$ is mapped to an elevator in
$E^{-\infty}(\D)$ adjacent to $m(i_0)$. 

\medskip
\noindent {\bf Step 1.} We claim that if  the set $J\cup\{i_0\}$ contains a
 pair $\{2k-1,2k\}$ with $k\le i$, then $\kappa(\D,m)(q)=0$.

Let $I\subset  \{1,\cdots,i\}\setminus\{k\}$. It follows from 
Definition \ref{def:refined mult s} that
\[
\mu_{S^I}(\D,m)(q)=\mu_{S^{I\cup\{k\}}}(\D,m)(q).
\]
Hence one has
\begin{align*}
\kappa(\D,m)(q)&= \sum_{l=0}^{i}\sum_{\tiny{\begin{array}{c}I\subset  \{1,\cdots,i\}\\
|I|=l\end{array}}}  (-1)^l \mu_{S^I}(\D,m)(q)
\\ &=  \sum_{l=0}^{i-1}
 \sum_{\tiny{\begin{array}{c}I\subset  \{1,\cdots,i\}\setminus\{k\}\\
|I|=l\end{array}}} \left( (-1)^l \mu_{S^I}(\D,m)(q) +
  (-1)^{l+1} \mu_{S^{I\cup \{k\}}}(\D,m)(q) \right)
  \\ &=0,
\end{align*}
and the claim is proved.
We assume from now on
that the set $J\cup\{i_0\}$ contains no
 pair $\{2k-1,2k\}$ with $k\le i$.
 
\medskip
\noindent {\bf Step 2.} 
We first study the case when $2i\le d_b(\Delta)$.

If $i_0\le 2i$, then
$|J|\le i-1$, and no element $k>2i$ is
mapped to an elevator in
$E^{-\infty}(\D)$ adjacent to $m(i_0)$.
The codegree of $(\D,m)$ is then at least
$d_b(\Delta) -|J|\ge d_b(\Delta)-i +1$ by
Lemma  \ref{lem:codegree}, see Figure \ref{fig:ineq codeg}a).
Hence this codegree is  at least
$i+1$ by assumption, which means that $\kappa(\D,m)(q)$ does not
contribute to $a_{i}^{(i)}(s)$. 
\begin{figure}[h]
\begin{center}
\begin{tabular}{ccc}
  \includegraphics[height=2.5cm, angle=0]{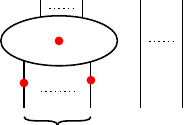}
  \put(-75,-10){$J$}
 & \hspace{2cm} &
  \includegraphics[height=2.5cm, angle=0]{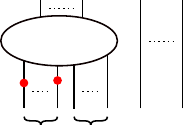}
 \put(-85,-10){$J$}
 \put(-57,-10){$K$}

\\  \\ a) $i_0\le 2i$ &&b) $i_0>2i$
\end{tabular}
\end{center}
\caption{Illustration of Step 2;  red dots represent points in
$m(\{1,\cdots,2i\})$.}
\label{fig:ineq codeg}
\end{figure}

Suppose now that $i_0>2i$, so in particular 
$m(\{1,\cdots, 2i\})\subset E^{-\infty}(\D)$. 
We denote by
$K\subset\{2i+1,\cdots,\eta(\Delta)\}$ the set of
elements $j$ such that $m(j)$ is mapped to an elevator in
$E^{-\infty}(\D)$ adjacent to $m(i_0)$. Note that
$|K|\le d_b(\Delta) -2i$.
Hence  Lemma  \ref{lem:codegree} implies that
 $(\D,m)$ has codegree at least
 \[
 d_b(\Delta)-|J|- |K|\ge  d_b(\Delta) -i-
 |K| =i + (d_b(\Delta) -2i-|K|),
 \]
 see Figure \ref{fig:ineq codeg}b).
 Hence $\kappa(\D,m)(q)$ can contribute
 to $a_{i}^{(i)}(s)$  only if
 $|K|= d_b(\Delta) -2i$.
 It follows from Lemma  \ref{lem:codegree} again that
$\kappa(\D,m)(q)$  contributes
 to $a_{i}^{(i)}(s)$ if and only if 
\begin{itemize}
  \item the order $\preccurlyeq$ is total on the set of floors of $\D$;
  \item elevators  in $E^{+\infty}(\D)$
    are all adjacent to the maximal floor
    of $\D$;
\item 
$m(\{1,\cdots,2i\}\setminus  J)$ consists of elevators in
$E^{-\infty}(\D)$ adjacent to the second lowest floor of $\D$;
\item any elevator in $E^{-\infty}(\D)\setminus m(\{1,\cdots,2i\})$ is
adjacent to $m(i_0)$;
\item  The set $ J$ contains exactly $i$ elements,
and no pair $\{2k-1,2k\}$;
\item the function $l:V(\D)\to d_l(\Delta)$ is decreasing, and the function
  $r:V(\D)\to d_r(\Delta)$ is increasing.

\end{itemize}
For such $(\D,m)$, we have
\[
\kappa(\D,m)(q)=\mu_{S}(\D,m)(q),
\]
since $\mu_{ S^I}(\D,m)(q)=0$ if $I\ne\emptyset$. The coefficient  of codegree~$0$ of
$\mu_{S}(\D,m)(q)$ is 1 by Definition \ref{def:refined mult s}.
The floor diagram 
$\D$ has codegree $i$, and there are
exactly $2^i$ such marked floor diagrams $(\D,m)$, one for each possible set
$J$, so we obtain that $a_i^{(i)}(s)=2^i$ as claimed. 

\medskip
\noindent {\bf Step 3.} 
We assume now that $2i\in\{d_b(\Delta)+1,d_b(\Delta)+2\}$.
In this case we necessarily have $i_0\le 2i$. As in Step 2, we have
$|J|\le i-1$, and the codegree of $(\D,m)$ is  at least
$d_b(\Delta) -|J|\ge d_b(\Delta)-i +1$ by
Lemma  \ref{lem:codegree}. Hence
$\kappa(\D,m)(q)$ can contribute
 to $a_{i}^{(i)}(s)$  only if
one of the following sets of 
conditions is satisfied:
\begin{enumerate}
\item $(\D,m)$ has codegree $i$, with
$2i=d_b(\Delta)+1$ and $|J|= i-1$;
\item $(\D,m)$ has codegree $i-1$, with
$2i=d_b(\Delta)+2$ and $|J|= i-1$;
\item $(\D,m)$ has codegree $i$, with
$2i=d_b(\Delta)+2$ and $|J|= i-1$;
\item $(\D,m)$ has codegree $i$, with
$2i=d_b(\Delta)+2$ and $|J|= i-2$.
\end{enumerate}
We end by studying these cases one by one. Recall that in the last three
cases, we make the additional assumption  that $\Delta=\Delta_{a,b,n}$. In this case, the
conditions $an+b+2=2i$ and $\iota(\Delta)\ge i$ ensure that $n\le i-2$.

\begin{enumerate}
\item $(\D,m)$ has codegree $i$, with
$2i=d_b(\Delta)+1$ and $|J|= i-1$.
As in Step 2, the Laurent polynomial
$\kappa(\D,m)(q)$  contributes
 to $a_{i}^{(i)}(s)$ if and only if (see Figure \ref{fig:ineq codeg2}a):
\begin{itemize}
  \item the order $\preccurlyeq$ is total on the set of floors of $\D$;
  \item elevators  in $E^{+\infty}(\D)$
    are all adjacent to the maximal floor
    of $\D$;
\item 
$m(\{1,\cdots,2i\}\setminus  (J\cup \{i_0\}))$ consists of all elevators in
$E^{-\infty}(\D)$ adjacent to the second lowest floor of $\D$;
\item the function $l:V(\D)\to d_l(\Delta)$ is decreasing, and the function
  $r:V(\D)\to d_r(\Delta)$ is increasing.
\end{itemize}
\begin{figure}[h]
\begin{center}
\begin{tabular}{ccc}
  \includegraphics[height=4cm, angle=0]{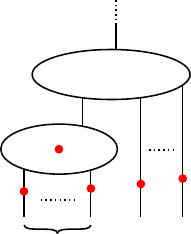}
  \put(-70,-10){$J$}
  \put(-65,57){$w$}
  \put(-50,93){$w'$}
 & \hspace{0cm} &
  \includegraphics[height=4cm, angle=0]{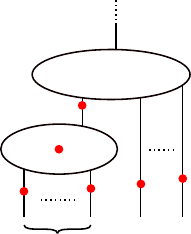}
  \put(-70,-10){$J$}
  \put(-65,57){$w$}
  \put(-50,93){$w'$}

\\  \\ a) $d_b(\Delta)=2i-1$ and $|J|=i-1$ &&b)  $d_b(\Delta)=2i-2$ and
  $|J|=i-1$  or $|J|=i-2$
  
\end{tabular}
\end{center}
\caption{Illustration of Step 3;  red dots represent points in
$m(\{1,\cdots,2i\})$.}
\label{fig:ineq codeg2}
\end{figure}
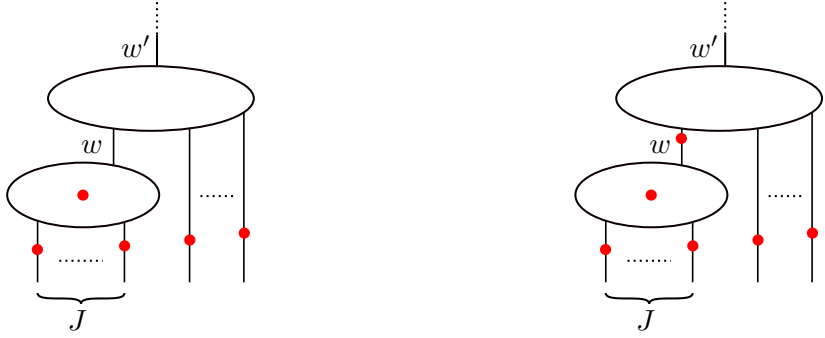
For such $(\D,m)$, we have
\[
\kappa(\D,m)(q)=\mu_{S}(\D,m)(q),
\]
since $\mu_{ S^I}(\D,m)(q)=0$ if $I\ne\emptyset$.
The coefficient of codegree~$0$ of 
$\mu_{ S}(\D,m)(q)$ is 1, and
there are
exactly $2^i$ such marked floor diagrams, one for each possible set
$J\cup\{i_0\}$. We obtain again that $a_i^{(i)}(s)=2^i$. 

\item $(\D,m)$ has codegree $i-1$, with
$2i=d_b(\Delta)+2$ and $|J|= i-1$.
As in Step 2, the Laurent polynomial
$\kappa(\D,m)(q)$  contributes
 to $a_{i}^{(i)}(s)$ if and only if (see Figures \ref{fig:d rat} and \ref{fig:ineq codeg2}b):
\begin{itemize}
 \item $\D=\D_{a,b,n,(i-1),0}$;
    \item $i_0=2i-1$, and $m(2i)$ is the elevator of $\D$ adjacent to and
  oriented away from $m(i_0)$.
\item 
$m(\{1,\cdots,2i-2\}\setminus  J)$ consists of all elevators in
$E^{-\infty}(\D)$ adjacent to the second lowest floor of $\D$.
\end{itemize}
For such $(\D,m)$, we have
\[
\kappa(\D,m)(q)=\mu_{S}(\D,m)(q) - \mu_{S^{\{i\}}}(\D,m)(q) ,
\]
since $\mu_{ S^I}(\D,m)(q)=0$ if $I\not\subset\{i\}$.
We have $[w]^2(q)-[w](q^2)=0$ if $w=1$, and
\[
[w]^2(q)-[w](q^2)=0q^{-w+1} + 2q^{-w+2} +...
\]
if $w\ge 2$.
Since $w=i-1-n$ in Figure \ref{fig:ineq codeg2}b), 
we have by Definition \ref{def:refined mult s} that
the coefficient of codegree~$1$ of 
$\kappa(\D,m)(q)$ is $0$ if $n=i-2$, and is $2$ if $n\le i-3$.
There are
exactly $2^{i-1}$ such marked floor diagrams, one for each possible set
$J$. So the total contribution of such $(\D,m)$ to
$a_i^{(i)}(s)$ is $0$ if $n=i-2$ and is $2\times 2^{i-1}=2^i$ if $n\le i-3$.

\item $(\D,m)$ has codegree $i$, with
$2i=d_b(\Delta)+2$ and $|J|= i-1$.
As in the previous cases $\kappa(\D,m)(q)$ can contribute
 to $a_{i}^{(i)}(s)$  only if $i_0=2i-1$, 
 and $m(2i)$
 and $m(2i-1)$ are not adjacent. This is possible if and
 only if  both 
 $m(2i-1)$ and $m(2i)$ are floors and $n=i-2$, see Figure \ref{fig:ineq codeg3}.
\begin{figure}[h]
\begin{center}
  \includegraphics[height=5cm, angle=0]{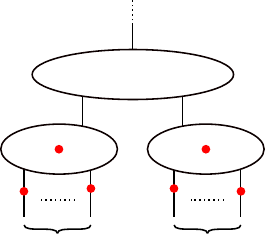}
  \put(-158,-10){$i-1=n+1$}
  \put(-65,-10){$i-1=n+1$}
\end{center}
\caption{Illustration of Step 3;  red dots represent points in
$m(\{1,\cdots,2i\})$.}
\label{fig:ineq codeg3}
\end{figure}
In this case $\kappa(\D,m)(q)=\mu_{S}(\D,m)(q)$, and
the coefficient of codegree~$0$ of 
$\mu_{ S}(\D,m)(q)$ is $1$. There are
exactly $2^i$ such marked floor diagrams, so the total contribution of
such $(\D,m)$ to 
$a_i^{(i)}(s)$ is $2^i$.

\item $(\D,m)$ has codegree $i$, with
$2i=d_b(\Delta)+2$ and $|J|= i-2$.
As in Step 2, the marked floor diagram
$\kappa(\D,m)(q)$  may contribute
 to $a_{i}^{(i)}(s)$  only if (see Figures \ref{fig:d rat} and \ref{fig:ineq codeg2}b):
\begin{itemize}
  \item $\D=\D_{a,b,n,(i),0}$;
    \item $i_0=2i-3$ or $i_0=2i-2$;
\item  $m(2i-1)$ or $m(2i)$ is the elevator of $\D$ adjacent to and
  oriented away from $m(i_0)$.
\item 
$m(\{1,\cdots,2i\}\setminus  (J\cup \{i_0\}))$ consists of all elevators
adjacent to and oriented toward  the second lowest floor of $\D$.
\end{itemize}
For such $(\D,m)$, we have
\[
\kappa(\D,m)(q)=\mu_{S}(\D,m)(q) - \mu_{S^{\{i\}}}(\D,m)(q) ,
\]
since $\mu_{ S^I}(\D,m)(q)=0$ if $I\not\subset\{i\}$.
We have 
\[
[w]^2(q)-\frac{[w][w+1]}{[2]}(q)=0q^{-w+1} + ...,
\]
so by Definition \ref{def:refined mult s}
the coefficient of codegree~$0$ of 
$\mu_{ S}(\D,m)(q)$ is $0$. Hence the total contribution of
such $(\D,m)$ to
$a_i^{(i)}(s)$ is $0$.
\end{enumerate}
Summing up all contributions, we obtain that $a_i^{(i)}(s)=2^i$ as announced.
\end{proof}


\section{Higher genus case}\label{sec:gen}
The generalization of Theorems \ref{thm:main1} and
\ref{thm:main1 expl g02} to higher genus is quite technical and requires some care.
Following \cite{FM} and \cite{ArdBlo}, we prove Theorems \ref{thm:maing},
\ref{thm:maing n=0}, and \ref{thm:maing2}
by decomposing floor diagrams into elementary building blocks that
we call \emph{templates}. Although  templates from this paper 
differ from those from \cite{FM} and \cite{ArdBlo}, we 
borrow their terminology since  we follow the overall strategy exposed
in \cite{FM}.

\subsection{Templates}

Recall that the orientation of an oriented acyclic graph $\Gamma$
induces a partial ordering $\o$
on $\Gamma$.
Such an oriented graph $\Gamma$ is said to be \emph{layered} if $\o$
induces a total order on vertices of $\Gamma$.
A layered graph $\Gamma$ is necessarily connected. 
We say that an edge $e$ of  $\Gamma$
is \emph{separating} if $\Gamma\setminus\{e\}$ is disconnected, and if
$e$ is comparable with any element of $\Gamma\setminus\{e\}$.
A \emph{short edge} of $\Gamma$ is an edge 
connecting two consecutive vertices of $\Gamma$, and we denote by
$E^c(\Gamma)$ the set of short edges of $\Gamma$. 
\begin{defi}\label{defi:template}
  A \emph{pre-template} is a couple~$\Theta=(\Gamma,\omega)$ such that
\begin{enumerate}
\item $\Gamma$ is a layered acyclic oriented graph with no separating edge;
\item
  $\omega$ is a weight function 
  $E(\Gamma)\setminus E^c(\Gamma)\to \Z_{>0}$;
\item every edge in $E^{\pm\infty}(\Gamma)$ has weight 1.
\end{enumerate}

One says that  $\Theta=(\Gamma,\omega)$ is a \emph{template} if it
satisfies the additional condition:
\begin{itemize}
\item[(4)] $E^{+\infty}(\Gamma)=\emptyset$ or
  $E^{-\infty}(\Gamma)=\emptyset$.
\end{itemize}
\end{defi}

Similarly to floor diagrams, we will not distinguish between a
pre-template $\Theta$ and its underlying graph, and the \emph{genus} of
$\Theta$ is defined to be its
first Betti number.
A template $\Theta$ which is not reduced to a vertex and for which
$E^{\pm\infty}(\Theta)=\emptyset$ is called \emph{closed}.
Denoting by
$v_1\prec v_2 \prec \cdots \prec v_{l(\Theta)}$
the vertices of~$\Theta$, we define $c(e)$ for a non-short edge $e$ by
\begin{itemize}
\item $c(e)=j-1$ if $e\in E^{-\infty}(\Theta)$ is adjacent to $v_j$;
\item $c(e)=j$ if $e\in E^{+\infty}(\Theta)$ is adjacent to $v_{l(\Theta)-j}$;
\item $c(e)= (k-j-1)\ \omega(e)$ if
  $e\in E^0(\Theta)\setminus  E^c(\Theta)$
    is adjacent to $v_j$ and $v_k$ with $v_j\o v_k$.
\end{itemize}
Finally, we defined the \emph{codegree} of $\Theta$ by
\[
\codeg(\Theta)=\sum_{e\in E(\Gamma)\setminus  E^c(\Gamma)} c(e).
\]
The integer $l(\Theta)$ is called the \emph{length} of $\Theta$.

\begin{exa}
  We depict in Figure \ref{fig:ex temp} all templates of genus at most 1 and
  codegree at most 2. Note that for a fixed $g$ and $i$, there are
  finitely many templates of genus $g$ and codegree $i$.
\begin{figure}[h]
\begin{center}
\begin{tabular}{|l|c|c|c|c|c|c|c|c|c|c|}
 \cline{2-11} \multicolumn{1}{c|}{} & \includegraphics[width=0.75cm]{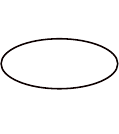}&
\includegraphics[width=1cm]{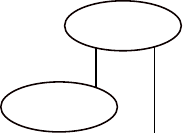}&
  \includegraphics[width=1cm, origin=c,  angle=180]{Figures1/T4.pdf}&
  \includegraphics[width=1cm]{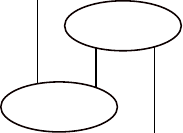}&
  \includegraphics[width=1cm]{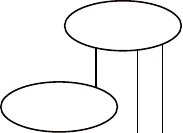}&
  \includegraphics[width=1cm, origin=c,  angle=180]{Figures1/T5.pdf}&
  \includegraphics[width=1cm]{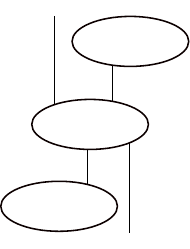}&
   \includegraphics[width=1cm]{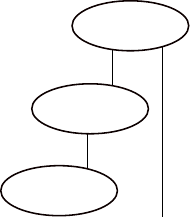}&
  \includegraphics[width=1cm, origin=c,  angle=180]{Figures1/T6.pdf}&
   \includegraphics[width=0.75cm]{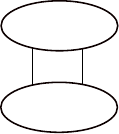}
\\ \hline
  \mbox{genus}& $0$ & $0$ & $0$& 0  & $0$ & $0$ & $0$ &  0&0 &1
  \\ \hline \mbox{codegree}&$0$ &$1$&$1$ & 2&$2$&$2$ & $2$ &2&2&0
   \\ \hline \mbox{length}&$1$ &$2$&$2$& 2 &$2$&$2$ & $3$ &3&3 &2
 \\\hline
\end{tabular}
\\ $  $ \\$  $ \\
\begin{tabular}{|l|c|c|c|c|c|c|c|c|c|c|c|}
 \cline{2-11} \multicolumn{1}{c|}{} &
  \includegraphics[width=1cm]{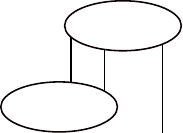}&
  \includegraphics[width=1cm, origin=c,  angle=180]{Figures1/T7.pdf}&
  \includegraphics[width=1cm]{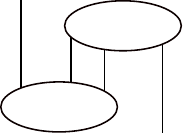}&
  \includegraphics[width=1cm]{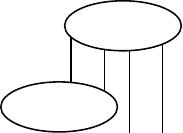}&
  \includegraphics[width=1cm, origin=c,  angle=180]{Figures1/T9.pdf}&
   \includegraphics[width=1cm]{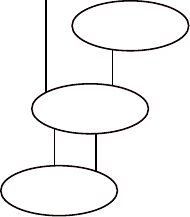}&
  \includegraphics[width=1cm, origin=c,  angle=180]{Figures1/T10b3.pdf}&
 \includegraphics[width=1cm]{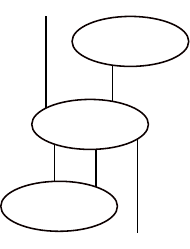}&
 \includegraphics[width=1cm, origin=c,  angle=180]{Figures1/T10b4.pdf}&
  \includegraphics[width=1cm]{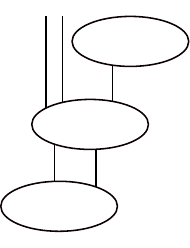}
 \\ \hline
  \mbox{genus} & $1$  & $1$&   $1$  & $1$  & $1$  &1&1&
  1 &1 &1 
  \\ \hline \mbox{codegree}&  1 &1 &2 & $2$ & $2$ & 1 & 1 &
  2&2 &2
   \\ \hline \mbox{length} & 2 &
  2&$2$ & $2$ & 2 & 3 &
  3 & 3& 3&3
 \\\hline
\end{tabular}
\\ $  $ \\$  $ \\
\begin{tabular}{|l|c|c|c|c|c|c|c|c|c|c|} 
 \cline{2-11} \multicolumn{1}{c|}{}&
 \includegraphics[width=1cm, origin=c,  angle=180]{Figures1/T10b5.pdf}&
  \includegraphics[width=1cm]{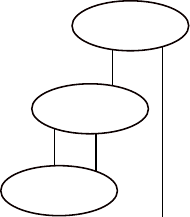}&
  \includegraphics[width=1cm, origin=c,  angle=180]{Figures1/T10.pdf}&
   \includegraphics[width=1cm, origin=c,  angle=180]{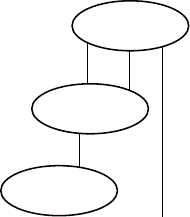}&
 \includegraphics[width=1cm]{Figures1/T11.pdf}&
 \hspace{0.5ex} \includegraphics[width=0.9cm]{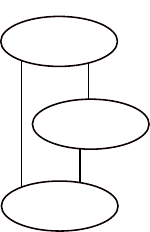}
  \put(-30,15){2}&
 \includegraphics[width=1cm]{Figures1/T8.pdf}&
  \includegraphics[width=0.9cm]{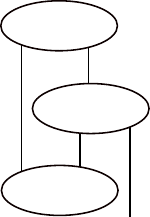}&
  \includegraphics[width=0.9cm, origin=c,  angle=180]{Figures1/T13.pdf}&
  \includegraphics[width=1cm]{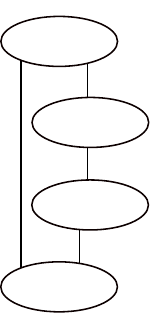}
  \\ \hline
  \mbox{genus}&1&  $1$ & $1$ & $1$ & $1$ &  $1$ & $1$ & $1$ & $1$ & 1
  \\ \hline \mbox{codegree}&2& $2$ &$2$& 2 & 2 & 2 & 1& 2 & 2 &2 
    \\ \hline \mbox{length}&$3$ &$3$&$3$&$3$&$3$ & $3$ & $3$ & 3 & 3 &
   4 
 \\\hline
\end{tabular}
\end{center}
\caption{Templates of genus at most 1 and codegree at most 2.}
\label{fig:ex temp}
\end{figure}
\end{exa}

\begin{lemma}\label{lem:temp c+g}
  Any pre-template $\Theta$ satisfies
  \[
  \codeg(\Theta)+g(\Theta)\ge l(\Theta)-1.
  \]
\end{lemma}
\begin{proof}
  The proof goes by induction on $\codeg(\Theta)$.
The lemma holds if $\codeg(\Theta)=0$, since any two consecutive
vertices of $\Theta$ are connected by at least two edges.
If $\codeg(\Theta)>0$, then an operation $A^{\pm}$ produces a graph
$\Theta'$ with
\[
l(\Theta')=l(\Theta),
\qquad\qquad
g(\Theta')=g(\Theta),
\qquad\mbox{and}\qquad
\codeg(\Theta')\le \codeg(\Theta)-1.
\]
There are now two cases: either $\Theta'$ is a template, or it
contains a separating edge.
In the former case, the lemma holds by induction.
In the latter case, denote by $e$ the separating edge of
$\Theta'$, and $\Theta_1'$ and $\Theta_2'$ the two connected
components of $\Theta'\setminus\{e\}$. Both $\Theta_1'$ and
$\Theta_2'$ are templates, and one has
\[
l(\Theta'_1)+l(\Theta'_2)=l(\Theta),
\qquad
\codeg(\Theta'_1)+\codeg(\Theta'_2)\le \codeg(\Theta)-1,
\qquad\mbox{and}\qquad
g(\Theta'_1)+g(\Theta'_2)=g(\Theta).
\]
Hence the lemma holds by induction again.
\end{proof}

Given a layered floor diagram $\D=(\Gamma,\omega)$,
we denote by $E^u(\D)$ the union  of
\begin{itemize}
\item
  the set of separating edges $e$ of
$\D$,
  \item the set of  edges in $E^{-\infty}(\Gamma)$ and $E^{+\infty}(\Gamma)$
    adjacent to the minimal and maximal floor of $\D$, respectively,
\end{itemize}
and we denote by
$\D_1,\cdots,\D_l$ the 
 connected components   of
 $\D\setminus E^u(\D)$ that are not reduced to a non-extremal vertex.
 Each $\D_j$  equipped with the 
 the weight function $\omega|_{E(\D_j)\setminus E^c(\D_j)}$ is a
 pre-template.
 If $\D_1$ is not a template, then necessarily 
 $E^u(\D)\subset E^{-\infty}(\Gamma)\cup E^{+\infty}(\Gamma)$ and
 $\D_1=\D\setminus E^u(\D)$.

 \begin{defi}\label{def:st layered}
A layered floor diagram $\D=(\Gamma,\omega)$ is said to be
\emph{strongly layered} if each $\D_j$ 
equipped with the 
 the weight function $\omega|_{E(\D_j)\setminus E^c(\D_j)}$  is  a template.
\end{defi}
 
\medskip
Now we explain how to reverse this decomposing process.
A collection of templates  $\Xi=(\Theta_1,\cdots,\Theta_m)$
is said to be \emph{admissible} if 
 $E^{+\infty}(\Theta_1)=E^{-\infty}(\Theta_m)=\emptyset$, and
$\Theta_2,\cdots,\Theta_{m-1}$ are closed.
Given $a\in\Z_{>0}$, we denote by $A_{a}(\Xi)$ the set of sequences of  positive
  integers
  $\kappa=(k_1=1,k_2,\cdots,k_m)$ such that
  \begin{itemize}
  \item  $\forall j\in\{1,\cdots,m-1\},\ k_{j+1}\ge k_j+l(\Theta_j)$;
    \item  $k_m+l(\Theta_m)= a+1$.
  \end{itemize}
Given $\kappa\in A_{a}(\Xi)$ and additional  integers $n\ge 0$ and 
$b\ge\Card(E^{+\infty}(\Theta_m))$, we
denote by
$B_{a,b,n}(\Xi,\kappa)$ the set of collections
$\Omega=(\omega_1,\cdots,\omega_m)$ where
$\omega_j:E(\Theta_j)\to \Z_{>0}$ is  a weight function extending
$\omega_j:E(\Theta)\setminus E^c(\Theta_j)\to \Z_{>0}$ 
  such that
  \begin{itemize}
  \item $\dive(v)=n$ for any non-extremal vertex $v$ of $\Theta_j$;
  \item $\dive(v)= -\left((a-k_j)n+b -\Card(E^{+\infty}(\Theta_j)\right)$ if 
    $v$ is the
    minimal vertex of $\Theta_j$, when $\Theta_j$ is not reduced to $v$.
  \end{itemize}
Note that by definition $\Theta_j$ may be reduced to $v$ only if $j=1$
or $j=m$.
  We denote by
  \[
  \omega_{\Xi,\Omega}: \bigsqcup_{j=1}^m 
  \Theta_j
   \longrightarrow \Z_{>0}
  \]
  the weight function whose restriction to $\Theta_j$ is $\omega_j$.

 Given three integers $a,b,n\ge 0$, an admissible  collection of
 templates  $\Xi=(\Theta_1,\cdots,\Theta_m)$, and two elements
 $\kappa\in  A_{a}(\Xi)$ and $\Omega\in B_{a,b,n}(\Xi,\kappa)$, we
 construct a strongly
 layered floor diagram $\D$ with Newton polygon $\Delta_{a,b,n}$
 as follows:
 \begin{enumerate}
 \item for each $j\in\{1,\cdots,m-1\}$, connect the maximal vertex of
   $\Theta_j$ to the minimal vertex of $\Theta_{j+1}$ by a chain of
   $k_{j+1}-k_j-l(\Theta_j)+1$ edges, oriented from   $\Theta_j$ to $\Theta_{j+1}$;  denote by $\widetilde\Gamma_{\Xi,\kappa}$ the resulting graph;
 \item extend the weight function
   $\omega_{\Xi,\Omega}$ to $\widetilde\Gamma_{\Xi,\kappa}$ such that
   each non-extremal vertex has divergence $n$; this extended
function is still denoted by   $\omega_{\Xi,\Omega}$;
     \item
     add $an+b -\Card(E^{-\infty}(\Theta_1))$ edges to $E^{-\infty}(\widetilde\Gamma_{\Xi,\kappa})$, all adjacent to the minimal vertex of $\widetilde\Gamma_{\Xi,\kappa}$, and extend $\omega_{\Xi,\Omega}$ by $1$ on these additional edges;
      \item 
      add $b -\Card(E^{+\infty}(\Theta_m))$ edges to $E^{+\infty}(\widetilde\Gamma_{\Xi,\kappa})$, all adjacent to the maximal vertex of $\widetilde\Gamma_{\Xi,\kappa}$, and extend $\omega_{\Xi,\Omega}$ by $1$ on these additional edges; denote by $\Gamma_{\Xi,\kappa}$ the resulting graph.
 \end{enumerate}
 The resulting weighted graph
 $\D_{\Xi,\kappa}=(\Gamma_{\Xi,\kappa},\omega_{\Xi,\Omega})$ is  a
 strongly layered floor 
 diagram  with Newton polygon $\Delta_{a,b,n}$ as announced. Note
 also that
 \[
 g(\D_{\Xi,\kappa})=\sum_{j=1}^m g(\Theta_j)
 \qquad\mbox{and}\qquad
 \codeg(\D_{\Xi,\kappa})=\sum_{j=1}^m\codeg(\Theta_j).
 \]
 These two quantities are called the genus and the codegree of
 $\Xi$, respectively.
 The next proposition generalizes Lemma \ref{lem:d rat} to higher genera.

 \begin{lemma}\label{lem:layered}
    Let $a,b,n,i\in\Z_{\ge 0}$ be such that
\[
\left\{\begin{array}{l}
b>i
\\ a>i+g+1
\end{array}
\right. .
\]
Then any floor diagram with Newton polygon $\Delta_{a,b,n}$
and of codegree at most $i$ is strongly layered. In particular,
the construction above 
establishes  a bijection between the set of triples
$(\Xi,\kappa,\Omega)$, with $\Xi$ admissible of genus
$g$ and codegree $i$,  with
$\kappa\in  A_{a}(\Xi)$ and $\Omega\in B_{a,b,n}(\Xi,\kappa)$ on 
one hand, and the
set of floor diagram with Newton polygon $\Delta_{a,b,n}$, of genus
$g$ and codegree $i$ on the other hand.
 \end{lemma}
 \begin{proof}
   The second assertion follows immediately from the first one.
   Assume that there exists
    a non-strongly layered floor diagram $\D$  with Newton polygon $\Delta_{a,b,n}$
    and of codegree at most $i$.

    Suppose first that $\D$ is not layered.
    This means that there exist two floors $v_1$
    and $v_2$ of $\D$ that are not comparable for $\o$. 
    As in the proof of Lemma~\ref{lem:d rat}, the
   floor diagram $\D$ has a unique minimal floor and a unique maximal
   floor. By finitely many
 applications of moves $A^\pm$ and Lemma \ref{lem:codegree}, we reduce to the case
 where
 \begin{itemize}
 \item  $\o$ induces a total order on
   $V(\D)\setminus\{v_1,v_2\}$;
   \item $\D\setminus\{v_1,v_2\}$ is
     disconnected;
  \item elevators in $E^{\pm\infty}(\D)$ are adjacent to an extremal
    floor of $\D$;
     \item  elevators in $E^0(\D)$ not adjacent to $v_1$ nor
    $v_2$ are adjacent to two consecutive floors;   
   \item  elevators in $E^0(\D)$ adjacent to $v_1$ or
    $v_2$
     are  as  depicted in Figure
     \ref{fig:layered} (where weights are not mentioned).
 \end{itemize}
\begin{figure}[h]
\begin{center}
\begin{tabular}{c}
  \includegraphics[height=6cm, angle=0]{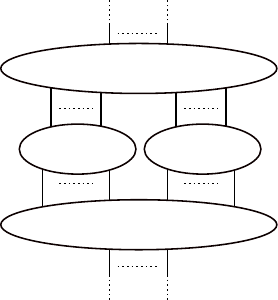}
   \put(-47,85){$v_2$}
   \put(-120,85){$v_1$}
      \put(-85,40){$v_0$}
\end{tabular}
\end{center}
\caption{A non-layered floor diagram}
\label{fig:layered}
\end{figure}
  Defining
     \[
     w=\sum_{\underset{v_0 \longrightarrow}{e}}\omega(e),
     \]
we have that
      \[
     w\ge  b+3n.
     \]
Finitely many
applications of moves $A^\pm$  and Lemma \ref{lem:codegree} also give
\[
\codeg(\D)\ge w-n\ge b+2n >i,
\]
in contradiction to our assumption. Hence $\D$ is layered.

Since $\D$ is not strongly layered, this means by Definition \ref{def:st layered}
that $E^u(\D)\subset E^{-\infty}(\Gamma)\cup E^{+\infty}(\Gamma)$ and
 $\D_1=\D\setminus E^u(\D)$. According to Lemma \ref{lem:temp c+g}, one
has
\[
  \codeg(\D_1)+g\ge a-1.
  \]
  Since $a>i+g+1$, we deduce that $\codeg(\D_1)>i$ in contradiction to
  our assumption. 
 \end{proof}

 \subsection{Polynomiality of
   $(a,b,n)\mapsto\coef[i]{G_{\Delta_{a,b,n}}(g)}$}\label{sec:polg}

 Similarly to floor diagrams, we define a
 \emph{marking of a template $\Theta$} as 
 a bijective
  map
  \[
  m\colon\{1,2,\dots,\Card(V(\Theta)\cup E^0(\Theta))\}\longrightarrow V(\Theta)\cup E^0(\Theta)
  \]
such that
 $j\le k$ whenever $m(j)\preccurlyeq m(k)$.
All markings of a given template $\Theta$ are considered
 up to 
 automorphisms
  of   oriented partially weighted
graph $\varphi:\Theta\longrightarrow \Theta$ such that
and $m=\varphi\circ m'$.

Denoting by
$v_1\prec v_2 \prec \cdots \prec v_{l(\Theta)}$
the vertices of $\Theta$, we define $\gamma_{j}$ to be the number of
edges connecting $v_j$ and $v_{j+1}$, and
\[
\mathcal A(\Theta)=\prod_{j=1}^{l(\Theta)-1} \frac{1}{\gamma_{j}!}.
\]
Next, given an admissible collection of templates
$\Xi=(\Theta_1,\cdots,\Theta_m)$, we set
  \[
  \mathcal A(\Xi)=\prod_{j=1}^m \mathcal A(\Theta_j).
  \]
If  $\kappa\in  A_{a}(\Xi)$ and
  $\Omega\in B_{a,b,n}(\Xi,\kappa)$, any
collection
$M=(M_1,\cdots,M_m) $
of markings of $\Theta_1,\cdots,\Theta_m$
extends uniquely to the graph
$\widetilde\Gamma_{\Xi,\kappa}\setminus
(E^{-\infty}(\widetilde\Gamma_{\Xi,\kappa})\cup
E^{+\infty}(\widetilde\Gamma_{\Xi,\kappa}))$ 
constructed out of
  $\Xi,\kappa$, and $\Omega$. The number of ways to extend this
  marking to a marking of the floor diagram $\D_{\Xi,\kappa}$
  depends on neither  $\kappa$ nor $\Omega$, and is
  denoted by $\nu_{\Xi,M}(a,b,n)$. Analogously to the function
  $\nu_{u,\widetilde u}$ from Section \ref{sec:g=0 gen}, the
  function
  $\nu_{\Xi,M}$ is polynomial and has degree at most
  $\Card(E^{-\infty}(\Theta_1))+\Card(E^{+\infty}(\Theta_m))$ in each of the variables
  $a,b$, and $n$.

\begin{lemma}\label{lem:decomp g}
  Let $a,b,n,i\in\Z_{\ge 0}$ be such that
\[
\left\{\begin{array}{l}
b>i
\\ a>i+g+1
\end{array}
\right. .
\]
Then for any $g\ge 0$ one has
\[
\coef[i]{G_{\Delta_{a,b,n}}(g)}=\sum_{\Xi,M}
\ \mathcal A(\Xi) \times
\nu_{\Xi,M}(a,b,n)
\sum_{\kappa\in  A_{a}(\Xi)}\  \ \sum_{\Omega\in B_{a,b,n}(\Xi,\kappa)}
\ \coef[i-\codeg(\Xi)]{\mu(\D_{\Xi,\Omega}) },
\]
where the first sum ranges over all admissible collections of templates
$\Xi=(\Theta_1,\cdots,\Theta_m)$ of genus $g$ and codegree at most
$i$, and over all collections
of markings $M$ 
of $\Theta_1,\cdots,\Theta_m$. 
\end{lemma}
\begin{proof}
  Given a floor diagram $\D$, we denote by $\nu(\D)$ its number of markings.
  By Theorem \ref{thm:fd}, we have
  \[
  \coef[i]{G_{\Delta_{a,b,n}}(g)}=\sum_{\D} \nu(\D) \coef[i-\codeg(\D)]{\mu(\D)},
  \]
  where the sum is taken over all floor diagrams $\D$ of genus $g$ and
  codegree at most $i$. Now the result follows from  Lemma \ref{lem:layered}.
\end{proof}

Lemma \ref{lem:decomp g} provides a decomposition of
$\coef[i]{G_{\Delta_{a,b,n}}(g)}$ into pieces that are combinatorially
manageable. We prove the polynomiality of 
$\sum_{\Omega\in B_{a,b,n}(\Xi,\kappa)}
\ \coef[i-\codeg(\Xi)]{\mu(\D_{\Xi,\Omega}) }$ in the next lemma, from which we deduce a
proof of Theorem \ref{thm:maing}.

\begin{lemma}\label{lem:pol fin}
  Let  $i,g\in\Z_{\ge 0}$, and $\Xi=(\Theta_1,\cdots,\Theta_m)$
  be an admissible collection of templates  of
  genus $g$ and codegree at most $i$. Given
   $(a,b,n)\in \Z_{\ge 0}$ such that
  \[
  \left\{\begin{array}{l}
  n\ge 1
  \\ b\ge \Card(E^{+\infty}(\Theta_m))
\\b+n>(g+2)i+g
\\ a\ge l(\Theta_1)+\cdots +l(\Theta_m)
\end{array}
\right.,
  \]
   and $\kappa\in A_a(\Xi)$, the
  sum
 \[
\sum_{\Omega\in B_{a,b,n}(\Xi,\kappa)}
\ \coef[i-\codeg(\Xi)]{\mu(\D_{\Xi,\Omega}) }
\]
is polynomial in $a,b,n,k_2,\cdots,k_{m-1}$,
of total degree at most $i-\codeg(\Xi)+g$, and of
\begin{itemize}
\item degree at most $i-\codeg(\Xi)+g$ in the variable $a$;
\item degree at most $g$
  in the variables $b$
     and $n$;
 \item degree at most $g(\Theta_j)$ in the variable $k_j$.
\end{itemize}

\medskip
\noindent
If
$\Xi=(\widetilde\Theta_1,\widetilde\Theta_2,\widetilde\Theta_2,\cdots,\widetilde\Theta_2,\widetilde\Theta_1)$,
with $\widetilde\Theta_1$ and
$\widetilde\Theta_2$ depicted in Figure \ref{fig:Gamma_i},
then the 
  sum
 \[
\sum_{\Omega\in B_{a,b,n}(\Xi,\kappa)}
\ \coef[i]{\mu(\D_{\Xi,\Omega}) }
\]
is polynomial in $a,b,n,k_2,\cdots,k_{g+1}$,
of total degree $i+g$, and of
\begin{itemize}
\item degree  $i+g$ in the variable $a$;
  \item degree  $g$ in the variables $b$ and $n$;
  \item degree  $g(\widetilde\Theta_2)=1$ in the variable $k_j$.
\end{itemize}
\begin{figure}[h]
\begin{center}
  \begin{tabular}{ccccc}
    \includegraphics[height=2cm, angle=0]{Figures1/T1.pdf}
   &\hspace{1.5cm}
   & \includegraphics[height=2cm, angle=0]{Figures1/T2.pdf}
   &\hspace{1.5cm}
   & \includegraphics[height=2.2cm, angle=0]{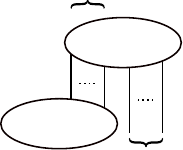}
\put(-18,-10){$i$}
\put(-53,68){$g+1$}
    
\\    \\ a) $\widetilde\Theta_1$ && b) $\widetilde\Theta_2$  &&
c) $\widetilde\Theta_{g,i}$ 
  \end{tabular}
\end{center}
\caption{}
\label{fig:Gamma_i}
\end{figure}
If
$\Xi=(\widetilde\Theta_{g,i},\widetilde\Theta_1)$,
with $\widetilde\Theta_{g,i}$ as depicted in Figure \ref{fig:Gamma_i},
then the
  sum
 \[
\sum_{\Omega\in B_{a,b,n}(\Xi,\kappa)}
\ \coef[0]{\mu(\D_{\Xi,\Omega}) }
\]
is polynomial in $a,b$, and $n$
of total degree $g$, and of degree  $g$ in each of the variables
$a,b$, and $n$.
\end{lemma}
\begin{proof}
  Let $v_{j,1}\prec \cdots \prec v_{j,l(\Theta_j)}$ be the vertices of
  $\Theta_j$, and let $e_{j,k,1},\cdots, e_{j,k,g_{j,k}+1}$ be
the edges of $\Theta_j$
  connecting $v_{j,k}$ and $v_{j,k+1}$. In particular we have
  \[
  \sum_{k=1}^{l(\Theta_j)-1} g_{j,k}\le g(\Theta_j).
  \]
  Given $\Omega\in  B_{a,b,n}(\Xi,\kappa)$, we also have
  \[
  \sum_{u=1}^{g_{j,k}+1} \omega_{\Xi,\Omega}(e_{j,k,u})=
  (a-k_j-k+1)n+b-c_{j,k},
  \]
with $c_{j,k}\in\{0,1,\cdots,i\}$ that only depends on $\Theta_j$.
Hence
$B_{a,b,n}(\Xi,\kappa)$ is in  bijection with  subsets
 of $\prod_{j,k}\Z_{>0}^{g_{j,k}}$ which correspond
to 
decompositions  of
each integer
\[
\beta_{j,k}=(a-k_j-k+1)n+b-c_{j,k}
\]
in an ordered sum of
$g_{j,k}+1$ positive integers. In particular we have
\[
\Card(B_{a,b,n}(\Xi,\kappa))=
\prod_{j,k} {{\beta_{j,k}-1}\choose{g_{j,k}}}.
\]
Note that since $b+n> (g+2)i+g\ge i+g$ by assumption, and
$\beta_{j,k}\ge b+n-i$, one has
\[
\forall j,k,\qquad \beta_{j,k}-1\ge g\ge g_{j,k}.
\]
In particular $\Card(B_{a,b,n}(\Xi,\kappa))$ is polynomial in
$a,b,n,k_2,\cdots,k_{m-1}$ of total degree at most $g$, and of degree
at most $g(\Theta_j)$ in the variable $k_j$. If 
$\coef[i-\codeg(\Xi)]{\mu(\D_{\Xi,\Omega})}$ were not depending on
$\Omega$, then the lemma would be proved. This is unfortunately not
the case, nevertheless there exists a partition of
$B_{a,b,n}(\Xi,\kappa)$ for which the independency holds
on each subset of this partition.

To show this, let $F= \prod_{j,k}\{0,\cdots,i\}^{g_{j,k}}$ and
\[
\begin{array}{cccc}
  \Upsilon:&B_{a,b,n}(\Xi,\kappa)  & \longrightarrow &F
  \\ & (\omega_1,\cdots,\omega_m) &\longmapsto
  &\left\{\begin{array}{ll}
  f_{j,k,u}=0 &\mbox{if }\omega_{j}(e_{j,k,u})>i-\codeg(\Xi)
  \\ f_{j,k,u}= \omega_{j}(e_{j,k,u})&\mbox{if }\omega_{j}(e_{j,k,u})\le i-\codeg(\Xi)
  \end{array}\right.    
  \end{array}.
\]
Given $f\in F$, we denote by $\lambda_{j,k}(f)$ the number of non-zero
coordinates $f_{j,k,u}$, and we define
\[
\lambda(f)=\sum_{j,k} \lambda_{j,k}(f).
\]
Since $b+n>(g+2)i+g\ge (g+2)i$, we  have that
\[
\beta_{j,k}\ge b+n-i> i(g+1) \ge i(g_{j,k}+1),
\]
which in its turn implies that
$\lambda_{j,k}(f)\le g_{j,k}$ and
$\lambda(f)\le g$ if $\Upsilon^{-1}(f)\ne\emptyset$.
As above, we have
\[
\Card(\Upsilon^{-1}(f))=
\prod_{j,k} {{\beta_{j,k}-\sum_u f_{j,k,u}-1}\choose{g_{j,k}-\lambda_{j,k}}}.
\]
Hence if $\Upsilon^{-1}(f)\ne\emptyset$, then for any $j$ and $k$ one has
\[
\beta_{j,k}-\sum_u f_{j,k,u}-1\ge
\beta_{j,k}-i\lambda(f)-1\ge \beta_{j,k} -ig -1\ge b+n- (g+1)i-1\ge g+i\ge
g_{j,k}-\lambda_{j,k}.
\]
In particular
$\Card(\Upsilon^{-1}(f))$ is polynomial in
$a,b,n,k_2,\cdots,k_{m-1}$ of total degree at most $g-\lambda(f)$,
and of degree at most $g(\Theta_j)-\lambda_j(f)$ in the variable $k_j$.

Furthermore, for any $\Omega\in \Upsilon^{-1}(f)$, we have
\[
\mu(\D_{\Xi,\Omega})= P_{\Xi,f}(q)\times \prod_{\omega_j(e_{j,k,u})>i-\codeg(\Xi)} [\omega_j(e_{j,k,u})]^2,
\]
where $P_{\Xi,f}(q)$ is a Laurent polynomial that only depends on
$\Xi$ and $f$. In particular
it follows from Corollary \ref{cor:key2} that
$\coef[i-\codeg(\Xi)]{\mu(\D_{\Xi,\Omega})}$ is a polynomial $Q_{\Xi,f}(a)$ in $a$ of
degree $i-\codeg(\Xi)$, which only depends on
$\Xi$ and $f$.
We deduce that
 \[
\sum_{\Omega\in \Upsilon^{-1}(f)}
\ \coef[i-\codeg(\Xi)]{\mu(\D_{\Xi,\Omega}) }
= \Card(\Upsilon^{-1}(f))\times Q_{\Xi,f}(a)
\]
is polynomial in $a,b,n,k_2,\cdots,k_{m-1}$,
of total degree at most $i-\codeg(\Xi)+g-\lambda(f)$, and of
\begin{itemize}
\item degree at most $i-\codeg(\Xi)+g-\lambda(f)$ in the variable $a$;
\item degree at most $g-\lambda(f)$
  in the variables $b$
     and $n$.
 \item degree at most $g(\Theta_j)-\sum_k\lambda_{j,k}(f)$ in the variable $k_j$.
\end{itemize}
The first part of the lemma now follows from the equality
\[
\sum_{\Omega\in B_{a,b,n}(\Xi,\kappa)}
\ \coef[i-\codeg(\Xi)]{\mu(\D_{\Xi,\Omega}) }=\sum_{f\in F} \ \sum_{\Omega\in \Upsilon^{-1}(f)}
\ \coef[i-\codeg(\Xi)]{\mu(\D_{\Xi,\Omega}) }.
\]
The second part of the lemma follows from a direct application of
the computations above in both
specific situations. 
\end{proof}

\begin{proof}[Proof of Theorem \ref{thm:maing}]
  Recall that
  $\mathcal U_{i,g}\subset \Z_{\ge 0}^3$
is the set of triples $(a,b,n)$ satisfying
\[
\left\{\begin{array}{l}
n\ge 1
\\ b> i
\\ b+n> (g+2)i+g
\\ a\ge i+2g+2
\end{array}
\right. .
\]
  Let $\Xi=(\Theta_1,\cdots,\Theta_m)$
be  an admissible collection of templates  of
genus $g$ and codegree at most $i$.
By Lemma \ref{lem:temp c+g}, we have
\[
l(\Theta_1)+\cdots+l(\Theta_m)\le i+g+m\le i+2g+2\le a
\qquad\mbox{and}\qquad
b+2n> b+n >(g+2)i+g\ge i.
\]
Hence the set of such  collections of templates is finite,
and the assumptions of Lemma \ref{lem:decomp g} are satisfied. 
Since 
\[
\codeg(\Xi)\ge \Card(E^{-\infty}(\Theta_1))+\Card(E^{+\infty}(\Theta_m)),
\]
 to prove the polynomiality of the function
$ (a,b,n) \mapsto  \coef[i]{G_{\Delta_{a,b,n}}(g)}$ and
to get an upper bound on its degree,
it is enough to prove that on $\mathcal U_{i,g}$, the function
\[
\sum_{\kappa\in  A_{a}(\Xi)}\  \ \sum_{\Omega\in B_{a,b,n}(\Xi,\kappa)}
\ \coef[i-\codeg(\Xi)]{\mu(\D_{\Xi,\Omega}) }
\]
is polynomial, of degree at most $g$ in the variables $b$ and,
and
of degree at most $i+2g-\codeg(\Xi)$ in the variable $a$.

\medskip
Let us describe precisely the set $A_{a}(\Xi)$ when $m\ge 3$, which is by
definition 
the subset of $\Z_{>0}^{m-2}$ defined by the system of inequalities
\[
 \left\{
\begin{array}{rl}
  k_2&\ge 1+l(\Theta_1)
  \\ k_3&\ge k_2+l(\Theta_2)
  \\ \vdots
  \\ k_{m-1}&\ge k_{m-2}+l(\Theta_{m-2})
  \\ a+1-l(\Theta_{m})&\ge k_{m-1}+l(\Theta_{m-1})
  \end{array}
\right. .
\]
Hence, in order to get a parametric description of $A_{a}(\Xi)$, we
need to estimate $l(\Theta_1)+\cdots + l(\Theta_m)$.
By Lemma \ref{lem:temp c+g}, we have
\[
\sum_{j=1}^m l(\Theta_j) \le g+i +m.
\]
Furthermore since $g(\Theta_j)\ge 1$ if
$j\in\{2,\cdots,m-1\}$, we have $m\le g+2$, and we deduce that 
\[
\sum_{j=1}^m l(\Theta_j) \le i+ 2g+2.
\]
In particular, since $a\ge i+ 2g+2 $ the set $A_{a}(\Xi)$ can be described
as the set of $(k_2,\cdots, k_{m-1})\subset \Z_{>0}^{m-2}$ such that
\[
 \left\{
\begin{array}{lcccl}
  1+l(\Theta_1)+\cdots+l(\Theta_{m-3})+l(\Theta_{m-2})&\le
&  k_{m-1}&\le& a+1 - l(\Theta_{m})- l(\Theta_{m-1})
  \\ 1+ l(\Theta_1)+\cdots+l(\Theta_{m-3})&\le &k_{m-2}&\le
&   k_{m-1} - l(\Theta_{m-2})
  \\ &&\vdots&
  \\1+l(\Theta_{1})&\le&  k_{2}&\le&
  k_3-l(\Theta_{2})
  \end{array}
\right. ,
\]
 in other words the sum over $A_a(\Xi)$ can be rewritten as
\[
\sum_{\kappa\in  A_{a}(\Xi)}=
\sum_{k_{m-1}=1+l(\Theta_1)+\cdots+l(\Theta_{m-2})}^{a+1 - l(\Theta_{m})- l(\Theta_{m-1})}
\ \ \sum_{k_{m-2}=1+ l(\Theta_1)+\cdots+l(\Theta_{m-3})}^{k_{m-1} - l(\Theta_{m-2})}
\ \ \cdots
\ \ \sum_{k_{2}=1+l(\Theta_{1})}^{k_3- l(\Theta_{2})}.
\]
Combining Faulhaber's formula with Lemma \ref{lem:pol fin}, we obtain
that the sum
 \[
\sum_{k_{2}=1+l(\Theta_{1})}^{k_3- l(\Theta_{2})}\  \sum_{\Omega\in B_{a,b,n}(\Xi,\kappa)}
\ \coef[i-\codeg(\Xi)]{\mu(\D_{\Xi,\Omega}) }
\]
is polynomial in $a,b,n,k_3,\cdots,k_{m-1}$,
of total degree at most $i-\codeg(\Xi)+g+1$, and of
\begin{itemize}
\item degree at most $i-\codeg(\Xi)+g$ in the variable $a$;
\item degree at most
  $g$ in the variables $b$
    and $n$;
  \item degree at most $g(\Theta_2)+g(\Theta_3)+1$ in the variable $k_3$;
  \item degree at most $g(\Theta_j)$ in the variable $k_j$ with $j\ge 4$.
\end{itemize}
As in the end of the proof of \cite[Theorem 5.1]{FM}, we eventually
obtain by induction that
\[
\sum_{\kappa\in  A_{a}(\Xi)}\  \ \sum_{\Omega\in B_{a,b,n}(\Xi,\kappa)}
\ \coef[i-\codeg(\Xi)]{\mu(\D_{\Xi,\Omega}) }
\]
is polynomial of degree at most $g$ in the variables $b$ and $n$, and
of degree at most $i-\codeg(\Xi)+g+m-2$ in the variable $a$.
Since $m-2\le g$, we obtain that the function
$ (a,b,n) \in\mathcal U_{i,g} \mapsto  \coef[i]{G_{\Delta_{a,b,n}}(g)}$
is polynomial, of degree at most
$i+g$ in the variables $b$ and $n$, and of degree at most
$i+2g$ in the variable $a$. The fact that it is indeed of degree
$i+g$ in the variables $b$ and $n$, and of degree
$i+2g$ in the variable $a$ follows from the second part of Lemma
\ref{lem:pol fin}.
\end{proof}

The proof of Theorem \ref{thm:maing n=0} is identical to the proof of
Theorem \ref{thm:maing}. The only place where the assumption  $n>0$
comes into play is Lemma \ref{lem:pol fin}, in the estimation of the degrees of
\[
\sum_{\Omega\in
  B_{a,b,n}(\Xi,\kappa)}\ \coef[i-\codeg(\Xi)]{\mu(\D_{\Xi,\Omega})}
\]
with respect to its
  different variables, and one sees easily how to adapt Lemma
  \ref{lem:pol fin} when $n=0$.
\begin{proof}[Proof of Theorem \ref{thm:maing n=0}]
If $n=0$, then Lemma  \ref{lem:pol fin} still holds with the following
edition: 
the
  sum
 \[
\sum_{\Omega\in B_{a,b,n}(\Xi,\kappa)}
\ \coef[i-\codeg(\Xi)]{\mu(\D_{\Xi,\Omega}) }
\]
is polynomial in $a$ and $b$,
of total degree at most $i-\codeg(\Xi)+g$, and of
\begin{itemize}
\item degree at most $i-\codeg(\Xi)$ in the variable $a$;
\item degree at most $g$
  in the variables $b$.
\end{itemize}
Indeed in this case we have 
\[
\beta_{j,k}=b-c_{j,k},
\]
which implies exactly as in the proof of  Lemma  \ref{lem:pol fin} that
$\Card(\Upsilon^{-1}(f))$ is polynomial in
$b$ of total degree at most $g-\lambda(f)$. Now the remaining of the
proof of  Lemma  \ref{lem:pol fin} proves the claim above.
The proof of Theorem \ref{thm:maing n=0} follows eventually from
this adapted  Lemma  \ref{lem:pol fin} exactly as Theorem
\ref{thm:maing} follows from
Lemma  \ref{lem:pol fin}.
\end{proof}

\subsection{$b=0$ and $n$ fixed}\label{sec:b=0}
As in the genus 0 case, one easily adapts
 the proof of Theorem \ref{thm:maing} in
 the case when one wants to fix $b=0$ and $n\ge 1$.
There is no additional technical difficulty here with respect to Sections
\ref{sec:g=b=0} and \ref{sec:polg}, so we briefly indicate the main steps.
Again,
the difference with the case $b\ne 0$ is that now a floor diagram $\D$
contributing to 
$ \coef[i]{G_{\Delta_{a,0,n}}(0)}$ may not be layered because of some
highest vertices.
\begin{defi}\label{def:capping templ}
  A \emph{capping template} 
with Newton polygon $\Delta_{a,n}$ is
  a couple
  $\mathcal C=(\Gamma, \omega)$ such that
  \begin{enumerate}
  \item $\Gamma$ is a connected weighted oriented acyclic graph
    with $a$ vertices and with no
    sources nor sinks;
  \item $\Gamma$ has a unique minimal vertex $v_1$, and
    $\Gamma\setminus \{v_1\}$ has at least two minimal vertices;    
\item   for every vertex $v\in V(\Gamma)\setminus \{v_1\}$,
one has $ \dive(v)=n$.
  \end{enumerate}
  The codegree of a capping template $\mathcal C$ with Newton polygon
  $\Delta_{a,n}$ is defined as
  \[
  \codeg(\mathcal C)=\frac{(a-1)(na-2)}{2}-g(\mathcal C)-\sum_{e\in E(\Gamma)}(\omega(e)-1)
  \]
\end{defi}

The proof of the next lemma is analogous to the proof of Lemma \ref{lem:codeg capping}.
\begin{lemma}\label{lem:codeg capping templ}
  A capping template with  Newton polygon
  $\Delta_{a,n}$ has codegree at least $n(a-2)$.  
\end{lemma}

\begin{proof}[Proof of Theorem \ref{thm:main1 expl g02}]
 Let $\D$ be a floor diagram of genus g, Newton polygon
 $\Delta_{a,0,n}$, and of codegree at most $i$.
 As in the proof of Theorem \ref{thm:main1 expl g02}, 
 we have that $\D$ has a unique minimal floor.
Suppose that $\D$ is not layered, and let
 $v_o$ be the lowest floor
of $\D$ such that $\D\setminus \{v_o\}$ is not connected and with
a non-layered  upper part.
 Let $\mathcal C$ be the weighted
subgraph of $\D$ 
obtained by removing from $\D$ all elevators and floors strictly below
$v_o$. As in the proof of Theorem \ref{thm:main1 expl g02}, one shows
that $\mathcal C$ is a capping template.
For a fixed $i$ and $g$, there exist
finitely many capping templates of codegree at
most $i$ and genus at most $g$. The end of the proof is now entirely
analogous to  the end of the proof of
Theorem \ref{thm:main1 expl g02}.
\end{proof}

  \appendix
 \section{Some identities involving quantum numbers}\label{app:quantum int}
For the reader's convenience,
we collect some easy or well-known properties of quantum integers.
Recall that 
given an integer $n\in \Z$, the quantum integer $[k](q)$ is defined by
\[
\displaystyle
    [k](q)=\frac{q^{\frac{k}{2}}-q^{-\frac{k}{2}}}{q^{\frac{1}{2}}-q^{-\frac{1}{2}}}=
    q^{\frac{k-1}{2}}  +q^{\frac{k-3}{2}}+\cdots
    +q^{-\frac{k-3}{2}}+q^{-\frac{k-1}{2}}
    \in\Z_{\ge 0}[q^{\pm \frac{1}{2}}].
    \]

    Given two elements $f,g  \in\Z_{\ge 0}[q^{\pm \frac{1}{2}}]$, we
      write $f\ge g$ if $f-g  \in\Z_{\ge 0}[q^{\pm \frac{1}{2}}]$.

\begin{lemma}\label{lem:quatum product}
  For any $k,l\in \Z_{\ge 0}$, one has
  \[
  [k]\cdot [k+l]=[2k+l-1] + [2k+l-3] + \cdots + [l+3]+[l+1].
  \]
\end{lemma}
\begin{proof}
  This is an easy consequence from the fact that, given
  $c\in\{1,\cdots k-1\}$, one has
  \[
  (q^{-\frac{k-c}{2}}+q^{\frac{k-c}{2}})\cdot [k+l](q) = [2k+l-c](q) +[l+c](q).
  \]
\end{proof}

\begin{cor}\label{cor:quantum ineq1}
  For any positive integers $k$ and $l$, one has
  \[
  [k]\cdot [k+l-1] = [k-1]\cdot [k+l] +[k].
  \]
  In particular, one has $[k]\cdot [k+l-1] \ge [k-1]\cdot [k+l]$
\end{cor}
\begin{proof}
  It follows from Lemma \ref{lem:quatum product} that
  \begin{align*}
    [k]\cdot [k+l-1]&= [2k+l-2] + [2k+l-4] + \cdots + [l+2]+[l]
    \\ &=  [k-1]\cdot [(k-1) +l +1] +[l],
  \end{align*}
  and the statement is proved.
\end{proof}

\begin{lemma}\label{lem:quantum div}
  For any positive integer $k$, one has
  \[
  \frac{[2k]}{[2]}(q)=[k](q^2).
  \]
  In particular
  $\frac{[2k]}{[2]}\in\Z_{\ge
    0}[q^{\pm1}]$, and one has
  \[
    [2k-1]\ge \frac{[2k]}{[2]}.
  \]
\end{lemma}
\begin{proof}
  One has
  \begin{align*}
    \frac{[2k]}{[2]}(q)
    =  \frac{q^{-2k}-q^{2k}}{q^{-1}-q}
=  \frac{(q^2)^{-k}-(q^2)^{k}}{(q^2)^{-\frac{1}{2}}-(q^2)^{\frac{1}{2}}}
  \end{align*}
  as announced.
\end{proof}

\begin{cor}\label{cor:quantum ineq2}
  For any positive integers $k$ and $l$, one has
  \[
 [k]^2\cdot[l]^2\ge   \frac{[k]\cdot [l]\cdot [k+l]}{[2]}.
  \]
\end{cor}
\begin{proof}
  Suppose first that $k+l$ is even. By Lemmas \ref{lem:quantum div} and
  \ref{lem:quatum product}, one has
    \begin{align*}
    \frac{[k+l]}{[2]}\le [k+l-1]
    \le [k]\cdot [l],
    \end{align*}
    and the lemma is proved in this case.

    If $k+l$ is odd, we may assume that $k$ is even. Then
    by Lemmas \ref{lem:quantum div} and
  \ref{lem:quatum product}, and Corollary \ref{cor:quantum ineq1}, one has
        \begin{align*}
    \frac{[k]\cdot [k+l]}{[2]}\le
         [k-1]\cdot [k+l]
         \le [k]\cdot [k+l-1]
         \le [k]^2\cdot [l],
    \end{align*}
    and the lemma is proved in this case as well.
\end{proof}

\bibliographystyle{alpha}
\bibliography{Biblio.bib}

\begin{thebibliography}{ABLdM11}

\bibitem[AB13]{ArdBlo}
F.~Ardila and F.~Block.
\newblock Universal polynomials for {S}everi degrees of toric surfaces.
\newblock {\em Adv. Math.}, 237:165--193, 2013.

\bibitem[ABLdM11]{Br8}
A.~Arroyo, E.~Brugall\'e, and L.~L\'opez~de Medrano.
\newblock Recursive formula for {W}elschinger invariants.
\newblock {\em Int Math Res Notices}, 5:1107--1134, 2011.

\bibitem[BG16a]{BlGo14bis}
F.~Block and L.~G{\"o}ttsche.
\newblock Fock spaces and refined {S}everi degrees.
\newblock {\em Int. Math. Res. Not. IMRN}, (21):6553--6580, 2016.

\bibitem[BG16b]{BlGo14}
F.~Block and L.~G{\"o}ttsche.
\newblock Refined curve counting with tropical geometry.
\newblock {\em Compos. Math.}, 152(1):115--151, 2016.

\bibitem[BGM12]{BGM}
F.~Block, A.~Gathmann, and H.~Markwig.
\newblock Psi-floor diagrams and a {C}aporaso-{H}arris type recursion.
\newblock {\em Israel J. Math.}, 191(1):405--449, 2012.

\bibitem[Blo19]{Blo19}
T.~Blomme.
\newblock A {C}aporaso-{H}arris type formula for relative refined invariants.
\newblock arXiv:1912.06453, 2019.

\bibitem[Blo20a]{Blo20b}
T.~Blomme.
\newblock Computation of refined toric invariants {II}.
\newblock arXiv:2007.02275, 2020.

\bibitem[Blo20b]{Blo20a}
T.~Blomme.
\newblock A tropical computation of refined toric invariants.
\newblock arXiv:2001.09305, 2020.

\bibitem[BM07]{Br7}
E.~Brugall\'e and G.~Mikhalkin.
\newblock Enumeration of curves via floor diagrams.
\newblock {\em Comptes Rendus de l'Académie des Sciences de Paris}, série I,
  345(6):329--334, 2007.

\bibitem[BM08]{Br6b}
E.~Brugall\'e and G.~Mikhalkin.
\newblock Floor decompositions of tropical curves : the planar case.
\newblock {\em Proceedings of 15th {G}\"okova {G}eometry-{T}opology
  Conference}, pages 64--90, 2008.

\bibitem[Bou19a]{Bou19}
P.~Bousseau.
\newblock Refined floor diagrams from higher genera and lambda classes.
\newblock arXiv:1904.10311, 2019.

\bibitem[Bou19b]{Bou17}
P.~Bousseau.
\newblock Tropical refined curve counting from higher genera and lambda
  classes.
\newblock {\em Invent. Math.}, 215(1):1--79, 2019.

\bibitem[Bru08]{Br10}
E~Brugall\'e.
\newblock Géométries énumératives complexe, réelle et tropicale.
\newblock In N.~Berline, A.~Plagne, and C.~Sabbah, editors, {\em Géométrie
  tropicale}, pages 27--84. Éditions de l'École Polytechnique, Palaiseau,
  2008.

\bibitem[Bru20]{Bru18}
E.~Brugall\'e.
\newblock On the invariance of {W}elschinger invariants.
\newblock {\em Algebra i Analiz}, 32(2):1--20, 2020.

\bibitem[BS19]{BleShu17}
L.~Blechman and E.~Shustin.
\newblock Refined descendant invariants of toric surfaces.
\newblock {\em Discrete Comput. Geom.}, 62(1):180--208, 2019.

\bibitem[DFI95]{DiFrIt}
P.~Di~Francesco and C.~Itzykson.
\newblock Quantum intersection rings.
\newblock In {\em The moduli space of curves (Texel Island, 1994)}, volume 129
  of {\em Progr. Math.}, pages 81--148. Birkh\"auser Boston, Boston, MA, 1995.

\bibitem[FM10]{FM}
S.~Fomin and G.~Mikhalkin.
\newblock Labelled floor diagrams for plane curves.
\newblock {\em Journal of the European Mathematical Society}, 12:1453--1496,
  2010.

\bibitem[FS15]{FiSt15}
S.~A. Filippini and J.~Stoppa.
\newblock Block-{G}\"{o}ttsche invariants from wall-crossing.
\newblock {\em Compos. Math.}, 151(8):1543--1567, 2015.

\bibitem[G\"98]{Got98}
L.~G\"{o}ttsche.
\newblock A conjectural generating function for numbers of curves on surfaces.
\newblock {\em Comm. Math. Phys.}, 196(3):523--533, 1998.

\bibitem[GK16]{GoKi16}
L.~G\"{o}ttsche and B.~Kikwai.
\newblock Refined node polynomials via long edge graphs.
\newblock {\em Commun. Number Theory Phys.}, 10(2):193--224, 2016.

\bibitem[GS14]{GotShe12}
L.~G{\"o}ttsche and V.~Shende.
\newblock Refined curve counting on complex surfaces.
\newblock {\em Geom. Topol.}, 18(4):2245--2307, 2014.

\bibitem[GS19]{GotSch16}
L.~G{\"o}ttsche and F.~Schroeter.
\newblock Refined broccoli invariants.
\newblock {\em J. Algebraic Geom.}, 28(1):1--41, 2019.

\bibitem[IKS04]{IKS2}
I.~Itenberg, V.~Kharlamov, and E.~Shustin.
\newblock Logarithmic equivalence of {W}elschinger and {G}romov-{W}itten
  invariants.
\newblock {\em Uspehi Mat. Nauk}, 59(6):85--110, 2004.
\newblock (in Russian). English version: Russian Math. Surveys 59 (2004), no.
  6, 1093-1116.

\bibitem[IM13]{IteMik13}
I.~Itenberg and G.~Mikhalkin.
\newblock On {B}lock-{G}\"ottsche multiplicities for planar tropical curves.
\newblock {\em Int. Math. Res. Not. IMRN}, (23):5289--5320, 2013.

\bibitem[KP04]{KlePie04}
S.~Kleiman and R.~Piene.
\newblock Node polynomials for families: methods and applications.
\newblock {\em Math. Nachr.}, 271:69--90, 2004.

\bibitem[KST11]{KST11}
M.~Kool, V.~Shende, and R.~P. Thomas.
\newblock A short proof of the {G}\"{o}ttsche conjecture.
\newblock {\em Geom. Topol.}, 15(1):397--406, 2011.

\bibitem[Mik05]{Mik1}
G.~Mikhalkin.
\newblock {Enumerative tropical algebraic geometry in $\mathbb R^2$}.
\newblock {\em J. Amer. Math. Soc.}, 18(2):313--377, 2005.

\bibitem[Mik17]{Mik15}
G.~Mikhalkin.
\newblock Quantum indices of real plane curves and refined enumerative
  geometry.
\newblock {\em Acta Math.}, 219(1):135--180, 2017.

\bibitem[NPS18]{NPS16}
J.~Nicaise, S.~Payne, and F.~Schroeter.
\newblock Tropical refined curve counting via motivic integration.
\newblock {\em Geom. Topol.}, 22(6):3175--3234, 2018.

\bibitem[Shu18]{Shu18}
E.~Shustin.
\newblock On refined count of rational tropical curves.
\newblock arXiv:1812.08038, 2018.

\bibitem[Tze12]{Tze12}
Y.-J. Tzeng.
\newblock A proof of the {G}\"{o}ttsche-{Y}au-{Z}aslow formula.
\newblock {\em J. Differential Geom.}, 90(3):439--472, 2012.

\bibitem[Wel05]{Wel1}
J.~Y. Welschinger.
\newblock Invariants of real symplectic 4-manifolds and lower bounds in real
  enumerative geometry.
\newblock {\em Invent. Math.}, 162(1):195--234, 2005.

\bibitem[Wel07]{Wel4}
J.~Y. Welschinger.
\newblock Optimalit\'e, congruences et calculs d'invariants des vari\'et\'es
  symplectiques r\'eelles de dimension quatre.
\newblock arXiv:0707.4317, 2007.

\end{thebibliography}

\end{document}